\documentclass[12pt, reqno]{amsart}
\usepackage[nocompress]{cite}
\usepackage{graphicx}
\usepackage{amssymb, comment, color}
\usepackage[toc,page]{appendix}
\usepackage{epsf,overpic,amssymb,amscd,times,euscript}
\usepackage [autostyle, english = american]{csquotes}
\MakeOuterQuote{"}
\usepackage[utf8]{inputenc}
\usepackage[english]{babel}
\usepackage{amsmath}
\usepackage{amsthm}
\usepackage{amsfonts}
\usepackage{amssymb}
\usepackage{amscd}
\usepackage{bezier}
\usepackage{enumerate}
\usepackage{tikz}
\usepackage{graphicx}
\usepackage[urlcolor=blue,citecolor=blue]{hyperref}
\usepackage{upgreek}

\makeatletter
\newcommand{\dtilde}[1]{{%
  \mathpalette\double@widetilde{#1}%
}}
\newcommand{\double@widetilde}[2]{%
  \sbox\z@{$\m@th#1\widetilde{#2}$}%
  \ht\z@=.9\ht\z@
  \widetilde{\box\z@}%
}
\makeatother

\hypersetup{colorlinks}
\usepackage{parskip}
\usepackage{bbm}

\newcommand{\N}{\mathbbm{N}}                     
\newcommand{\Z}{\mathbbm{Z}}                     
\newcommand{\R}{\mathbbm{R}}                     

\newcommand{\D}{\mathbb{D}}

\newcommand{\NUH}{\mathrm{NUH}}

\newcommand{\dist}{\mathrm{dist\,}}             
\newcommand{\diam}{\mathrm{diam\,}}             

\newcommand{\im}{\mathrm{im}}

\newcommand{\CH}{\mathrm{CH}} 
\newcommand{\dom}{\mathrm{dom}}

\newcommand{\topo}{\mathrm{top}}

\newcommand{\Per}{\mathrm{Per}}
\newcommand{\per}{\mathrm{per}}
\newcommand{\nuh}{\mathrm{NUH}}
\newcommand{\bi}{\mathbf{i}}

\newtheorem*{qu*}{Question}
\newtheorem{mainthm}{\sc Theorem}           
\newtheorem{thm}{Theorem}[section]               
\newtheorem*{thm*}{Theorem}               
\newtheorem{cor}[thm]{Corollary}        
\newtheorem*{cor*}{Corollary}        
\newtheorem{lem}[thm]{Lemma}  
\newtheorem*{lem*}{Lemma}
\newtheorem{prop}[thm]{Proposition}     
\theoremstyle{definition}
\newtheorem{defn}[thm]{Definition}      
\newtheorem{rem}[thm]{Remark}           
 \newtheorem*{acknowledgement*}
 {\protect\acknowledgementname}
\newcounter{claim}
\newenvironment{claim}[1][]{\refstepcounter{claim}\par\noindent\underline{Claim~\theclaim:}\space#1}{}
\newenvironment{claimproof}[1]{\par\noindent\underline{Proof:}\space#1}{\qed}

 \providecommand{\acknowledgementname}{Acknowledgement}

\author{Matthias Meiwes}
\thanks{The author  was supported by the ISF grant~938/22 and the ERC Starting Grant 757585}

\address{Matthias Meiwes,
School of Mathematical Sciences, Tel Aviv University, Ramat Aviv, Tel Aviv 69978, Israel.}
\email{\texttt{matthias.meiwes@live.de}}

\title[
orbit growth in link complements]{Topological entropy and orbit growth in link complements}
\begin{document}
\maketitle
\begin{abstract}
In this article, we exhibit certain linking properties of periodic orbits of $C^{1+\alpha}$ flows with positive topological entropy  on closed $3$-manifolds $M$.  
It is shown that any such flow $\varphi$ contains a link $\mathcal{L}$ of periodic orbits and a horseshoe $K$ in $M\setminus \mathcal{L}$, such that all periodic orbits in $K$ are unique in their  homotopy class in $M\setminus \mathcal{L}$ (among 
 periodic orbits in $M$). 
Moreover, the entropy of $\varphi$ can be approximated by the entropies of such horseshoes $K$. 
 A version 
 of that result for chords is obtained. Our main motivation comes from 
Reeb dynamics, and as an application, we address a question by Alves-Pirnapasov,  and obtain that the  topological entropy of a $3$-dimensional, $C^{\infty}$-generic Reeb flow can be approximated by the exponential homotopical growth rates of contact homology in link complements.  
\end{abstract}

\section{Introduction and 
 main results}
In the study of many topological and smooth dynamical systems of interest,  
 symbolic dynamics plays an important role for the description and  understanding of an essential part of their orbit structures. Relating the system to shift maps of symbolic sequences   reveals some remarkable structure of rather chaotic orbit sets, and allows one to study them from a qualitative as well as quantitative perspective. 
It is moreover the interplay between the properties of the shift maps on the one hand and the  geometric, topological, ergodic theoretic properties etc.\ of the original system on the other hand, that makes symbolic dynamics so attractive. 

The use of symbolic dynamics goes back to the work of Hadamard, Morse, Hedlund, and others.  A celebrated result on the prevalence of symbolic dynamics  is  
 Katok's theorem,  
\cite{KatokLyapunov_exp},  which asserts  that any $C^{1+\alpha}$ surface diffeomorphism with positive topological entropy  contains a horseshoe, 
a hyperbolic invariant subset on which, up to passing to an iterate, the diffeomorphism is conjugated to a full shift. Moreover, the topological entropy of the diffeomorphism can be approximated by the topological entropy of a sequence of horseshoes contained therein, \cite[Suppl.]{Hasselblatt-Katok}. 
  For recent works on generalizations of those results,  I refer to the works of Sarig, \cite{Sarig_diffeo2013}, and Gelfert, \cite{Gelfert2016}, and references therein. 
  Similar results  also hold for flows on $3$-manifolds (see \cite{LianYoung2012}, \cite{LimaSarig2019}). 
  Given any $C^{1+\alpha}$ flow $\varphi$  on a closed smooth $3$-manifold $M$ and a $\varphi$-invariant ergodic measure of positive entropy, Lima and Sarig in  \cite{LimaSarig2019} construct a topological Markov flow that "codes" the dynamics on a set of full measure, see Theorem \ref{thm:fullmeasurecoding} below for details.

  The aim of this article is to revisit those generalizations of Katok's result to $3$-dimensional flows from the aspect of linking of orbits. We will then show that, under similar assumptions on the flow as above, there  exists a link $\mathcal{L}$ of periodic orbits with the following property: the number of homotopy classes of loops in the complement of $\mathcal{L}$ that carry exactly one orbit grows exponentially with the period. Moreover, the growth rate can be as close as desired to the topological entropy of the flow (Theorem \ref{thm:link_growth}). We also obtain a relative version for chords (Theorem \ref{thm:link_growth_rel}). 
Our main motivation comes from symplectic dynamics. We address a question of Alves and Pirnapasov, \cite{AlvesPirnapasov}, and show that, with some natural assumptions (which are known to hold $C^{\infty}$-generically),  the topological entropy of a Reeb flow in a contact $3$-manifold can be approximated by the growth of the contact homology in the complement of links (Theorem \ref{thm:CH_recover}). 
This yields an approach to recover, in low dimensions, topological entropy by quantities defined via holomorphic curves or Floer theory that is distinct from the approach by \c{C}ineli, Ginzburg, and G\"urel  that relates topological entropy to barcode entropy, \cite{CGG_barcode}, cf.\  Remark~\ref{rem:barcode}. The results of Theorem \ref{thm:CH_recover} will 
 also be important in a forthcoming paper, joint with Marcelo Alves, Lucas Dahinden, and Abror Pirnapasov,  \cite{ADMP}, see Remark~\ref{rem:robustness}. 
Symbolic dynamics plays a central role in the present article. Moreover, the underlying setup will be given by the countable Markov flows constructed by Lima and Sarig. Within those Markov flows, we will obtain horseshoes with some specific properties  (Theorem~\ref{thm:main_coding}).  We will  eventually derive Theorems~\ref{thm:link_growth},~\ref{thm:link_growth_rel}, and~\ref{thm:CH_recover} from Theorem~\ref{thm:main_coding}.

\subsection{Growth in the complement of a link}\label{sec:growth_complement}
In the following let $M$ be a closed smooth manifold of dimension $3$, $\alpha>0$, and $\varphi= \varphi^t:\R \times M \to M$, $(t,x) \mapsto \varphi^t(x)$,  a $C^{1 + \alpha}$ flow generated by a nowhere vanishing vector field. A central notion for this paper is the topological entropy $h_{\topo}(\varphi)$ of the flow, and we recall for the reader's convenience the definition of $h_{\topo}$, following Bowen (see 
\cite{Hasselblatt-Katok}). We fix a metric $d$ on $M$ that induces the topology of $M$. A family $(d_T)_{T\geq 0}$ of metrics on $M$ is defined by  
$d_T(x,y) := \sup_{0\leq t\leq T} d(\varphi^t(x), \varphi^t(y))$. For any given $T\geq 0$, $\epsilon>0$, a subset $X \subset M$ is said to be  \textit{$(T, \epsilon)$-spanning}  
 if $\inf_{x\in X} d_{T}(x,y) \leq \epsilon$ for any $y\in M$. 
 The \textit{topological entropy of $\varphi$} is defined as  
\begin{align}\label{defn:entropy}
h_{\topo}(\varphi) = \lim_{\epsilon\to 0} h_{\topo,\epsilon}(\varphi),
\end{align}
where $h_{\topo,\epsilon}(\varphi) := \limsup_{T \to \infty} \frac{1}{T}\log \left(\, \min \{\# X |\, X  \text{ is }(T,\epsilon)\text{-spanning} \}\right)$.

Denote by $\Per(\varphi)$ the collection of periodic orbits of $\varphi$. This is the collection of loops $\gamma:\R/T\Z \to M$ with $\gamma(t) = \varphi^t(p)$ for some $p\in M$, up to identification of two loops by a reparametrization $t\mapsto t + c$, $c\in \R$. We write  $\per(\gamma):=T$ for the period of the periodic orbit $\gamma$.
Let $\mathcal{L} \subset M$ be a link in $M$, that is,  a finite collection of 
 pairwise disjoint closed simple loops, which we call the \textit{components} of $\mathcal{L}$. We also often identify the components of $\mathcal{L}$ with their image  and view $\mathcal{L}$ as a subset in $M$.   We say that $\mathcal{L}$ is a \textit{link of periodic orbits} of $\varphi$ if each component is a periodic orbit of $\varphi$. 
 Denote by  $\mathcal{H}_{\mathcal{L}}$ the set of free homotopy classes of loops in $M \setminus \mathcal{L}$. For $\rho \in \mathcal{H}_{\mathcal{L}}$, we set $\per(\rho):= \sup\{\per(\gamma)\, |\,  \gamma \in \Per(\varphi) ,  [\gamma] = \rho\}$,  with the convention that $\sup \emptyset =+\infty$. 
We say that $\rho \in \mathcal{H}_{\mathcal{L}}$ is \textit{singular} if there is exactly one orbit in $\Per(\varphi)$ that is a representative of   $\rho$, and let $\mathcal{H}_{\mathcal{L}, \mathrm{sing}}\subset\mathcal{H}_{\mathcal{L}}$ be the subset of singular free homotopy classes. 
Let $N(T) := \#\{ \rho \in \mathcal{H}_{\mathcal{L}, \mathrm{sing}} \, |\, \per(\rho) \leq T\}$. 
 Define 
 $$H^{\infty}(\varphi, M\setminus \mathcal{L}) := \limsup_{T\to+\infty} \frac{1}{T} \log(N(T)).$$
It holds that    $H^{\infty}(\varphi, M\setminus \mathcal{L})\leq h_{\topo}(\varphi)$\footnote{The inequality also holds if classes $\rho$ are considered that carry some periodic orbit and if  we change $\sup$ to $\inf$ in the definition of $\per(\rho)$.}, see  \cite{AlvesPirnapasov,Bowen_fundgroup}. 
\begin{mainthm}\label{thm:link_growth}
Let $M$ and $\varphi$ be as above, and assume that $h_{\topo}(\varphi)>0$.  Then, for any $\varepsilon$ with $0 < \varepsilon< h_{\topo}(\varphi)$, there is a link $\mathcal{L}$ of periodic orbits of $\varphi$,   
such that 
$$H^{\infty}(\varphi, M \setminus \mathcal{L}) > h_{\topo}(\varphi) - \varepsilon.$$
\end{mainthm}

We  obtain also a relative version of Theorem \ref{thm:link_growth}. 
To formulate it, let $\Lambda_1, \Lambda_2$ be two embedded loops in   $M\setminus \mathcal{L}$. 
We call any orbit segment $\gamma:[a,b] \to M\setminus \mathcal{L}$ of $\varphi$ with $\gamma(a)\in \Lambda_1$ and $\gamma(b) \in \Lambda_2$  a  \textit{chord relative to $(\Lambda_1,\Lambda_2)$ of length $b - a$}.  
Denote by $\mathcal{P}$ 
 the homotopy classes of paths from $\Lambda_1$ to $\Lambda_2$  in $M \setminus \mathcal{L}$ relative to $(\Lambda_1, \Lambda_2)$, and write  $\mathrm{len}(\rho)$ for the supremum of lengths of chords in $\rho$.
  We say that $\rho \in \mathcal{P}$ is \textit{singular} if it carries exactly one chord of $\varphi$ relative to $(\Lambda_1,\Lambda_2)$.  Let $N_{\Lambda_1,\Lambda_2}(T)$ be the number of singular  $\rho \in \mathcal{P}$ with $\mathrm{len}(\rho) \leq T$. Define 
$$\mathrm{H}^{\infty}(\varphi,M\setminus \mathcal{L}, \Lambda_1, \Lambda_2):= \limsup_{T\to +\infty} \frac{1}{T} \log(N_{\Lambda_1,\Lambda_2}(T)).$$
\begin{mainthm}\label{thm:link_growth_rel}
Let $M$ and $\varphi$ be as above, and assume that $h_{\topo}(\varphi)>0$.  Then, for any $\varepsilon$ with $0<\varepsilon< h_{\topo}(\varphi)$, there is a link $\mathcal{L}$ of periodic orbits of $\varphi$ and two embedded loops $\Lambda_1, \Lambda_2 \subset M \setminus \mathcal{L}$, such that 
$$\mathrm{H}^{\infty}(\varphi, M \setminus \mathcal{L},\Lambda_1,\Lambda_2) > h_{\topo}(\varphi) - \varepsilon.$$ 
\end{mainthm}
Before we describe the symbolic dynamics setting for these theorems and outline the strategy of the proofs in Section \ref{sec:suspension_horseshoes}, we first discuss some applications to Reeb dynamics.  

\subsection{An application to Reeb flows and forcing}\label{sec:Reebflows_forcing}
A theme that appears in various areas of dynamical systems is the following. Some relatively simple assumptions on the maps in question have quite  
 rich dynamical consequences. A celebrated example is Sharkovski's  forcing order theorem  for 
 continuous mappings on the interval. This theorem implies in particular that the existence of a periodic  orbit of period  $3$  forces the existence of periodic orbits of any 
 integer period. Also, the existence of a periodic orbit of  minimal period that is not a power of $2$ implies that the topological entropy is positive.  
For surface homeomorphisms quite similar phenomena exist. Here, a natural replacement for the period of an orbit is the so-called "braid type" of a periodic orbit. This takes into account the braiding of the periodic orbit as a lift to the  mapping torus. Similarly as in the one-dimensional case, if a certain (isotopy class) of braid appears, then it forces a rich orbit structure, and positive topological entropy etc.\ We refer to 
 \cite{Boyland} for a nice exposition of that theory.  

A straightforward generalization of this theory to flows on $3$-manifolds fails. In fact, any given knot in a $3$-manifold can be realized as a periodic orbit of a volume preserving smooth flow with zero topological entropy, see \cite{RechtmanHurder}.  
As it was  discovered by Alves and Pirnapasov,  \cite{AlvesPirnapasov}, the situation 
 is very different when one considers the same problem in the  category of Reeb flows, and they exhibit  forcing phenomena in that setting. 
Let us recall the central 
 notion from \cite{AlvesPirnapasov}. 
For that recall that a contact structure on a $3$-manifold $M$ is a  
hyperplane field in $TM$ of the form $\xi = \ker \alpha$, where  $\alpha$ is a one-form on $M$ satisfying $\alpha \wedge d\alpha >0$. Such one-forms $\alpha$ are called \textit{contact forms supporting $\xi$}. We will denote the set of supporting contact forms by $\mathcal{C}(\xi)$. Note that $f\alpha \in \mathcal{C}(\xi)$, for any positive  function $f$ on $M$ and $\alpha \in \mathcal{C}(\xi)$.  
Any $\alpha\in \mathcal{C}(\xi)$ defines its \textit{Reeb vector field} $R_{\alpha}$ by the equations  $i_{R_\alpha} d\alpha = 0$ and $\alpha(R_{\alpha}) = 1$.
We denote the flow of $R_{\alpha}$, the \textit{Reeb flow of $\alpha$} by $\varphi_{\alpha}$. 
A link $\mathcal{L}$ in $M$ is called a \textit{transverse link} in $(M,\xi)$ if $\mathcal{L}$ is everywhere transverse to~$\xi$. Contact manifolds $(M,\xi)$ for which every Reeb flow $\varphi_{\alpha}$, $\alpha \in \mathcal{C}(\xi)$, has positive topological entropy exist in abundance, see e.g.\ \cite{MacariniSchlenk2011, Alves1, AlvesMeiwes2018, AlvesColinHonda2017}. On the other hand, many standard constructions yield contact manifolds and supporting contact forms with vanishing topological entropy. The latter are considered in  the following definition.   
\begin{defn}[\hspace*{-3px}\cite{AlvesPirnapasov}]
Let $(M,\xi)$ be a contact $3$-manifold that admits Reeb flows with vanishing topological entropy. A transverse link $\mathcal{L}$ in $(M,\xi)$ is said to \textit{force topological entropy} if every Reeb flow $\varphi = \varphi_{\alpha}$, $\alpha \in \mathcal{C}(\xi)$,  that has $\mathcal{L}$ as a set of Reeb orbits has positive topological entropy.
\end{defn}

The authors of \cite{AlvesPirnapasov} exhibit first examples of transverse links that force topological entropy  in closed contact manifolds.
Moreover, they show that for any closed contact $3$-manifold $(M,\xi)$ that admits a Reeb flow with vanishing topological entropy, there exists a transverse knot in $(M,\xi)$ that forces  topological entropy \cite[Theorem 1.6]{AlvesPirnapasov}.

The work of Momin \cite{Momin} on the contact homology in a link complement plays an important role for the forcing results in \cite{AlvesPirnapasov}. In suitable situations, Momin associates to a triple $(\alpha, \mathcal{L}, \rho)$ of a contact form $\alpha$, a link $\mathcal{L}$ of Reeb orbits 
of $\varphi_{\alpha}$ (i.e.\ $\mathcal{L} \subset \Per(\varphi_{\alpha})$), and a free homotopy class $\rho$ of loops in $M\setminus \mathcal{L}$, the \textit{contact homology in the complement of $\mathcal{L}$} ($\CH_{\mathcal{L},\rho}(\alpha)$).  Following \cite{AlvesPirnapasov}, we say that a contact form $\alpha$ is \textit{hypertight in the complement of $\mathcal{L}$ if $\mathcal{L} \subset \Per(\varphi_{\alpha})$} and if every closed Reeb orbit for $\alpha$ is non-contractible in the complement of $\mathcal{L}$, that is, all disks in $M$ bounding a closed orbit must intersect $\mathcal{L}$ in their interior, see Definition \ref{def:hypertight}. 
In this situation, Alves and Pirnapasov in \cite{AlvesPirnapasov} define \textit{the exponential homotopical growth rate $\Gamma_{\mathcal{L}}(\alpha)$} of $\CH_{\mathcal{L}}(\alpha)$, which is, roughly speaking,  defined as the exponential growth rates in $T$ of homotopy classes $\rho$ with  $\per(\rho)\leq T$ such that $\CH_{\mathcal{L}, \rho}(\alpha)$ is well defined and non-zero, see Section  \ref{sec:proofforc} for details. 
It is shown in \cite{AlvesPirnapasov}, see Theorem \ref{thm:alvesPirna2} below,   that if $\Gamma_{\mathcal{L}}(\alpha_0)>0$ 
for some contact form $\alpha_0$ with $\mathcal{L} \subset \Per(\alpha_0)$, then, for contact forms $\alpha$ with $\mathcal{L} \subset \Per(\varphi_{\alpha})$  it holds that 
$h_{\topo}(\varphi_{\alpha}) \geq \frac{\Gamma_{\mathcal{L}}(\alpha_0)}{\max f_{\alpha}}$, 
where $f_{\alpha} :M \to (0,+\infty)$ is the function with  $\alpha = f_{\alpha}\alpha_0$.  In particular,  
$\mathcal{L}$ forces topological  entropy. A question of Alves and Pirnapasov is the following. 
\begin{qu*}[\hspace*{-3px}\cite{AlvesPirnapasov}]
Can $h_{\topo}(\varphi_{\alpha})$ be approximated by the exponential homotopical growth rates $\Gamma_{\mathcal{L}_k}(\alpha)$ for some $\mathcal{L}_k \subset \Per(\varphi_{\alpha})$? 
\end{qu*} 
The next theorem asserts that this is indeed the case, provided there is a link $\mathcal{L}_0$ such that $\alpha$ is hypertight in the complement of $\mathcal{L}_0$.
\begin{mainthm}\label{thm:CH_recover}
Let $(M,\xi)$ be a closed contact $3$-manifold, $\alpha$ a supporting  contact form with $h_{\topo}(\varphi_{\alpha})>0$ 
and  for which there is a link  $\mathcal{L}_0\subset\Per(\varphi_{\alpha})$  such that $\alpha$ is hypertight in the complement of $\mathcal{L}_0$.  
Then there is a sequence of links $(\mathcal{L}_k)_{k\in \N}$, $\mathcal{L}_k \subset \Per(\varphi_{\alpha})$,   such that 
$\Gamma_{\mathcal{L}_k}(\alpha) \to h_{\topo}(\varphi_{\alpha})$  as $k\to \infty$. 
\end{mainthm}
\begin{rem}
In fact, we obtain that the exponential growth of singular homotopy classes $\rho$ in the complement of  $\mathcal{L}_k$ for which $\CH_{\mathcal{L}_k,\rho}(\alpha)$ is well defined,  approaches $h_{\topo}(\varphi_{\alpha})$.   
\end{rem}
By the work \cite{Colin_broken} on the existence of  broken book decompositions, and the recent results of \cite{Irie_equi}  and \cite{Hryniewicz_generic}, and also \cite{Contreras_generic}, there exists, for a $C^{\infty}$-open and dense set of contact forms $\alpha$, a link $\mathcal{L}_0$ of closed Reeb orbits such that $\alpha$ is hypertight in the complement of $\mathcal{L}_0$, cf.\ Section \ref{sec:proofforc}. This gives the following. 
\begin{cor}\label{cor:generic}
Let $(M,\xi)$ be a closed contact $3$-manifold. Then, for a $C^{\infty}$-open and dense set of supporting contact forms $\alpha$ the following holds. If 
$h_{\topo}(\alpha)>0$, then there is a sequence of links $\mathcal{L}_k\subset \Per(\varphi_{\alpha})$ such that $\Gamma_{\mathcal{L}_k}(\alpha) \to h_{\topo}(\varphi_{\alpha})$  as $k\to \infty$.
\end{cor}
\begin{rem}
In \cite{AlvesPirnapasov}, also the exponential homotopical growth of the Legendrian contact homology in the complement of a link is investigated, a theory that  goes back to  \cite{Alves-thesis}. In view of Theorem \ref{thm:link_growth_rel}, I expect that a version of Theorem \ref{thm:CH_recover} holds involving that homology theory. In the recent work \cite{Barcode_Lagr}, we obtain a version of the theorem for Lagrangian Floer homology in surfaces, and, together with it, relations to barcode entropy, see next remark, and some applications to the growth rate of the length of closed curves for Hamiltonian diffeomorphisms. 
\end{rem}
\begin{rem}\label{rem:barcode}
In the recent work of \c{C}ineli, Ginzburg, and G\"urel,   \cite{CGG_barcode, CGG_subexp_barcode},  the topological entropy of Hamiltonian diffeomorphisms is studied via another 
 symplectic topological growth invariant, the barcode entropy of the Floer chain complex.  It was shown that, in dimension $2$,  the barcode entropy is equal to the topological entropy.  See also  \cite{Mazzucchelli_barcode} where similar results are obtained for geodesic flows, and \cite{beomjun_barcode} with related results for Reeb flows. The above approximation result can be considered as an alternative way to recover topological entropy by quantities defined via holomorphic curves or Floer theory. 
\end{rem}

As a consequence to Theorem \ref{thm:CH_recover} and the estimates in \cite{AlvesPirnapasov} we obtain a result on the existence of links that force topological entropy. To formulate our result, consider, for a given contact manifold $(M,\xi)$, the partial order "$\preceq$" on the set of  supporting contact forms $\mathcal{C}(\xi)$ of $\xi$ defined by $\alpha \preceq \beta$ if $\beta = f\alpha$ for a function $f: M \to [1,\infty)$.  
For a supporting contact form $\alpha$  and a transverse link  $\mathcal{L}\subset {\Per(\varphi_{\alpha})}$ define 
$$h_{\topo}(\alpha,\mathcal{L}) := \min \{h_{\topo}(\varphi_{\beta})\, |\, \alpha \preceq \beta, \, \mathcal{L} \subset {\Per(\varphi_{\beta})}\}.$$
\begin{cor}\label{thm:approximation_transverse}
Let $(M, \xi)$ be a closed contact $3$-manifold that admits a Reeb flow with vanishing topological entropy. Let $\alpha$ be as in Theorem~\ref{thm:CH_recover}.  
Then there is a transverse link $\mathcal{L} \subset {\Per(\varphi_{\alpha})}$ that forces topological entropy.
Moreover, 
\begin{align}\label{eq:approx_L}
h_{\topo}(\varphi_{\alpha}) = \sup \{ h_{\topo}(\alpha,\mathcal{L})\, |\,  {\mathcal{L} \subset {\Per(\alpha)}}\}.
\end{align}
\end{cor}
\begin{proof}[Proof of Corollary \ref{thm:approximation_transverse}] For a $C^{\infty}$-generic  $\alpha$,  let $\mathcal{L}_k\subset \Per(\varphi_{\alpha})$ be 
 the sequence from Corollary \ref{cor:generic}. Then, for any $k\in \N$, and a contact form  $\beta = f\alpha$, $f\geq 1$,  with $\mathcal{L}_k \subset \Per(\varphi_{\beta})$, 
$$
h_{\topo}(\varphi_{\beta}) \geq \frac{\Gamma_{\mathcal{L}_k}(\alpha)}{\max f}\geq \Gamma_{\mathcal{L}_k}(\alpha),$$ and hence $h_{\topo}(\alpha,\mathcal{L}_k) \to h_{\topo}(\varphi_\alpha)$, as $k\to +\infty$. 
\end{proof}
\begin{rem}
The condition $\alpha \preceq \beta$ in the definition of $h_{\topo}(\alpha, \mathcal{L})$ is important in order that \eqref{eq:approx_L} can hold. In fact it follows from \cite{AASS} that, for any  fixed transverse link $\mathcal{L}$,  there is no positive lower bound on the topological entropy $h_{\topo}(\varphi_{\beta})$ for  contact forms $\beta$ such that  $\mathcal{L} \subset \Per(\varphi_{\beta})$, even if the total contact volume is fixed. 
 \end{rem}
\begin{rem}
A similar approximation result 
 for the topological  entropy forced by  braid types in  a specific horseshoe model on the $2$-disk appears in the work of Hall, \cite{Hall1994}. This can be extended to the horseshoes constructed by Katok-Mendoza,  as was observed in  \cite[Theorem B.1]{AlvesMeiwesBraids} using an argument from  \cite{FranksHandel1988}. That approximation result is used to show that with respect to the Hofer metric, topological entropy is lower semi-continuous on the group of Hamiltonian diffeomorphism on surfaces,  \cite{AlvesMeiwesBraids}, and more generally, on the group of area preserving surface diffeomorphisms, \cite{HutchingsBraids}.
\end{rem}

\begin{rem}\label{rem:robustness}
In \cite{MMLL}, we obtained some robustness results   for the topological entropy of geodesic flows on Riemannian manifolds with respect to the $C^0$-distance on the space of Riemannian metrics. 
In the forthcoming work \cite{ADMP}, using also the results obtained in the present article, we generalize some of those results to $3$-dimensional Reeb flows, where in the situation of Reeb flows we study robustness properties with respect to the  $C^0$-distance on contact forms. We 
find an abundance of contact forms, at which $h_{\topo}(\varphi_{\alpha})$, as a function of $\alpha$, is lower semi-continuous with respect to this distance.
\end{rem}

\subsection{Horseshoes}\label{sec:suspension_horseshoes}
Let us now return to the situation of a $C^{1+\alpha}$ flow $\varphi$ without fixed points on a closed $3$-manifold $M$. We give a symbolic dynamics setting for the approximation results above.

Recall that an invariant set $K$ for $\varphi$ is \textit{hyperbolic} if the tangent bundle along $K$ admits a continuous, $d\varphi^t$-invariant splitting  $E\oplus E^s \oplus E^u$ such that $E$ is tangent to the flow lines, $d\varphi^1|_{E^s}:E^s\to E^s$ is contracting and $d\varphi^1|_{E^u}:E^u \to E^u$ is expanding.

Recall that a \textit{full shift in $L$ symbols $\Omega_1, \ldots, \Omega_L$} is a dynamical system $(\Sigma, \sigma)$, where $\Sigma = \{\underline{x} = (x_i)_{i\in \Z}\, |\, x_i \in \{\Omega_1, \ldots, \Omega_L\}, i\in \Z\}$ is the set of bi-infinite sequences in those symbols, equipped with the metric
\begin{align}\label{metricTMS}
d(\underline{x}, \underline{y}) :=  \exp( -\sup\{|n| \, |\, x_n \neq y_n, n\in \Z\}),
\end{align}
and $\sigma: \Sigma \to \Sigma$ is the left shift map on $\Sigma$, defined by $\sigma((x_i)_{i\in \Z}) := (x_{i+1})_{i\in \Z}$. Let $r: \Sigma \to (0,+\infty)$ be a H\"older continuous function, bounded away from zero and infinity. 
Define $r_n: \Sigma \to \R$, $n\in \Z$, by $r_n := r + r\circ \sigma  + \cdots + r\circ \sigma^{n-1}$ for $n\geq 1$, $r_0 := 0$, and $r_{n} := -r_{|n|} \circ \sigma^{-|n|}$ for $n\leq -1$.
Let 
\begin{align}\label{Sigmar}
\Sigma_r := \{(\underline{x}, t) \, |\, \underline{x} \in \Sigma, 0 \leq t < r(\underline{x})\}.
\end{align}
The topological Markov flow induced by $(\Sigma, \sigma)$ with roof function $r$ is the flow $\sigma_r: \Sigma_r \to \Sigma_r$ given by 
\begin{align}\label{sigmar}
\sigma_r(\underline{x},t) =(\sigma^n(\underline{x}),t + \tau - r_n(\underline{x})),
\end{align}
where $n$ is the unique integer such that $0 \leq t + \tau -r_n(\underline{x}) < r(\sigma^n(\underline{x}))$. 
$\Sigma_r$ is equipped with the Bowen-Walters metric, \cite{Bowen-Walters}, induced by $d$  (for its definition see also \cite{LimaSarig2019}). With respect to that metric, $\sigma_r$ is continuous.

A \textit{local cross section} (in short, \textit{local section}) to the flow $\varphi^t$ is an embedded connected compact surface $D$ with piecewise smooth boundary that is everywhere transverse to the trajectories of the flow. This means that the vector field generating $\varphi$ is nowhere tangent to $D$ (including its boundary). 
 A local section $D$  determines a \textit{Poincar\'e return map} $f_{D}:\dom(f_D)\to D$ which is defined on $$\dom(f_D) = \{x \in D \, |\, \varphi^t(x) \in D \text{ for some } t>0\},$$ and sends $x\in \dom(f_D)$ to the first intersection point of the trajectory $(\varphi^t(x))_{t>0}$ with $D$. Define the \textit{roof function} $R_D: \dom(f_D) \to (0,+\infty)$ by $R_D(x):= \inf\{t>0 \, |\, \varphi^t(x) \in D\}$.  This means that $f_D(x) = \varphi^{{R_D}(x)}(x)$ for $x \in \dom(f_D)$.
Similarly define $f^{-1}_D$, $R^{-1}_D$ to be the Poincar\'e return map and roof function on $D$ for the flow $\varphi^{-t}, t\in \R$.

 We say that a section $D$ is \textit{rectangular} if $D$ is a rectangle, which means that $D$ is the image of an embedding $\iota:[0,1]^2 \to M$. 
In this case, we denote the boundary of $D$ as $\partial D = v_-\cup v_+ \cup h_- \cup h_+$ with \textit{sides} $v_-:= \iota(\{0\}\times [0,1])$, $v_+:= \iota(\{1\}\times [0,1])$, $h_-:= \iota([0,1]\times \{0\})$, $h_+:= \iota([0,1]\times \{1\})$.  Let $D_i, i=1,2$, be two rectangles and write $\partial D_i = v^i_-\cup v^i_+\cup h_-^i \cup h_+^i$. We say that  $D_1$ \textit{crosses}  $D_2$ if $v^1_{\pm} \cap D_2 = \emptyset$,  $h^2_{\pm} \cap D_1 =\emptyset$,  $h^1_{\pm} \cap D_2$ are two paths with their endpoints in the segments $v^2_-$ and $v_+^2$, and if $v_{\pm}^2\cap D_1$ are two paths with their endpoints in $h_-^1$ and $h_+^1$. 

Let $D$ be a rectangular section, and $f=f_D$.
We say that a finite collection of rectangular sections $D_1, \ldots, D_L \subset D$ is of \textit{Markov type} if 
\begin{itemize}
\item $D_1, \ldots, D_L \subset \dom(f)$
\item $f_D(D_i)$ crosses   $D_j$, for all $i,j\in\{1, \ldots, L\}$.  
\end{itemize}
For a collection of rectangles $D_1, \ldots, D_L$ of Markov type, write $\hat{V} = \bigcup_{i=1}^L D_i$. Define $D_i^j$, $j\in \Z$, as follows. We put $D_i^0:=D_i$, we define inductively for $j>0$, $D_i^j := f^{-1}(D^{j-1}_i \cap f(\hat{V}))$, and define inductively for $j<0$,  $D_i^j := f(D_i^{j+1} \cap \hat{V})$.
We set \begin{align}\label{hatVj}
\hat{V}^j = \bigcup_{i=1}^L D_i^j,
\end{align}
 and define the \textit{maximal invariant subset in $D$ induced by $D_1, \ldots, D_L$ } to be $\hat{K} = \bigcap_{j\in \Z} \hat{V}^j$.  
We say that $D_1, \ldots, D_L$ \textit{induce a full shift} on $\hat{K}$ 
if there is a full shift $(\Sigma, \sigma)$ in $L$ symbols $\{\Omega_1, \ldots, \Omega_L\}$ such that for every  $\underline{x} = (x_j)_{j\in \Z} = (\Omega_{i_j})_{j\in \Z} \in \Sigma$, $$\pi(\underline{x}):=\bigcap_{j\in \Z} D_{i_j}^{j}$$  defines a unique point in $\hat{K}$. Note that in that case, $f \circ \pi = \pi \circ \sigma$.

We say that a compact  $\varphi$-invariant  hyperbolic 
 set $K\subset M$ is a 
\textit{horseshoe over a local section $D$} if there is $L\in \N$, disjoint rectangular subsections $D_1, \ldots, D_L \subset 
D$, a topological Markov flow $\sigma_r: \Sigma_r \to \Sigma_r$ over a full shift $\sigma:\Sigma \to \Sigma$ in $L$ symbols $\Omega_1, \ldots, \Omega_L$, and a map $\pi_r : \Sigma_r \to M$ such that 
\begin{enumerate}[(i)]
\item $K = \im(\pi_r)$;
\item $\pi_r$ is injective and  H\"older continuous;
\item $\pi_r \circ {\sigma}^t_r = \varphi^t \circ \pi_r$ for all $t$;
\item $\pi_r(\underline{x}, t) \in D$ if and only if $t=0$;
\item $ \hat{K}= K\cap D$ is the maximal invariant subset given by  $D_1, \ldots, D_L$, as above, and  $D_1, \ldots, D_L$ induce the full shift $\sigma$.  
\end{enumerate}
We say that $D_1, \ldots, D_L$ are the \textit{Markov rectangles} and $(\Sigma_r, \sigma_r, \pi_r)$ the \textit{coding of }$K$. 
\begin{rem}
By \cite{Bowen-onedimensional}, any one-dimensional basic set, that is, a locally maximal, transitive, hyperbolic invariant set with a dense set of periodic orbits, is in fact an embedded Markov flow over a shift of finite type. In the definition of a horseshoe given here, we include some additional geometric information how the shift is obtained. A crucial feature that is used in this paper is property (iv), namely that a full shift can be detected already at the first return to $D$, rather than at some higher iterate of the return map. The other geometric conditions for the horseshoes are more standard, and their analogous conditions for diffeomorphisms are satisfied by the horseshoes obtained in \cite[Suppl.]{Hasselblatt-Katok}, and also by those in other related approximation results, for example in \cite{Gelfert2016,LuzzattoSalas}.   
\end{rem}

To formulate the next theorem we introduced some further notation. Let $K$ be  a  horseshoe as above. 
Let $\mathcal{S} = \bigcup_{n\in \N}\{1,\ldots, L\}^n$ be the set of all finite ordered tuples in $\{1,\ldots, L\}$. For $(i_0,\ldots, i_{n-1})$ in $\mathcal{S}$, let  
$\underline{\mathfrak{a}} = (\mathfrak{a}_i)_{i\in\Z}\in \Sigma$ be the sequence  that is obtained by periodically extending $\mathfrak{a}_0\cdots \mathfrak{a}_{n-1} = \Omega_{i_0} \cdots \Omega_{i_{n-1}}$, and denote by 
$P_{(i_0,\ldots, i_{n-1})}:[0,r_n(\underline{\mathfrak{a}})] \to M$ the loop $t\mapsto  \pi_r(\underline{\mathfrak{a}},t), 0\leq t\leq r_n(\underline{\mathfrak{a}})$.
For a negative sequence   $\underline{\mathfrak{a}} = (\mathfrak{a}_i)_{i\leq 0}$ and a positive sequence  $\underline{\mathfrak{b}}= (\mathfrak{b}_i)_{i\geq 0}$ in  $\{\Omega_1,\ldots, \Omega_L\}$, and a $(n-2)$-tuple  $(i_1,\ldots, i_{n-2})$ in $\mathcal{S}$, $n\geq 3$, we put  $\underline{\mathfrak{c}} = \underline{\mathfrak{c}}^{\underline{\mathfrak{a}}, \underline{\mathfrak{b}};i_1, \ldots, i_{n-2}}\in \Sigma$, 
$\underline{\mathfrak{c}} = (\mathfrak{c}_j)_{j\in \Z}$,
\begin{align}\label{eq:c} 
\mathfrak{c}_j  =\begin{cases} \mathfrak{a}_j  &\text{ if } j\leq 0, \\ \Omega_{i_j} &\text{ if } 1\leq j \leq n-2, \\ \mathfrak{b}_{j-(n-1)} &\text{ if } j\geq n-1.
\end{cases}
\end{align}
Denote by  $C_{i_1,\ldots, i_{n-2}}= C_{i_1,\ldots, i_{n-2}}(\underline{\mathfrak{a}},\underline{\mathfrak{b}})$ 
 the path $C_{i_1,\ldots, i_{n-2}}: [0,r_{n}(\underline{\mathfrak{c}})] \to M$,   $t\mapsto  \pi_r(\underline{\mathfrak{c}}, t)$, $0\leq t\leq r_{n}(\underline{\mathfrak{c}})$.
\begin{mainthm}\label{thm:main_coding}
Let $M$ be a closed $3$-manifold, and $\varphi$ be a fixed-point-free $C^{1+\alpha}$ flow with $h_{\topo}(\varphi)>0$. Then, for any $\varepsilon$ with $0<\varepsilon< h_{\topo}(\varphi)$,   there is a rectangular local section  $D$,  a  horseshoe $K$  over $D$  coded by $(\Sigma_r, \sigma_r, \pi_r)$ with Markov rectangles $D_1, \ldots, D_L$, $L \in\N$, and a link  $\mathcal{L}$ consisting of hyperbolic orbits of  $\varphi$,  such that 
\begin{enumerate}
\item $h_{\topo}(\sigma_r) \geq h_{\topo}(\varphi)-\varepsilon$; 
\item the link $\mathcal{L}$ intersects $D$ in the corners of $D$ and $D_1, \ldots, D_n$; 
\item the map $\Pi:\mathcal{S} \to \widehat{\pi}(M\setminus \mathcal{L})$ that maps a tuple $(i_0,\ldots, i_{n-1})$ to the homotopy class of $P_{(i_0,\ldots, i_{n-1})}$ in $M\setminus \mathcal{L}$ is injective up to cyclic permutation of $(i_0,\ldots, i_{n-1})$, and the image consists only of singular homotopy classes; 
\item for any $\underline{\mathfrak{a}} = (\mathfrak{a}_i)_{i\leq0}, \underline{\mathfrak{b}} = (\mathfrak{b}_i)_{i\geq 0}$, there are segments $l_u$ in $\hat{V} = \bigcup_{i=1}^L D_i$ and  $l_s \in f_D(\hat{V})$ with boundaries in $\partial \hat{V}$ and $\partial( f_D(\hat{V}))$, respectively, 
 such that for any $(i_1,\ldots, i_{n-2})$ in $\mathcal{S}$,  
  the path $C_{i_1,\ldots, i_{n-2}}$ is a chord from $l_u$ to $l_s$, its homotopy class $\sigma_{i_1,\ldots, i_{n-2}}$ of paths relative to $(l_u, l_s)$ is singular, and the map that maps $(i_1,\ldots, i_{n-2})\in \mathcal{S}$ to $\sigma_{i_1,\ldots, i_{n-2}}$ is 
 injective; 
\item there are closed curves $\Lambda_u$ and $\Lambda_s$ in $M \setminus \mathcal{L}$ with $\Lambda_u \cap \hat{V} = l_u$, $\Lambda_s \cap f_D(\hat{V}) = l_s$,   such that with $\underline{\mathfrak{a}}$, $\underline{\mathfrak{b}}$, and $\underline{\mathfrak{c}}$ as in $(4)$, the chord  $C_{i_1,\ldots, i_{n-2}}$ is, up to reparametrization, still the only chord from $\Lambda_u$ to $\Lambda_s$ among homotopic paths in $M\setminus \mathcal{L}$ relative to $(\Lambda_u,\Lambda_s)$. 
\end{enumerate}
\end{mainthm}
\begin{rem}
In fact, if $\underline{\mathfrak{a}}$ and $\underline{\mathfrak{b}}$ are periodic, the segments $l_u$ and $l_s$  
 in (4) are contained in the intersections of local unstable and  stable manifolds of $\pi_r((\underline{\mathfrak{a}},0))$ and $\pi_r((\underline{\mathfrak{b}},0))$ with $D$, respectively.  
\end{rem}

\begin{rem}
The Markov flows $\sigma_r$ on $\Sigma_r$ in Theorem \ref{thm:main_coding} are embedded in  the Markov flow of the main theorem in \cite{LimaSarig2019} (see Theorem~\ref{thm:fullmeasurecoding} below). More precisely, if   $(\Sigma'_r, \sigma'_r)$ are the Markov flows from  Theorem~\ref{thm:fullmeasurecoding} below, then there exist an injective map $\iota: \Sigma \to \Sigma'$ and $N\in \N$ (depending on $\varepsilon$)  such that
 $\iota \circ  \sigma = (\sigma')^N \circ \iota$, and the roof function $r$ is given by $r = r'_N \circ \iota$. 
\end{rem}

From Theorem \ref{thm:main_coding} we will derive Theorems  \ref{thm:link_growth}, \ref{thm:link_growth_rel}, and \ref{thm:CH_recover}. 
Let us remark on the strategy for the proof of Theorem  \ref{thm:main_coding}. As indicated above, our main framework will be given by the work of Lima and Sarig, \cite{LimaSarig2019}. Applied to a measure of maximal entropy for $\varphi$, the main result in \cite{LimaSarig2019}  yields  a countable topological Markov flow that provides a coding of the dynamics of $\varphi$ and has the same topological entropy.  The countable topological Markov flow arises from a countable topological Markov shift $(\Sigma',\sigma')$, and the latter codes the dynamics of a Poincar\'e return map to a (discontinuous) Poincar\'e section $\Lambda\subset M$ for the flow. 
Our task is then to find suitable full shifts of finite type  $(\Sigma, \sigma)$ in an iterate of $(\Sigma',\sigma')$ with the desired properties. 
We can divide this process roughly into two steps. In a first step, we use the specific properties of the coding $\pi': \Sigma'\to \Lambda$ %
to obtain a finite collection of periodic orbits of the Poincar\'e return map with suitable separation properties. This allows us to obtain a shift that is induced by a collection of rectangles of Markov type. The arguments are variations of those of Katok and Mendoza that lead to their approximation result in \cite[Suppl.]{Hasselblatt-Katok}. Moreover, by carrying out this process carefully, we can guarantee that the full shift is induced already at the first return to a small section inside $\Lambda$. Let us refer to the compact invariant set of the flow obtained  this way as the first horseshoe. In a second step, we modify a construction of Fried, \cite{Fried83}, and obtain the link  $\mathcal{L}$ which is a union of $L+1$ links $\mathcal{L}_0, \mathcal{L}_1, \ldots, \mathcal{L}_L$, where each link $\mathcal{L}_i$ consists of three orbits of the first horseshoe that span a pair of pants $F_i$. The surfaces $F_i$ and $F_j$ are disjoint if $i\neq j$, and $i,j\neq 0$. Inside each $F_i$ lies the rectangular local section $D_i$, where $D_0=D$. This gives us inside the first horseshoe a second horseshoe, the horseshoe of Theorem \ref{thm:main_coding}. In all the steps of the construction, the drop of the topological entropy can be made as small as desired. 
Let us sketch the idea of how property (3) is obtained, properties (4) and (5) are similar. Different patterns of intersections of the  horseshoe orbits with the sections $D_1, \ldots, D_L$ give rise to a different linking behaviour with the link $\bigcup_{i=1}^L\mathcal{L}_i$. This will imply that two different horseshoe orbits are not homotopic in the complement of $\bigcup_{i=1}^L\mathcal{L}_i$. 
Taking into account also the surface $F_0$ that contains the  local section $D=D_0$, and hence also the sections $D_1, \ldots, D_L$, we can show  that the horseshoe orbits are not homotopic in the complement of $\mathcal{L}$ to any other orbit of the flow: any orbit not contained in the horseshoe that has, say, $k$ intersections with  $D_1, \ldots, D_L$ must have more than $k$ intersections with $D_0=D$ and hence is not homotopic in $M\setminus \mathcal{L}$ to the horseshoe orbit with the same intersection pattern with the surfaces $D_1, \ldots, D_L$.

\textbf{Organization of the paper:} In Section \ref{sec:Lima_Sarig_thm}, some properties of the Markov flows of \cite{LimaSarig2019} are recalled that will be relevant for the proof of Theorem \ref{thm:main_coding}, in Section \ref{sec:construction_horseshoes} we construct the horseshoes and prove Theorem  \ref{thm:main_coding}, as well as Theorems \ref{thm:link_growth} and \ref{thm:link_growth_rel}. Finally, in Section \ref{sec:proofforc}, Theorem \ref{thm:CH_recover} is proved.

\textbf{Acknowledgements:}
I am grateful to Umberto Hryniewicz at the RWTH  Aachen, and Lev Buhovsky, Yaron Ostrover and Leonid Polterovich at Tel Aviv University for their support. Special thanks go to Marcelo Alves, Lucas Dahinden, and Abror Pirnapasov for their encouragement and very helpful discussions and comments related to this work. 
\section{Preliminaries: The Markov flows of Lima-Sarig}\label{sec:Lima_Sarig_thm}

The Markov flows in Theorem \ref{thm:main_coding} will be obtained, as mentioned above, from countable Markov flows constructed by Lima and Sarig in \cite{LimaSarig2019}. In this section we recall the  main theorem and properties of the Markov flows in \cite{LimaSarig2019} that will be relevant for the proof of Theorem~\ref{thm:main_coding}. We refer the reader to \cite{LimaSarig2019, Sarig_diffeo2013} for details. 

\subsection{Main properties of the Markov flows of Lima-Sarig}\label{sec:main_LimaSarig}
We recall the notion of a countable topological Markov shift. Let $\mathcal{G}$ be a directed graph with a countable set of vertices $\mathcal{A}$. 
If $x$ and $y$ are two (possibly equal) vertices,  we write $x\rightarrow y$ if there is an edge from $x$ to $y$. We consider the set $$\Sigma = \Sigma(\mathcal{G}) = \{\underline{x} = (x_i)_{i\in \Z}\, |\, x_i \rightarrow x_{i+1} \text{ for all } i \in \Z\} $$ of infinite paths on $\mathcal{G}$. 
The \textit{topological Markov shift (TMS)}  associated to $\mathcal{G}$ is the discrete topological dynamical system $(\Sigma,\sigma)$ where $\sigma: \Sigma \to \Sigma, (x_i)_{i\in \Z} \mapsto (x_{i+1})_{i\in \Z}$ is the left shift map  on $\Sigma$. The metric on $\Sigma$ is defined as in \eqref{metricTMS}, in the case of a full shift.
Let $r: \Sigma \to (0,+\infty)$ be a H\"older-continuous function, bounded away from zero and infinity. Analogous to the situation of a full shift (cf.\ formulas \eqref{Sigmar} and \eqref{sigmar}) one defines the Markov flow $(\Sigma_r, \sigma_r)$ induced by $(\Sigma,\sigma)$ with roof function $r$. The space $\Sigma_r$ is equipped with the Bowen-Walters metric, with respect to which $\sigma_r$ is continuous.    
The shift invariant \textit{set of recurrent paths} $\Sigma^{\#}(
\mathcal{G})$ on $\mathcal{G}$ is the set of $\underline{x} = (x_t)_{t\in \Z} \in\Sigma(\mathcal{G})$ for which there is $v,w\in \mathcal{A}$ such that $x_t = v$ for $\infty$-many $t>0$, and $x_t = w$ for $\infty$-many $t<0$. 
\begin{thm}\label{thm:fullmeasurecoding}\cite[Theorem 1.2]{LimaSarig2019}
Let $M$ be closed $3$-manifold, $\varphi$ a $C^{1+\alpha}$ flow, and let $\mu$ be a $\varphi$-invariant ergodic Borel probability measure with $h_{\mu}(\varphi)>0$. Then there exists a topological Markov flow $\sigma_r:\Sigma_r \to \Sigma_r$ and a map $\pi_r:\Sigma_r \to M$ such that 
\begin{enumerate}
    \item $r: \Sigma \to (0,+\infty)$ is H\"older continuous and bounded away from zero and infinity.
    \item $\pi_r$ is H\"older continuous with respect to the Bowen-Walters metric.
    \item $\pi_r \circ \sigma_r^t = \varphi^t \circ \pi_r$ for all $t \in \R$. 
    \item $\pi_r(\Sigma_r^{\#})$ has full measure with respect to $\mu$.
    \item If $p= \pi_r(\underline{x},t)$ where $x_i = v$ for infinitely many $i<0$ and $x_i = w$ for infinitely many $i>0$, then $\# \{ (\underline{y},s) \in \Sigma^{\#}_r \, |\, \pi_r(\underline{y},s) = p\} \leq N(v,w)  < \infty$, for some function $N:\mathcal{A}^2 \to \N$. 
    \item $\exists N = N(\mu) < \infty$ such that $\mu$-a.e.\ $p\in M$ has exactly $N$ pre-images in $\Sigma^{\#}_r$. 
\end{enumerate}
\end{thm}

\subsection{Adapted Poincar\'e sections}
Throughout the remaining of Section \ref{sec:Lima_Sarig_thm}, $M$, $\varphi$ and $\mu$ as in Theorem \ref{thm:fullmeasurecoding} are fixed. Denote by $X$ the vector field generating $\varphi$. Also a Riemannian metric on $M$ is fixed.
We will recall very briefly some important notions and constructions of \cite{LimaSarig2019} that are used to prove Theorem  \ref{thm:fullmeasurecoding}, and we refer the reader to \cite{LimaSarig2019} for more details. 

A \textit{Poincar\'e section} for $\varphi$ is a Borel-set $\Lambda\subset M$ such that for all $p\in M$, there are $t>0$ and $s<0$ such that $\varphi^t(p), \varphi^s(p)\in\Lambda$. 
The \textit{roof function} $R_{\Lambda}: \Lambda \to (0,\infty)$ is defined by $R_{\Lambda}(p) := \inf \{ t>0 \, |\, \varphi^t(p) \in \Lambda\}$, and the 
 \textit{Poincar\'e return map} $f_{\Lambda}:\Lambda \to \Lambda$ is defined as 
 $f_{\Lambda}(p) := \varphi^{R_{\Lambda}(p)}$. We write also $f = f_{\Lambda}$.  
The measure $\mu$ induces an $f_{\Lambda}$-invariant  measure $\mu_{\Lambda}$ on $\Lambda$; for a formula see \cite{LimaSarig2019}.

In \cite{LimaSarig2019} a constant $\mathfrak{r}$ is defined (depending only on $\varphi$ and $M$), and 
 \textit{standard} Poincar\'e sections are considered. These are defined to satisfy  $R_{\Lambda} < \mathfrak{r}$ and to be   of the form 
$\Lambda = \bigcup_{i=1}^N S_r(p_i)$, where $S_r(p_i)$ are pairwise disjoint transverse discs. Moreover, $S_r(p_i)$ are images of the exponential map at $p_i$ of a disc of radius $r< \mathfrak{r}$ in the orthogonal complement of $X_{p_i}$ in $T_{p_i}M$.
Since $\partial \Lambda$ is non-empty, the return maps to standard Poincar\'e sections have points of discontinuity.  Denote by $
\mathfrak{S}(\Lambda)$ the set of points $x\in 
\Lambda$ such that there is $t
\in \{-1,0,1\}$ with $f^t_{\Lambda}(x) \in \partial \Lambda$.
A standard Poincar\'e section $\Lambda$ is called \textit{adapted} if
\begin{itemize}
\item $\mu_{\Lambda}(
    \mathfrak{S}(\Lambda)) = 0$;
\item $\lim_{n\to \infty} \log \dist_{\Lambda}(f^n_{\Lambda}(p),\mathfrak{S}(\Lambda)) = 0$ for $\mu_{\Lambda}$-a.e.\ $p \in \Lambda$;
\item $\lim_{n\to \infty} \log \dist_{\Lambda}(f^{-n}_{\Lambda}(p),\mathfrak{S}(\Lambda)) = 0$ for $\mu_{\Lambda}$-a.e.\ $p \in \Lambda$. 
\end{itemize}
Here, $\dist_{\Lambda}$ is the distance induced by the Riemannian metric on $M$ to $\Lambda$, where points on separate discs are defined to have infinite distance.
\begin{thm}\cite[Theorem 2.8]{LimaSarig2019}
There are adapted Poincar\'e sections $\Lambda$ for $\mu$ with arbitrarily small roof functions. 
\end{thm}

By using adapted Poincar\'e sections, the authors of \cite{LimaSarig2019} carry over to flows some parts of the theory of   
\cite{Sarig_diffeo2013} for surface diffeomorphisms.  

By Oseledets Theorem,  $\chi(p,v) := \lim_{t\to \infty} \frac{1}{t} \log \|(d\varphi^t)_pv\|$ exist for $\mu$-a.e.\ $p\in M$, and non-zero $v \in T_pM$. The values $\chi(p,\cdot)$ are called the \textit{Lyapunov exponents} at $p$. Since $M$ is $3$-dimensional and  $h_{\mu}(\varphi)>0$, it follows from  Ruelle's inequality that  $\mu$ is \textit{$\chi_0$-hyperbolic}  for any $0< \chi_0< h_{\mu}(\varphi)$, which means that $\mu$-a.e.\ $p \in M$ has a Lyapunov exponent in $(-\infty,\chi_0)$, another in $(\chi_0, \infty)$, and a third that vanishes. 
If $\Lambda$ is an adapted Poincar\'e section, then 
also the return map  $f_{\Lambda}$ has well-defined Lyapunov exponents $\mu_{\Lambda}$-a.e.,  and if $\mu$ is $\chi_0$-hyperbolic and $\chi := \chi_0\inf R_{\Lambda}$, then  the return map $f_{\Lambda}$ has one Lyapunov exponent in $(-\infty, -\chi)$ and another in $(\chi, \infty)$ for $\mu_{\Lambda}$-a.e.\ $x \in \Lambda$, see \cite[Lemma 2.6]{LimaSarig2019}.

\subsection{Non-uniform hyperbolic set and Pesin charts}
The \textit{non-uniform hyperbolic set}  $\nuh_{\chi}(f_{\Lambda})$
is the set of $x \in \Lambda\setminus (\bigcup_{n\in \Z} f^{-n}(\mathfrak{S}))$ such that 
$T_{f^n(x)}\Lambda = E^u(f^n(x)) \oplus E^s(f^n(x))$, $n\in \Z$, where $E^u$ and $E^s$ are one-dimensional linear subspaces invariant under $df_{\Lambda}$,  and 
\begin{enumerate}
    \item $\lim_{n\to \pm \infty} \frac{1}{n}\log \| df^n_x v\|< -\chi$ for all non-zero $v\in E^s(x)$,
    \item $\lim_{n\to \pm \infty} \frac{1}{n}\log \| df^{-n}_x v\|< -\chi$ for all non-zero $v\in E^u(x)$, 
    \item $\lim_{n\to \pm \infty} \frac{1}{n}\log | \sin \measuredangle (E^s(f^n(x)), E^u(f^n(x)))| = 0$.
    \end{enumerate}
By the above, and Oseledets Theorem, $\mu_{\Lambda}(\nuh_{\chi}(f_{\Lambda})) = 1$ for $\chi = \chi_0 \inf R_{\Lambda}$. 
\begin{thm}\cite[Theorem 3.1]{LimaSarig2019}\label{Theorem3.1}
There is a measurable family of linear transformations $C_{\chi}(x):\R^2 \to T_x\Lambda$,  $x\in \nuh_{\chi}(f_{\Lambda})$, with $\|C_{\chi}\|\leq 1$ and a constant $C_{\varphi}$  such that 
$C_{\chi}(f(x))^{-1} \circ df_x \circ C_{\chi}(x) = \left(\begin{smallmatrix} A_x & 0 \\ 0 & B_x \end{smallmatrix}\right)$, with $C_\varphi^{-1} \leq |A_x | \leq e^{-\chi}$, and $e^{\chi} \leq |B_x| \leq C_\varphi$. 
\end{thm}

For any $x \in \nuh_{\chi}(f_{\Lambda})$ and $Q(x)$ sufficiently small,  the \textit{Pesin chart of size $\eta$} for $0<\eta\leq Q(x)$ is defined as 
$\Psi^{\eta}_x: [-\eta, \eta]^2 \to \Lambda \setminus \mathfrak{S}$, $\Psi^{\eta}_x(u,v) := \mathrm{Exp}_x\left[C_{\chi}(x)\left(\begin{smallmatrix} u \\ v\end{smallmatrix}\right)\right] $,
where $\mathrm{Exp}_x$ denotes the exponential map with respect to the induced metric on $\Lambda$. Write $\Psi_x$ instead of $\Psi_x^{Q(x)}$. 
A family $Q_{\epsilon}(x)$, $\epsilon>0$ sufficiently small, of parameters $Q(x)$ as above is carefully chosen, depending in particular on $\dist(x,\mathfrak{S})$. 
It is shown, cf.\  \cite[Lemma 3.3 and 3.4]{LimaSarig2019}, that for a suitable $f_{\Lambda}$-invariant subset $\nuh^*_{\chi}(f_{\Lambda})\subset \nuh_{\chi}(f_{\Lambda})$ of full measure it holds that   
$\lim_{n \to \pm \infty} \frac{1}{n} \log Q_{\epsilon}(f^n(x)) = 0$, and that there exists a  Borel function $q_{\epsilon}: \nuh_{\chi}^*(f) \to (0,1)$  such that $0< q_{\epsilon}(x) \leq \epsilon Q_{\epsilon}(x)$ and $e^{-\epsilon/3} \leq \frac{q_{\epsilon} \circ f}{q_{\epsilon}} \leq e^{\epsilon/3}$.
Written in  the local coordinates given by the Pesin charts, the Poincar\'e map $f_{\Lambda}$ is close to a uniformly  hyperbolic linear map: 
\begin{thm}\cite[Theorem 3.2]{LimaSarig2019}\label{thm:closehyp}
For all $\epsilon$ small enough the following holds. For each $x\in \nuh_{\chi}(f)$, $f_x:= \Psi^{-1}_{f(x)} \circ f_{\Lambda} \circ \Psi_x$  is well defined and injective on $[-Q_{\epsilon}(x), Q_{\epsilon}(x)]^2$, and can be put into the form $f_x(u,v) = (A_x u + h_x^1(u,v), B_xv + h^2_x(u,v))$, where 
\begin{enumerate}
\item $C_{\varphi}^{-1} \leq |A_x|\leq e^{-\chi}$ and $e^{\chi}\leq |B_x| \leq C_{\varphi}$, with $C_{\varphi}$ as in Theorem~\ref{Theorem3.1};
\item $h_x^i$ are $C^{1+\alpha/2}$-functions such that $h_x^i(0) = 0$, 
$(\nabla h_x^i)(0) = 0$;
\item $\|h_x^i\|_{C^{1+\alpha/2}} < \epsilon$ on $[-Q_{\epsilon}(x), Q_{\epsilon}(x)]^2$.
\end{enumerate}
\end{thm}
The property that the maps $\Psi^{-1}_{f(x)} \circ f_{\Lambda} \circ \Psi_x$ in Theorem \ref{thm:closehyp} are close to a uniformly hyperbolic map is stable under small perturbations of the charts. The notion of  \textit{$\epsilon$-overlapping} of two charts $\Psi^{\eta_1}_{x_1}$ and  $\Psi^{\eta_2}_{x_2}$, in symbols $\Psi^{\eta_1}_{x_1} \overset{\epsilon}{\approx} \Psi^{\eta_2}_{x_2}$, was introduced in \cite{Sarig_diffeo2013} and adapted in \cite{LimaSarig2019}. This notion includes a  condition on  $\eta_1/\eta_2$ being small, as well as closeness condition between the pairs $(x_1,C_{\chi}(x_1))$ and $(x_2, C_{\chi}(x_2))$. We refer to \cite{Sarig_diffeo2013, LimaSarig2019} for the precise definition and further discussion. 
If $\epsilon$ is small enough, $x, y\in \nuh_{\chi}(f)$, and $\Psi^{\eta'}_{f(x)} \overset{\epsilon}{\approx} \Psi^{\eta}_{y}$, then 
$f_{xy}:= \Psi^{-1}_{y} \circ f_{\Lambda} \circ \Psi_x:[-Q_{\epsilon},Q_{\epsilon}]^2 \to \R^2$ is close to a uniformly hyperbolic linear map (see \cite[Corollary 3.6]{LimaSarig2019} for a detailed statement).

In the following, we recall the construction of the Markov partition in Lima-Sarig in \cite{LimaSarig2019} for the Poincar\'e return map $f= f_{\Lambda}:\Lambda \to \Lambda$ and some of its properties that are relevant for our applications. Main parts of the construction go back to \cite{Sarig_diffeo2013} in the situation of surface diffeomorphisms.  

\subsection{Generalized pseudo orbits, admissible manifolds, and shadowing}\label{sec:GPO}
\begin{defn}
A \textit{$\epsilon$-double chart} is an ordered pair $$\Psi_x^{p^u,p^s}:= (\Psi_x|_{[-p^u,p^u]^2}, \Psi_x|_{[-p^s,p^s]^2}),$$ where $x \in \nuh_{\chi}(f)$ and $0< p^u, p^s \leq Q_{\epsilon}(x)$.
\end{defn}
As in \cite{LimaSarig2019} we use the  notation: $p\wedge q := \min\{p,q\}$.
\begin{defn}
A \textit{$\epsilon$-generalized pseudo-orbit (gpo)} is a sequence $(v_i)_{i\in \Z} = (\Psi_{x_i}^{p_i^u,p_i^s})_{i \in \Z}$ of $\epsilon$-double charts such that, for all $i \in \Z$,
\begin{enumerate}[({gpo}1)]
\item $\Psi_{f(x_i)}^{p_{i+1}^u \wedge p_{i+1}^s} \overset{\epsilon}{\approx} \Psi_{x_{i+1}}^{p_{i+1}^u \wedge p_{i+1}^s}$ and 
 $\Psi_{f^{-1}(x_{i+1})}^{p_{i}^u \wedge p_{i}^s} \overset{\epsilon}{\approx} \Psi_{x_{i}}^{p_{i}^u \wedge p_{i}^s}$; 
\item $p_{i+1}^u = \min\{e^{\epsilon} p_i^u, Q_{\epsilon}(x_{i+1})\}$ and $p_i^s = \min\{e^{\epsilon}p_{i+1}^s, Q_{\epsilon}(x_i)\}$. 
\end{enumerate}
\end{defn}
Conditions (gpo1) and (gpo2) are abbreviated by $v_i \overset{\epsilon}{\rightarrow} v_{i+1}$. It is often useful to consider \textit{positive and negative half-gpos}, that is, sequences $(v_i)_{i\geq 0}$ and $(v_i)_{i\leq 0}$, respectively, that satisfy $v_i\overset{\epsilon}{\rightarrow} v_{i+1}$, $i\in \Z$.

We proceed by recalling the definitions of admissible manifolds, graph transforms and the existence of local stable and unstable manifolds in the context of \cite{LimaSarig2019}.
\begin{defn}
A \textit{$s$-admissible manifold} in a $\epsilon$-double chart $v =  \Psi_x^{p^u,p^s}$ is a set of the form 
$\Psi_x(\{(t, F(t)) \, |\, |t|\leq p^s\})$, where $F:[-p^s, p^s] \to \R$ satisfies: 
\begin{enumerate}
\item $|F(0)| \leq 10^{-3}(p^u \wedge p^s)$;
\item $|F'(0)| \leq \frac{1}{2} (p^u \wedge p^s)^{\alpha/3}$;
\item $F$ is $C^{1 +\alpha/3}$ and $\sup|F'| + \sup_{a,b}\frac{|F(a)-F(b)|}{|a-b|^{\alpha/3}} \leq \frac{1}{2}$. 
\end{enumerate}
A \textit{$u$-admissible manifold} in $v$ is a set of the form $\Psi_x(\{(F(t), t) \, |\, |t|\leq p^u\})$, where $F$ satisfies the above three conditions. 
\end{defn}
The $s$- and $u$-admissible manifolds in $v$  are  subsets of $\Psi([-Q_{\epsilon}(x), Q_{\epsilon}(x)]^2)$. 
For two $s$-admissible manifolds $V_1 = \Psi_x(\{(t, F_1(t))\, |\, |t|\leq p^s\})$ and $V_2 = \Psi_x(\{(t, F_2(t)) \, |\, |t|\leq p^s\})$, define  
$$\dist(V_1,V_2):= \max |F_1 - F_2|.$$ 
Analogously define $\dist(V_1,V_2)$ 
for a pair of $u$-admissible manifolds. Note that, since $\|C_\chi(x)
\| \leq 1$ for all $x\in \NUH_{\chi}(f)$, there is a constant $C_{
\Lambda}\geq 1$ such that for any $\epsilon$-double chart $v$ and any pair $(V_1, V_2)$ of  $s$-admissible manifolds (or $u$-admissible manifolds) in $v$,
\begin{align}\label{C_Lambda}
\dist_{\Lambda}(V_1, V_2) 
\leq C_{\Lambda}\dist(V_1,V_2) 
.
\end{align}
If $\epsilon$ is small enough and $v_i \overset{\epsilon}{\rightarrow}v_{i+1}$, then the image of a $u$-admissible manifold $V^u$ in $v_i$ contains a unique $u$-admissible manifold in $v_{i+1}$, the \textit{graph transform} of $V^u$,  denoted by $\mathcal{F}_u[V^u]$, see \cite[p.15]{LimaSarig2019}. 
By iteration we obtain  a  $u$-admissible manifold $\mathcal{F}_u^{j-i}[V_u]$, associated to a chain $v_i \overset{\epsilon}{\rightarrow} v_{i+1} \overset{\epsilon}{\rightarrow} \cdots \overset{\epsilon}{\rightarrow} v_j$.
Similarly, the preimage of a $s$-admissible manifold $V^s$ in $v_{i+1}$ contains a unique $s$-admissible manifold in $v_{i}$, denoted by $\mathcal{F}_s[V^s]$, and in the same way we can define iterations $\mathcal{F}^k_s[V^s]$, $k>1$. 
We also recall the following proposition. 
\begin{prop}\label{prop:contraction}\cite[Prop.4.14]{Sarig_diffeo2013}
If $\epsilon$ is sufficiently small and $v_i \overset{\epsilon}{\rightarrow} v_{i+1}$, then for any $u$-admissible manifolds $V^u_-$ and $V^u_+$ in $v_i$, 
$$\dist(\mathcal{F}_u[V^u_-], \mathcal{F}_u[V^u_+]) \leq e^{-\chi/2}\dist(V^u_-, V^u_+),$$
and,  for any $s$-admissible manifolds $V^s_-$ and $V^s_+$ in $v_{i+1}$, 
$$\dist(\mathcal{F}_s[V^s_-], \mathcal{F}_s[V^s_+]) \leq e^{-\chi/2}\dist(V^s_-, V^s_+).$$
\end{prop}
The \textit{unstable manifold} of a negative half-gpo  $\underline{v}=(v_i)_{i\leq 0}$ is defined by  $$V^u[\underline{v}] := \lim_{n\to \infty}\mathcal{F}^n_u[V^u_{-n}],$$ 
where $V^u_{-n}$, $n\in\N$, is any sequence of $u$-admissible manifolds in $v_{-n}$, $n\in\N$,  and where $\mathcal{F}^n_u$ are defined with respect to $v_{-n}\overset{\epsilon}{\rightarrow} v_{-n+1} \overset{\epsilon}{\rightarrow} \cdots \overset{\epsilon}{\rightarrow}v_0$. The definition of $V^u[\underline{v}]$ is independent of the chosen sequence $V^u _{-n}$.  
Analogously one defines the stable manifold $V^s[\underline{v}]$ of a positive half-gpo $\underline{v} = (v_i)_{i\geq 0}$, cf.\  \cite[p.15]{LimaSarig2019}. We will sometimes write $V^u[\underline{v}]$ and $V^s[\underline{v}]$ for gpos $\underline{v}$, although the definition does not depend on the positive resp.\ negative part. Furthermore, the following holds.
\begin{enumerate}
\item $V^s[(v_i)_{i\geq 0}]$, 
 $V^u[(v_i)_{i\leq 0}]$  are $s$- resp. $u$- admissible manifolds in $v_0$; 
\item  $f(V^s[(v_i)_{i\geq 0}]) \subset V^s[(v_i)_{i\geq 1}]$, $f^{-1}(V^u[(v_i)_{i\leq 0}]) \subset V^u[(v_i)_{i\leq -1}]$;
\item if $x,y \in V^s[\underline{v}]$,  then $\dist_{\Lambda}(f^{n}(x), f^{n}(y)) \rightarrow 0$ for $n\to \infty$, and if $x,y \in V^u[\underline{v}]$,  then $\dist_{\Lambda}(f^{-n}(x), f^{-n}(y)) \rightarrow 0$ for $n\to \infty$. The rates of convergence are exponential.
\end{enumerate}

The following Lemma follows directly from  Prop. 4.11 in \cite{Sarig_diffeo2013}. 
\begin{lem}\label{lem:admissible_square}
Let $v= \Psi^{p^u,p^s}_x$ be a $\epsilon$-double chart, let $V^u_- \neq V^u_+$  two $u$-admissible manifolds  and $V^s_- \neq V^s_+$ be two $s$-admissible manifolds in $v$. Then, if $\epsilon$ is small enough, for  any $a,b \in \{-,+\}$, the sets $V^s_a$ and $V^u_b$ intersect in a unique point $x_{ab}$ in $\Psi_{x}([-10^{-2} p^u \wedge p^s, 10^{-2} p^u \wedge p^s]^2)$. Hence they define a rectangle $\mathcal{Q}  = \mathcal{Q}(V^u_-, V^u_+; V^s_-, V^s_+)$ bounded by those  segments of $V^u_-, V^u_+, V^s_-, V^s_+$  that connect two of the vertices $x_{ab}$, $a,b \in \{-,+\}$. The rectangle $\mathcal{Q}$ is contained in $\Psi_{x}([-10^{-1} p^u \wedge p^s, 10^{-1} p^u \wedge p^s]^2)$. 
\end{lem}

We say that a rectangle $\mathcal{Q}$ of the form $\mathcal{Q}= \mathcal{Q}(V_-^s,V_+^s;V_-^u, V_+^u)$ as in Lemma~\ref{lem:admissible_square}  is an  \textit{admissible rectangle in $v$} and  call the segments of $V^s_{\pm}$ resp.\  $V^u_{\pm}$ that connect  
 corners of the rectangle   $\mathcal{Q}$ its  \textit{$s$-sides} resp.\ \textit{$u$-sides}.

Below we will also  consider the following  \textit{partial order} relations $\preceq_s$ and $\preceq_u$  on the families of $s$- resp.\ $u$-admissible manifolds in a double chart $v$. 
Given $s$-admissible manifolds  $V_1 = \Psi_x (\{(t, F_1(t)) \, |\, |t|\leq p^s\})$ and $V_2 = \Psi_x(\{(t, F_2(t)) \, |\, |t|\leq p^s\})$, we say that $V_1 \preceq_s V_2$ if either  $F_1(t) = F_2(t)$ for all $t$ with $|t| \leq p^s$, or $F_1(t) < F_2(t)$  for all $t$ with $|t| \leq p^s$. 
Similarly,  we define $\preceq_u$ on the family of $u$-admissible manifolds in a double chart $v$. 

Finally, we recall the shadowing lemma for  generalized pseudo-orbits. Let $x\in \Lambda$. A gpo $(\Psi^{p_i^u, p_i^s}_{x_i})_{i\in \Z}$ is said to \textit{shadow the orbit of $x$}, if $f^i(x) \in \Psi_{x_i}([-\eta_i,\eta_i]^2)$, where $\eta_i := p_i^u \wedge p_i^s$. 
\begin{thm}\label{thm:shadowing}\cite[Theorem 4.2]{LimaSarig2019}
For $\epsilon$ small enough, every $\epsilon$-gpo  shadows a unique orbit. 
\end{thm}
In fact,  $f^i(x) \in \Psi_{x_i}([-10^{-2}\eta_i,10^{-2} \eta_i])$, for the shadowed orbit of a  gpo $(\Psi^{p_i^u, p_i^s}_{x_i})_{i\in \Z}$ in Theorem \ref{thm:shadowing}. 
\subsection{The Markov partition}\label{sec:Markovpartition} 
A crucial step in the construction of the topological Markov shift in \cite{Sarig_diffeo2013} and \cite{LimaSarig2019} is to find a suitable countable set $\mathcal{A}$ of double charts for which the set of gpos are still "sufficiently rich", in particular for which  $\mu_{\Lambda}$-a.e.\ point in $\Lambda$ is shadowed by a gpo in $\mathcal{A}$. We refer to \cite{LimaSarig2019} for the details on the construction and properties of $\mathcal{A}$, and recall here only some consequences which are  relevant for our  application. Let $\epsilon$ be small enough that the results of the previous section hold. Denote by  $\mathcal{G}$ the directed graph with the set of vertices $\mathcal{A}$ and the set of edges $\{(v,w) \in \mathcal{A} \times \mathcal{A}: v\overset{\epsilon}{\rightarrow} w\}$. Denote as above by  
$\Sigma(\mathcal{G}) := \{ \underline{v} \in \mathcal{A}^{\Z} \, |\, v_i \overset{\epsilon}{\rightarrow} v_{i+1} \text{ for all } i \in \Z\}$ the set of infinite  paths on $\mathcal{G}$, and by 
$\sigma: \Sigma(\mathcal{G}) \to \Sigma(\mathcal{G})$ the left shift. 
Let $\pi:\Sigma(\mathcal{G}) \to \Lambda$ be defined by 
$$\pi(\underline{v}) := \text{ unique point whose }f\text{-orbit is shadowed by }\underline{v}.$$
Then, the following hold, cf.\ \cite[p.\ 23]{LimaSarig2019}: 
\begin{enumerate}
\item $\pi$ is H\"older continuous;
\item $f\circ \pi = \pi \circ \sigma$;
\item  $\mu_{\Lambda}(\Lambda \setminus \pi [\Sigma^\#(\mathcal{G})]) = 0$;
\item for all $x \in \Lambda, i \in \Z, \, \#\{v_i \, |\, \underline{v} \in \Sigma^{\#}(\mathcal{G}), \pi(\underline{v}) = s\} < \infty$.
\end{enumerate}

\begin{rem}\label{rem:smanifold_welldefined}
If $\underline{v}$ and $\underline{u}$ are two gpos with $v_0 = u_0$ and  $\pi(\underline{v}) = \pi(\underline{u})$, then 
$V^u[\underline{v}] = V^u[\underline{u}]$, and $V^s[\underline{v}] = V^s[\underline{u}]$, cf.\ \cite[Prop. 6.4]{Sarig_diffeo2013}. Hence, if the entry $v_0$ of $\underline{v}$ is clear from the context, we also denote the above manifolds by   
$V^s(\pi(\underline{v}))$ resp.\  $V^u(\pi(\underline{v}))$.
For two gpos $\underline{v}$ and $\underline{u}$ with $v_0 = u_0=v$, 
we write $\mathcal{Q}(\pi(\underline{v}); \pi(\underline{u}))$ for the admissible rectangle $\mathcal{Q}(V^u[\underline{v}], V^u[\underline{u}];V^s[\underline{v}], V^s[\underline{u}])$. 
\end{rem}
The next step is to consider the covering of $\pi[\Sigma^{\#}(\mathcal{G})]$, given by 
$$\mathcal{Z} := \{ Z(v) \, |\, v\in \mathcal{A}\}, \text{ where } Z(v) := \{ \pi(\underline{v}) \, |\, \underline{v} \in \Sigma^\#(\mathcal{G}), v_0 = v\}.$$
It is shown that for all $Z \in \mathcal{Z}$ , $\#\{ Z'\cap Z \neq \emptyset\}< \infty$.

For $x \in Z$,  the \textit{$s$-fibre of $x$ in $Z$} is defined  by 
$$W^s(x,Z) = V^s(x,Z) \cap Z,$$
where $V^s(x,Z) := V^s[(v_t)_{t\in \Z}]$ for some (and hence any) $\underline{v} \in \Sigma^\#(\mathcal{G})$ such that  $v_0 = v$ and $\pi(\underline{v}) = x$. 
Analogously the \textit{$u$-fibre $W^u(x,Z)$ of $x$ in $Z$} is defined. 

The Markov partition is defined as a refinement of $\mathcal{Z}$, similar to a construction due to Bowen and Sinai (cf.\  \cite{Bowen-book}): 
Enumerate $\mathcal{Z} = \{Z_i\, |\, i \in \N\}$ and define for each $Z_i, Z_j \in \mathcal{Z}$ with $Z_i \cap Z_j \neq \emptyset$, 
\begin{align*}
T^{us}_{ij} &:= \{ x \in Z_i \, :\, W^u(x,Z_i) \cap Z_j \neq \emptyset, \, W^s(x,Z_i) \cap Z_j \neq \emptyset \}, \\
T^{u\emptyset}_{ij} &:= \{ x \in Z_i \, :\, W^u(x,Z_i) \cap Z_j \neq \emptyset, \, W^s(x,Z_i) \cap Z_j = \emptyset \}, \\
T^{\emptyset s}_{ij} &:= \{ x \in Z_i \, :\, W^u(x,Z_i) \cap Z_j = \emptyset, \, W^s(x,Z_i) \cap Z_j \neq \emptyset \}, \\
T_{ij}^{\emptyset \emptyset} &:= \{ x \in Z_i \, :\, W^u(x,Z_i) \cap Z_j = \emptyset, \, W^s(x,Z_i) \cap Z_j = \emptyset \}.
\end{align*}
Let $\mathcal{T} := \{ T_{ij}^{\alpha \beta} \, |\, i,j \in \N, \alpha \in \{u, \emptyset\}, \beta \in \{s, \emptyset\}\}$. 
The partition $\mathcal{R}$ is the collection of sets of the form 
$R(x) := \bigcap \{ T \in \mathcal{T} \, |\, x\in T \}$ for some $x \in \bigcup_{i\geq 1} Z_i$. 

\begin{prop}\cite[Prop. 5.2]{LimaSarig2019}\label{zlocfinite}
The collection of sets $\mathcal{R}$ form a countable pairwise disjoint cover of $\pi[\Sigma^{\#}(\mathcal{G})]$.
It refines $\mathcal{Z}$, and each element of $\mathcal{Z}$ contains only finitely many elements in $\mathcal{R}$.
\end{prop}
The family $\mathcal{R}$ forms a {Markov partition} in the sense of Sinai, see \cite[Prop.~5.3]{LimaSarig2019}.
For $R, S \in \mathcal{R}$ we write $R \rightarrow S$, whenever there is $x \in \R$ such that  $f(x) \in S$. 
Let $\widehat{\mathcal{G}}$ be the graph with $\mathcal{R}$ as vertices and $\{(v,w) \in \mathcal{R}\times \mathcal{R} \, |\, v\rightarrow w\}$ as edges. 
Let  $\Sigma(\widehat{\mathcal{G}}) = \{ \underline{R} \in \mathcal{R}^{\Z}\, |\, R_i \rightarrow R_{i+1}, \, \forall i \in \Z\}$  be the set of infinite paths in $\widehat{\mathcal{G}}$, and 
$\Sigma^{\#}(\widehat{\mathcal{G}}) = \{ \underline{R} \in \Sigma(\widehat{\mathcal{G}}) \,  |\, (R_i)_{i\leq 0}$ and $(R_i)_{i\geq 0} \text{ contain constant subsequences} \}$.

For a finite path $R_m \rightarrow \cdots \rightarrow R_n$, $n,m \in \N$,  $n\geq m$, and any $l\in \N$, the sets
\begin{align*}
{}_l[R_m, \ldots, R_n]:= f^{-l}(R_m) \cap f^{-l-1}(R_{m+1}) \cap \cdots \cap f^{-l -(n-m)}(R_n)
\end{align*}
are nonempty, and it holds that for $\underline{R} \in \Sigma(\widehat{\mathcal{G}})$, there is $c>0$ and $\theta \in (0,1)$ such that
\begin{align}\label{eq:leq:theta}
\diam_{\Lambda}({}_{-t}[R_{-t}, \ldots, R_t]) \leq c \theta^t, \text{ for all  }t \in \N,
\end{align}
see \cite[p. 26]{LimaSarig2019}. 
Hence the map  $\widehat{\pi}: \Sigma(\widehat{\mathcal{G}}) \to \Lambda$, 
$$\widehat{\pi}(\underline{R}) :=  \text{ unique point in }  \bigcap_{t=0}^{\infty} \overline{{}_{-t}[R_{-t}, \ldots, R_t]}$$
is well defined. 

We mention further properties of the  sequences in $\Sigma(\widehat{\mathcal{G}})$. We write for a path $v_m \rightarrow \cdots \rightarrow v_n$, $l\in \Z$, $$Z_{l}(v_{m},\ldots,v_{n}) := \{\pi(\underline{u}) \, |\, u_i=v_i \text{ for } i=l,\ldots, l+n-m\}.$$
The next Lemma is a slightly stronger formulation of the statement of Lemma 5.4 in \cite{LimaSarig2019}. For a proof see \cite[Lemma 12.2]{Sarig_diffeo2013}.
\begin{lem}\label{lem:RinZ}
For every $\underline{R} \in \Sigma(\widehat{\mathcal{G}})$ there is a gpo $\underline{v} = (v_t)_{t\in\Z}$ in $\Sigma^{\#}(\mathcal{G})$ such that for all $t\in \Z$, $R_t \subset Z(v_t)$. Moreover, for any such gpo $
\underline{v}$ it holds that for all $t\in\Z$, ${}_{-t}[R_{-t},\ldots,R_{t}] 
\subset Z_{-t}(v_{-t},\ldots,v_{t})$.
\end{lem}
\begin{lem}\label{lem:uniqueness}
Let $\underline{R} = (R_t)_{t\in \Z}$, ${\underline{S}} = ({S}_t)_{t\in \Z}  \in \Sigma^{\#}(\widehat{\mathcal{G}})$, and $\underline{v} = (v_t)_{t\in \Z}$, $\underline{w} = (w_t)_{t\in \Z}  \in \Sigma^{\#}(\mathcal{G})$  with $R_t \subset Z(v_t)$,  and
$S_t \subset Z(w_t)$ for all $t\in \Z$.
Assume that there is $k,l >0$ such that   
\begin{align*}
R_{-k} = S_{-k},\,R_{l}  = S_{l},\, v_{-k} = w_{-k},\, v_{l} = w_{l}.
\end{align*} 
If $R_0 \neq S_0$, then $\widehat{\pi}(\underline{R}) \neq \widehat{\pi}(\underline{S})$.
\end{lem}
\begin{proof}
The arguments for the proof can be found in the proof of Theorem 5.6 (5) in \cite{LimaSarig2019}, for the convenience of the reader we give a sketch here. 
First, given any two points $x_R\in {}_{-k}[R_{-k},\ldots, R_l]$ and $x_S \in{}_{-k}[S_{-k},\ldots, S_l]$, one considers the points 
$z_R\in Z(v_0)$ and $z_S\in Z(w_0)$ defined by the conditions  \begin{align*} \{f^{-k}(z_R)\} &= W^u(f^{-k}(x_S),Z(v_{-k})) \cap W^s(f^{-k}(x_R),Z(v_{-k})), \text{ and} \\ \{f^{l}(z_S)\} &= W^u(f^{l}(x_S),Z(v_{l})) \cap W^s(f^{l}(x_R),Z(v_{l})).\end{align*} One can then deduce from the properties of $\mathcal{R}$, see \cite{LimaSarig2019}, that in fact $z_R \in {}_{-k}[R_{-k},\ldots, R_l]$ and  $z_S \in {}_{-k}[S_{-k},\ldots, S_l]$. Since $R_0 \neq S_0$, this means in particular $z_R \neq z_S$. The points $z_R$ and $z_S$ are shadowed by gpos $\underline{\alpha}$ and $ \underline{\beta}$ in $\Sigma^{\#}(\mathcal{G})$ that satisfy
 $\alpha_{-k} = v_{-k} = w_{-k}$ and $\beta_l = v_l = w_l$.

Write $x:=\widehat{\pi}(\underline{R})$ and $y:=\widehat{\pi}(\underline{S})$. 
By Lemma \ref{lem:RinZ}, for every $n\in\N$, $x\in \overline{{}_{-n}[R_{-n}, \ldots, R_n]} \subset  \overline{Z_{-n}(v_{-n},\dots,v_n)}$. Since the  diameters of the sets $Z_{-n}(v_{-n},\dots,v_n)$ tend to zero as $n\to \infty$, it follows that $x= \pi(\underline{v})$, and hence in particular $x\in
Z_{-k}(v_{-k}, \ldots, v_l)$. Analogously, $y=\pi(\underline{w})\in
Z_{-k}(w_{-k}, \ldots, w_l)$.
This means also that $f^i(x) \in Z(v_i)$, $f^i(y) \in Z(w_i)$ for $-k\leq i \leq l$.
Therefore, if we assume by contradiction that $x=y$, then $Z(v_i)$ and $Z(w_i)$ intersect for $-k\leq i \leq l$. It is shown (intersecting  charts property in \cite{LimaSarig2019}) that then 
$Z(v_i) \cup Z(w_i) \in \Psi_{x_i}([-Q_{\epsilon}(x_i),Q_{\epsilon}(x_i)]^2)$, for $-k\leq i \leq l$, where $v_i :=\Psi_{x_i}^{p_i^u,p_i^s}$, $-k\leq i \leq l$. 
Define $\underline{c}$ by $c_i := \alpha_i$, if $i\leq -k$, $c_i := v_i$ if $-k\leq i\leq l$, and $c_i := \beta_i$ if $l\leq i$. We write $c_i=\Psi^{p_i^u,p^i_s}_{x_i}$.
By the construction of $z_R, z_S$, one can additionally show that $f^i(z_R), f^(z_S) \in \Psi_{x_i}([-Q_{\epsilon}(x_i),Q_{\epsilon}(x_i)]^2)$. It follows then from the proof of the Shadowing Theorem \ref{thm:shadowing} that $z_R$ and $z_S$ are both shadowed by $\underline{c}$.  It follows that $z_R=z_S$ which yields a contradiction.
\end{proof}

\begin{thm}\label{thm:hatpi}\cite[Theorem 5.6]{LimaSarig2019}
\begin{enumerate}
\item $\widehat{\pi} \circ T = f \circ \widehat{\pi}$,
\item $\widehat{\pi}:\Sigma(\widehat{\mathcal{G}}) \to \Lambda$ is H\"older continuous,
\item $\widehat{\pi}(\Sigma^{\#}(\widehat{\mathcal{G}}))$ has full $\mu_{\Lambda}$-measure. 
\end{enumerate}
\end{thm}
The topological Markov flow $\sigma_r:  \Sigma_r \to \Sigma_r$ in Theorem \ref{thm:fullmeasurecoding} is obtained from the topological Markov shift $(\Sigma(\widehat{\mathcal{G}}), \sigma)$ and roof function $r:\Sigma(\widehat{\mathcal{G}}) \to (0, \infty)$ defined by 
$$r(\underline{R}) := R_{\Lambda} (\widehat{\pi}(\underline{R})),$$
where $R_{\Lambda}: \Lambda \to (0, \infty)$ is the roof function on $\Lambda$ for $\varphi$. The mapping  
${\pi}_r:\Sigma_r(\widehat{\mathcal{G}}) \to M$ in Theorem \ref{thm:fullmeasurecoding} is given by 
\begin{align*}
{\pi}_r(\underline{x},t) := \varphi^t[\widehat{\pi}(\underline{x})].
\end{align*}

\section{Construction of the horseshoes}\label{sec:construction_horseshoes}
In this section we will show the existence of horseshoes as defined in the introduction and prove Theorem \ref{thm:main_coding}. We will derive Theorems \ref{thm:link_growth} and \ref{thm:link_growth_rel} as a consequence. We first discuss an approximation result for countable Markov shifts, and then return to the setting of \cite{LimaSarig2019}, which was outlined in Section~\ref{sec:Lima_Sarig_thm}.  
\subsection{An approximation  result for 
 countable Markov shifts}\label{sec:TMS}
Suppose \linebreak $(\Sigma= \Sigma(\mathcal{G}), T)$ is  a two-sided countable topological Markov shift (TMS)  with left shift $T:\Sigma \to \Sigma$, 
see Section \ref{sec:main_LimaSarig}. 
Denote by $\mathcal{A}$ the set of vertices of $\mathcal{G}$, and write $a\to b$, $a,b\in \mathcal{A}$, if there is a vertex in $\mathcal{G}$ from $a$ to $b$. 
A \textit{finite path} from a vertex $a\in \mathcal{A}$ to a vertex $b\in \mathcal{A}$ of \textit{length $n$} is an ordered tuple $\gamma = a_0a_1 \cdots a_{n}$  of elements in $\mathcal{A}$  with  $a_0 = a$, $a_n = b$, and  $a_i \rightarrow a_{i+1}$, for all $i\in\{0, \ldots, n-1\}$. We say that a path $\gamma = a_0a_1 \cdots a_{n}$ \textit{passes through} $c$ (at position $i$) if $a_i = c$ for some $i\in \{1, \ldots, n-1\}$. 
A path from $a$ to $a$ is called a \textit{loop based at }$a$. 
If $\gamma = a_0 \cdots a_l$ is a path from $a$ to $b$ of length $l$ and $\gamma'= a'_0 \cdots a'_{l'}$ is a path from $b$ to $c$ of length $l'$, we can compose them and get a path $\gamma\gamma'= a_0 \cdots a_{l-1} a_l a'_1 \cdots a'_{l'}$ from $a$ to $c$ of length $l+l'$. We define in the obvious way the composition of an ordered tuple of paths whenever neighbouring paths share endpoints. 
If $\gamma = a_0\cdots a_{n}$ is a loop,  we denote by $\overline{\gamma}$ the infinite path obtained by periodically repeating $\gamma$, i.e.,  $\overline{\gamma} := \underline{x} = (x_i)_{i\in \Z}$, where $x_{i+kn} = a_i$ for all $k\in \Z, i\in \{0, \ldots, n-1\}$.  
For any $\underline{x} = (x_i)_{i\in \Z}$ in $\Sigma(\mathcal{G})$, and $s,t  \in \Z$ with $s<t$,  we denote by  $\underline{x}|_{[s,t]}$ the finite path $x_s x_{s+1} \cdots x_t$ from $x_s$ to $x_t$.

The TMS $(\Sigma(\mathcal{G}),T)$ is \textit{topological transitive} if for any $a, b\in \mathcal{A}$ there exists a path from $a$ to $b$; it is \textit{topological mixing} if for any $a,b \in \mathcal{A}$ there is $n_0 \in \N$ such that for all $n\geq n_0$ there exists a path from $a$ to $b$ of length $n$. 
The following spectral decomposition theorem holds for topological transitive TMS, see for example \cite[Section 7.2]{Kitchens}. 
\begin{thm}\label{thm:SDT}
Let $(\Sigma(\mathcal{G}),T)$ be a topologically transitive TMS. There is $p\in \N$ and a disjoint partition 
$\Sigma(\mathcal{G}) = X_0 \cup X_1 \cup \cdots \cup X_{p-1}$ such that $T(X_i) = X_{i+1(\mathrm{mod}\, p)}$ and such that, for all $i\in \{0, \ldots, p-1\}$,   $T^p|_{X_i}$ is topologically conjugated to a topologically mixing TMS.
\end{thm}
In the following we consider a H\"older continuous function $\phi: \Sigma \to \R$, bounded away from $0$ and $\infty$.   
Define $\phi_n: \Sigma \to \R$ by  $$\phi_n(x) := \sum_{k=0}^{n-1} \phi(T^k(x)).$$
Denote by 
$P_{\mu}(\phi) := h_{\mu}(T) + \int \phi d\mu$ the \textit{pressure} with respect to a shift invariant probability measure $\mu$. 
In \cite{Sarig_thermo1999}, Sarig introduced the notion of Gurevich pressure of $\phi$ and proved a variational principle for it. The statement of the latter includes that the  Gurevich pressure is bounded from below by $P_{\mu}(\phi)$. We will need a variant of that statement.  
\begin{lem}\label{lem:loops}
Let $(\Sigma ,T)$ be a topological transitive TMS, $\mu$ a $T$-invariant ergodic  probability measure on $\Sigma$, and $\phi: \Sigma \to \R$ a H\"older continuous function, bounded away from $0$ and $\infty$. 
Then, for any $\epsilon$, $0<\epsilon< h_{\mu}(T)$, there is a vertex $a\in \mathcal{A}$ and an infinite set $\mathcal{M} \subset \N$ such that for all $m\in \mathcal{M}$  there are loops $\gamma_1, \ldots, \gamma_k$, $k\in \N$,  of length $m$, based at $a$, such that,  
\begin{enumerate}[(i)]
\item $k\geq e^{m(h_{\mu}(T)-\epsilon)}$;
\item  for all $l\in \N$ and $(i_1,\ldots, i_l) \in \{1,\ldots, k\}^l$,
$$|\phi_{lm}(\overline{\gamma_{i_1}\gamma_{i_2}\cdots \gamma_{i_l}}) - lm \int \phi d\mu| \leq lm\epsilon.$$
\end{enumerate}
In particular,  
$$\frac{1}{m} \log \sum_{i=1}^k e^{\phi_m(\overline{\gamma_i})} > P_{\mu}(\phi) - 2\epsilon.$$ 
\end{lem}
\begin{proof}
Our proof uses modifications of arguments in \cite{Sarig_thermo1999} and \cite{BuzziSarig}.
It is sufficient to show the lemma for one-sided shifts, cf.\ \cite[Theorem $3.1$ 
 and \S $6$]{DaonYair2013}.
Hence, consider 
$$\Sigma = \Sigma^+(\mathcal{G}) := \{\underline{x} = (x_i)_{i\in \N_0}\, |\, x_i \rightarrow x_{i+1} \text{ for all } i \in \N_0\} $$ with left shift $T:\Sigma \to \Sigma,\, (x_i)_{i\in \N_0} \mapsto  (x_{i+1})_{i\in \N_0}$.  

We first assume that $T$ is topologically mixing.
For $\epsilon_0 >0$, $n_0 \in \N$, let 
$$S_{\epsilon_0,n_0} := \left\{ x\in \Sigma  \mid \forall n\geq n_0 : \left|\phi_n(x) - n\left(\int \phi d\mu\right)\right| \leq n\epsilon_0\right\}.$$ 
Let $\epsilon>0$. By the ergodic theorem and Egorov's theorem we can choose $n_0 \in \N$ such that 
$\mu(S_{\epsilon/4, n_0}) \geq 1/2$. 
We write $S:=S_{\epsilon/4, n_0}$.

Choose a denumeration $\{1, 2, \ldots \}$ of $\mathcal{A}$ and consider, for $l\in \N$, the finite partition $\alpha_l = \{[1], \ldots, [l], [> l]\}$ in cylinder sets, where $[j] := \{\underline{x} \in \Sigma \, |\, x_0 = j\}$ and  $[> l] =\bigcup_{j=l+1}^{\infty}[j]$.
We have that $h_{\mu}(T,\alpha_l) \to h_{\mu}(T)$ for $l \to +\infty$, and hence for $l$ sufficiently large,   $$h_{\mu}(T,\alpha_l)\geq h_{\mu}(T) -\frac{\epsilon}{4}.$$
For $\beta = \alpha_l$, denote by $\beta_0^{n}$ the collection of cylinder sets of length $n+1$ with respect to this partition, that is, the collection of sets of the form  $[\underline{b}] = [b_0, \ldots, b_n] =\{\underline{x} \in \bigcap^n_{j=0} T^{-j}b_j \, |\, b_j \in \beta\}$. Let  $K^{n}_l := \beta_0^{n} \cap  \bigcup_{a,b\neq[> l]} \{a \cap T^{-n}b \}$. Consider the set $\hat{K}^n_l$ of those cylinder sets of $K^n_l$ that intersect $S$. 
We fix $l$ to be so large that the union of the cylinder sets of $\hat{K}_l^{n_0}$ (and hence also of those of  $\hat{K}_l^{n}$ for $ n\geq n_0$) cover a set of measure larger  than $\frac{1}{4}$. Since $\mu$ is ergodic, this implies by a formula for entropy in \cite[Ch.\ 5]{Rudolph} that 
$$h_{\mu}(T,\alpha_l)\leq \liminf_{n\to \infty} \frac{1}{n} \log \# \hat{K}_l^n.$$
Hence for $n\geq n_0$ sufficiently large,  
$$\#\hat{K}_l^n \geq e^{n(h_{\mu}(T) - \epsilon/2)}.$$
By the H\"older continuity of $\phi$ we can choose $n$ 
additionally so large that for any $[\underline{b}] \in \beta_{0}^n$ and any $x,y\in [\underline{b}]$, 
\begin{align}\label{eq:phi_n}
|\phi_n(x) -\phi_n(y) |< \frac{n\epsilon}{4}.
\end{align}
Hence, if $[\underline{b}] \in \hat{K}_l^n$, then for all $x\in [\underline{b}]$, 
\begin{align}\label{eq:Knl}
|\phi_n(x) -n\int \phi d\mu| \leq \frac{ n\epsilon}{2}.
\end{align}
Since $T$ is topological mixing, there is $s_0 = s_0(l)$ such that for any $s\geq s_0$, there is a path of length $s$ from any vertex $b\in \{1, \ldots, l\}$ to any vertex $a \in \{1, \ldots, l\}$.
We choose $n$ additionally so large that 
\begin{align}\label{eq:ns_0}
\frac{n\epsilon}{4} \geq s_0\max\{h_\mu(T), |\max_{x\in \Sigma} \phi(x) - \min_{x\in \Sigma} \phi(x)|\},
\end{align}
and such that 
\begin{align*}
e^{\frac{n\epsilon}{4}}\geq l.
\end{align*}
We get a collection $Y_{n+s_0}$ of pairwise distinct loops of length $n+s_0$, based at vertices in $\{1, \ldots, l\}$, with 
$$\# Y_{n+s_0} \geq e^{n(h_{\mu}(T) - \epsilon/2)} \geq e^{(n+s_0) (h_{\mu}(T) - 3\epsilon/4)},$$
and hence there is $a\in \{1,\ldots, l\}$ and a collection of pairwise distinct 
 loops $\gamma_1, \ldots, \gamma_k$ of length $n+s_0$ based at $a$  with 
\begin{align}\label{eq:for1}
k \geq \frac{1}{l} e^{(n+s_0)(h_{\mu}(T) - 3\epsilon/4)}\geq e^{(n+s_0)(h_{\mu}(T) - \epsilon)}.
\end{align}

By \eqref{eq:Knl}, 
\begin{align}\label{eq:for2}
\begin{split}
&|\phi_{n+s_0}(\overline{\gamma_i}) - (n+s_0) \int \phi d\mu| \\ &\leq |\phi_n(\overline{\gamma_i}) - n\int\phi d\mu| + s_0|\max_{x\in \Sigma} \phi(x) - \min_{x\in \Sigma} \phi(x)| \leq \frac{3n\epsilon}{4}.
\end{split}
\end{align}
Note that there is $a\in \{1, \ldots, l\}$, such that \eqref{eq:for1} and \eqref{eq:for2} hold for infinitely many $n$. This implies that for infinitely many $m=n+s_0$, assertion (i) hold, as well as assertion (ii), for the case $l=1$. 

Let $m=n+s_0$ be as above, and $\gamma_1, \ldots, \gamma_k, k\in \N$, a collection of loops of length $m$ and based at $a$ as above. Let $l\in \N$ and $(i_1, \ldots, i_l) \in \{1, \ldots, k\}^l$.
Then, by   \eqref{eq:phi_n} and \eqref{eq:for2}, 
\begin{align*}
|\phi_{lm}&(\overline{\gamma_{i_1}\cdots\gamma_{i_l}}) - lm\int \phi d\mu | \\ &=|\phi_m(\overline{\gamma_{i_1}\cdots\gamma_{i_l}}) + \sum_{j=1}^{l-1}(\phi_{(j+1)m}-\phi_{jm})(\overline{\gamma_{i_1}\cdots\gamma_{i_l}}) -lm\int \phi d\mu| \\ &= 
|\sum_{j=0}^{l-1} \phi_m(T^{jm}(\overline{\gamma_{i_1}\cdots\gamma_{i_{l}}})) -lm\int \phi d\mu| \\ &\leq |\sum_{j=1}^{l} \phi_m(\overline{\gamma_{i_{j}}}) -lm\int \phi d\mu| + \frac{lm \epsilon}{4} \leq  lm \epsilon.
\end{align*}

This finishes the proof in the case that $T$ is topologically mixing. Consider now the more general case of a topologically transitive TMS $(\Sigma,T)$, and choose $p\in \N$ from Theorem \ref{thm:SDT}. Let $\epsilon'>0$. By the result above for topologically mixing TMS and with $\epsilon = p\epsilon'>0$, there is $a\in \mathcal{A}$, an infinite set $\mathcal{M} \subset \N$ such that for all $m\in \mathcal{M}$,  there are $k\geq e^{m(h_{\mu}(T^p) -p\epsilon')}$ many pairwise distinct loops $\gamma_1, \ldots, \gamma_k$ of length $pm$,  based at  $a$,  such that,  for $l\in \N$ and $(i_1, \ldots, i_l)\in \{1, \ldots, k\}^l$,
$$|(\phi_{p})_{lm}(\overline{\gamma_{i_1}\gamma_{i_2}\cdots \gamma_{i_l}}) - lm\int\phi_p d\mu| \leq lm(p\epsilon').$$
Since $h_{\mu}(T^p) = ph_{\mu}(T)$ and  $\int \phi_p d\mu = p\int \phi d\mu$, the assertions of the lemma hold for $(T,\Sigma)$. 
\end{proof}

\subsection{Separation of orbits}\label{sec:separated}
We return to the setting in Section \ref{sec:Lima_Sarig_thm}, and let $\varphi:M \to M$ be a $C^{1+\alpha}$ flow induced by a non-vanishing vector field on a closed smooth $3$-maniold $M$.
Let $h= h_{\topo}(\varphi)$, and assume that $h>0$.  
We choose an ergodic probability measure of maximal entropy for $\varphi$, and apply 
 Theorem \ref{thm:fullmeasurecoding} to it. We obtain a topological Markov flow that satisfies the properties in Theorem \ref{thm:fullmeasurecoding}. We moreover fix the choices that were made in the construction of $(\Sigma_r,\sigma_r)$ as described in Section \ref{sec:Lima_Sarig_thm}.   
As explained in \cite{LimaSarig2019}, $\mu$ induces a probability measure of maximal entropy for $(\Sigma_r, \sigma_r)$. By the ergodic decomposition theorem,  $\sigma_r$ has an ergodic probability measure of maximal entropy.
This induces an ergodic shift invariant  probability measure $\mu_{\Sigma}$ on $\Sigma  = \Sigma(\widehat{\mathcal{G}})$, see \cite[p.31]{LimaSarig2019}. 
  The measure is supported on a topological transitive countable (possibly finite) subshift $\Sigma'$ in $\Sigma$, \mbox{see~\cite[\S 2]{AaronsonDenkerUrbanski}},  and we assume without loss of generality that $\Sigma'= \Sigma$. By Abramov's formula,  \cite{Abramov},  
${\mu}_{\Sigma}$ is an equilibrium measure for $\phi := -hr$, i.e., $P_{{\mu_\Sigma}}(\phi) = 0$. In the following write $\hat{h}:=h_{\mu_{
\Sigma}}(\sigma) = -\int \phi d\mu_{\Sigma}$ and fix $\delta$ with $0<\delta < \hat{h}$. 

Before we state the main result of this section, recall that, given $\kappa>0$, $N\in \N$, a subset $X\subset \Lambda$ is called  \textit{$(N,\kappa)$-separated} for $f:\Lambda \to \Lambda$  if for all $x,y\in X$,  $\sup_{0\leq i\leq N}d_{\Lambda}(f^i(x), f^i(y)) \geq  \kappa$. 
Note also that by Proposition \ref{zlocfinite}, for a subgraph $\mathcal{G}'$ of $\widehat{\mathcal{G}}$ with a finite set of vertices $\mathcal{A}'$, it holds that 
$\mathbf{w}(\mathcal{G}'):= \inf\{10^{-1}(p^u \wedge p^s)\, |\,  v=\Psi_x^{p^u,p^s}, v \supset S, S\in \mathcal{A}'\}$ is positive.  
Finally, let $C_{\Lambda}\geq 1$ be the constant in \eqref{C_Lambda}.
\begin{thm}\label{thm:separated}
There exist $R \in \mathcal{R}$, $v= \Psi_{x}^{p^u,p^s} \in \mathcal{A}$ with $R\subset Z(v)$, a finite subgraph $\mathcal{G}'\subset \widehat{\mathcal{G}}$ containing $R$,  $\kappa>0$ with  $\kappa < \mathbf{w}(\mathcal{G}')$,
 and an infinite subset $\mathcal{N} \subset \N$, such that the following holds. 
For all $N\in \mathcal{N}$, there exist an admissible rectangle $\mathcal{Q} = \mathcal{Q}(V^u_-, V^u_+; V^s_-, V^s_+)$ in $v$, $K\in \N$,  and $K$ loops $\Theta_1, \ldots, \Theta_K$ of length $N$ in ${\mathcal{G}}'$, based at $R$, such that,  
\begin{enumerate}[(i)]
\item $\dist(V^u_-, V^u_+), \dist(V^s_-, V^s_+) < \frac{1}{3C_{\Lambda}}\kappa$,
\item $x_i := \widehat{\pi}(\overline{\Theta_i}) \in {\mathcal{Q}}$, $\quad i=1, \ldots, K$,
\item $\{x_1, \ldots, x_K\}$ is $(N,\kappa)$-separated for the map $f:\Lambda \to \Lambda$,
\item $\phi^j(x_i) \notin \mathcal{Q}$  and  $\dist_{\Lambda}(\phi^j(x_i), \mathcal{Q}) > 3\kappa$, for all $i=1, \ldots, K$ and $0< j< N$,
\item $K\geq e^{N(\hat{h} - \delta)}$, 
\item $|\phi_{N}(\overline{\Theta_i}) - N\int\phi d\mu_{\Sigma}| \leq N \delta$, $\quad i=1, \ldots, K$.
\end{enumerate}
\end{thm}
\begin{proof}[Proof of Theorem \ref{thm:separated}]
By Lemma \ref{lem:loops}
with $\epsilon = \delta/8$, there is $R \in \mathcal{R}$ and an infinite set $\mathcal{M} \subset \N$   such that for all $m\in \mathcal{M}$, there is a finite collection  $\{\gamma_1, \ldots, \gamma_k\}, k \in \N$, of pairwise distinct loops based at $R$ of length $m$ such that 
\begin{align}\label{eq:cond1}
k\geq e^{m(\hat{h}-\delta/8)},
\end{align}
and 
\begin{align}\label{eq:concationations}
|\phi_{lm}(\overline{\gamma_{i_1}\gamma_{i_2}\cdots \gamma_{i_l}}) - lm \int \phi d\mu_{\Sigma}| \leq lm\delta/8,
\end{align} 
for all $l \in \N$ and $(i_1, \ldots, i_l) \in \{1, \ldots, k\}^l$.

We denote  by 
$\mathcal{D}$ the finite set of those $Z \in \mathcal{Z}$ such that $R \subset Z$, and let $D:= \# \mathcal{D}$.
By Theorem \ref{thm:fullmeasurecoding} (5), there is $d\in \N$ such that there are at most $d-1$ many $\underline{R} = (R_t)_{t\in \Z} \in \Sigma(\widehat{\mathcal{G}})$ in the same fibre of $\widehat{\pi}$  for which $R_t = R_0$ for infinitely positive and negative $t\in \N$.\footnote{In fact, one can choose $d= D^2+1$, see 
 \cite[p. 43]{LimaSarig2019}.}
In the following we will fix $m \in \mathcal{M}$ such that $k\in \N$ above satisfies 
 $$k\geq d+1,$$
and choose    
$d+1$ pairwise distinct loops 
$\alpha, \beta_1, \ldots, \beta_d$ among the loops $\gamma_1, \ldots, \gamma_k$.  Let $\mathcal{G}'$ be a finite subgraph in $\widehat{\mathcal{G}}$ that contains the loops $\gamma_1, \ldots, \gamma_k$. 

Fix $q \in \N$ so large that 
\begin{align}\label{eq:q_delta}
{q\delta} \geq 8\max\left\{ \hat{h}-\frac{\delta}{2},\, \,  \log D \right\}.
\end{align}
Let $s, n \in \N$. Then, for any $\bi= (\bi_1, \ldots, \bi_{ndq}) \in \{1,\ldots, k\}^{ndq}$, consider the loop $\Theta_\bi = \Theta_\bi(s,n)$ defined by 
\begin{align*}
\Theta_\bi = \alpha^s \prod_{i=0}^{n-1} \prod_{j=0}^{d-1} \Gamma_{\bi,i,j}\alpha^s,
\end{align*}
where
\begin{align*}
    \Gamma_{\bi,i,j}  = \gamma_{\bi_{idq + jq + 1}}\cdots \gamma_{\bi_{idq + jq + q}} \beta_{j+1}.
\end{align*}
The loops $\Theta_{\bi}$ are based at $R$ and are of length $(2s + nd(q+1))m$. 
We will find the loops of the theorem among the set $\left\{\Theta_{\bi}\, |\,  \bi \in \{1,\ldots, k\}^{ndq}\right\}$, for suitable choices of $s$ and $n$.

Before we proceed, define for a fixed $s \in \N$, 
and for $\underline{\bi} = (\bi_j)_{j\in \Z} \in \{1, \ldots, k\}^{\Z}$,  the path $\Theta_{\underline{\bi}}$ in $\Sigma(\widehat{\mathcal{G}})$ by 
\begin{align*}\Theta_{\underline{\bi}} = 
 \cdots \Gamma_{\underline{\bi},-2,d-1}\, \Gamma_{\underline{\bi},-1,0}\, &\Gamma_{\underline{\bi},-1,1}\cdots \Gamma_{\underline{\bi}, -1,d-1}\, \alpha^{2s}\, \\ &\Gamma_{\underline{\bi},0,0} \, \Gamma_{\underline{\bi}, 0,1} \cdots \Gamma_{\underline{\bi},0,d-1} \, \Gamma_{\underline{\bi},1,0} \cdots, \end{align*} 
 where $\Gamma_{\underline{\bi}, i,j} = \gamma_{{\bi}_{idq + jq + 1}} \cdots \gamma_{{\bi}_{idq + jq + q}}\beta_{j+1}$,  $i\in \Z$, $j\in \{0, \ldots, d-1\},$
 and the zeroth position is determined via  $\Theta_{\underline{\bi}}|_{[-sl,sl]} = \alpha^{2s}.$
 
 Note that since $\alpha \neq \beta_i$, for all $i=1,\ldots, d$,
 $\Theta_{\underline{\bi}}$ cannot be periodic. And since $\widehat{\pi}$ is finite-to-one,  $\widehat{\pi}(\Theta_{\underline{\bi}})$ cannot be a periodic point of $f$. Indeed, assume it was a periodic point of $f$, say of period $\mathfrak{p}$, then the preimage, a finite set, is invariant under $\sigma^{\mathfrak{p}}$ and hence would consist of periodic sequences. 
 
 Furthermore, for $\underline{\bi} =  (\bi_j)_{j\in \Z}\in \{1, \ldots, k\}^{\Z}$, define $\Gamma_{\underline{\bi}}$ in $\Sigma(\widehat{\mathcal{G}})$ by 
 $$\Gamma_{\underline{\bi}} = \cdots \Gamma_{\underline{\bi},-2,d-1}\, \Gamma_{\underline{\bi},-1,0}\,\Gamma_{\underline{\bi},-1,1}\cdots \Gamma_{\underline{\bi}, -1,d-1}\,  \Gamma_{\underline{\bi},0,0}\, \Gamma_{\underline{\bi}, 0,1} \cdots \Gamma_{\underline{\bi},0,d-1}\,  \Gamma_{\underline{\bi},1,0} \cdots, $$ 
 where the zeroth position is determined via 
 $\Gamma_{\underline{\bi}}|_{[0,qm+l]} = \Gamma_{\underline{\bi},0,1}$.

\begin{claim}\label{cl:xi}
There is $s\in \N$, $\xi = \xi(s) >0$, and $n_0 \in \N$, such that for all $n \geq n_0$, and all $\bi \in \{1, \ldots, k\}^{ndq}$,
\begin{align*}
\min_{0< \tau<N}\dist_{\Lambda}(\widehat{\pi}(\overline{\Theta_{\bi}}), f^{\tau}(\widehat{\pi}(\overline{\Theta_{\bi}}))) \geq \xi, 
\end{align*}
where $N = (2s + nd(q+1))m$ is the length of $\Theta_{\bi}$.  
\end{claim}

\begin{claimproof}
Assume the contrary. Then, there is a sequence $(s_i)_{i\in \N}$ with $s_i \to \infty$ such that for each $s = s_i$, there are sequences $(n_j)_{j\in \N}$ with $n_j \to \infty$, $(\bi_j)_{j\to \infty}$ in $\{1,\ldots, k\}^{n_jdq}$, and $(\tau_j)_{j\in \N}$ with $0 < \tau_j < N_j  = (2s + n_jd(q+1))m$ such that 
\begin{align}\label{eq:dist:xi}
\dist_{\Lambda}\left(\widehat{\pi}(\overline{\Theta_{\bi_j}}), f^{\tau_j}(\widehat{\pi}(\overline{\Theta_{\bi_j}})\right)  \to 0, \text{ as } j \to \infty,
\end{align}
where $\Theta_{\bi_j} = \Theta_{\bi_j}(s,n_j)$. Choose for each $s_i$ such a sequence $(n_j,\bi_j,\tau_j)_{j\in \N}$. 

We first show that for all $s = s_i$ above, necessarily  
\begin{equation}\label{limsupinfty}
\limsup_{j\to \infty}  \tau_j = + \infty, \text{ and } \limsup_{j\to \infty}( N_j -\tau_j) = +\infty.
\end{equation}
We assume that $\limsup \tau_j < +\infty$ and will derive a contradiction, the argument to exclude $\limsup (N_j - \tau_j) < +\infty$ is analogous. 
After passing to a subsequence of $(n_j,\bi_j, \tau_j)_{j\in \N}$, we may assume that $\tau_j \equiv \tau \neq 0$. 
After passing to a further subsequence if necessary, we can find $\underline{\bi} \in \{1, \ldots, k\}^{\Z}$ and a sequence  $(M_j)_{j\in \N}$ in $\N$ with $M_j \to \infty$ such that
\begin{align*}
\sigma^{\tau}(\Theta_{\underline{\bi}})|_{[-M_j,M_j]} = \sigma^{\tau}(\overline{\Theta_{\bi_j}})|_{[-M_j,M_j]}.  
\end{align*}
Since $\widehat{\pi}$ is H\"older continuous, $f^{\tau}(\widehat{\pi}(\overline{\Theta_{\bi_j}})) = \widehat{\pi}(\sigma^{\tau}(\overline{\Theta_{\bi_j}})) \to \widehat{\pi}(\sigma^{\tau}(\Theta_{\underline{\bi}}))$ and 
$\widehat{\pi}(\overline{\Theta_{\bi_j}}) \to \widehat{\pi}(\Theta_{\underline{{\bi}}})$.
Hence, by \eqref{eq:dist:xi},  $\widehat{\pi}(\sigma^{\tau}(\Theta_{\underline{\bi}}))= \widehat{\pi}(\Theta_{\underline{{\bi}}})$. This is a contradiction, since $\widehat{\pi}(\Theta_{\underline{{\bi}}})$ cannot be a periodic point of $f$.

Recall that for each $s_i$ we chose a sequence $(n_j, \bi_j, \tau_j)_{j\in \N}$. By taking a diagonal subsequence for $i\to \infty$, we obtain a sequence $(s_j, n_j, \bi_j, \tau_j)_{j\in \N}$, and by  \eqref{limsupinfty} we may  assume that
\begin{equation}\label{limsupinfty1}
\limsup_{j\to \infty}  \tau_j - s_j = + \infty, \text{ and } \limsup_{j\to \infty}( (N_j - s_j) -\tau_j) = +\infty.
\end{equation}
After passing to a subsequence we can find $\underline{\bi} \in \{1, \ldots, k\}^{\Z}$, a sequence  $(M_j)_{j\in \N}$ in $\N$ with $M_j \to \infty$, and $u \in \{0, \ldots, d(q + 1)m  -1\} $ such that
\begin{align*}
\sigma^{u}(\Gamma_{\underline{\bi}})|_{[-M_j,M_j]} = \sigma^{\tau_j}(\overline{\Theta_{\bi_j}})|_{[-M_j,M_j]}
\end{align*}
and 
\begin{align*}
\overline{\alpha}|_{[-M_j,M_j]} = \overline{\Theta_{\bi_j}}|_{[-M_j,M_j]}, \text{ for all } j \in \N.
\end{align*}
By the H\"older continuity of $\widehat{\pi}$, we have that 
$f^{\tau_j}(\widehat{\pi}(\overline{\Theta_{\bi_j}})) =  \widehat{\pi}(\sigma^{\tau_j}(\overline{\Theta_{\bi_j}}))  \to \widehat{\pi}(\sigma^{u}(\Gamma_{\underline{\bi}})) =:y$ and 
$\widehat{\pi}(\overline{\Theta_{\bi_j}})\to \widehat{\pi}(\overline{\alpha})=:z$. By \eqref{eq:dist:xi}, necessarily $y=z$. It follows that for $j=0,1,\ldots, d-1$, 
\begin{align*}
\widehat{\pi}(\sigma^{j(q+1)m + u}(\Gamma_{\underline{\bi}})) &= f^{j(q+1)m}(y) \\ &= f^{j(q+1)m}(z) =  \widehat{\pi}(\sigma^{j(q+1)m}(\overline{\alpha})) = \widehat{\pi}(\overline{\alpha}).
\end{align*}
On the other hand, since the loops $\alpha, \beta_1, \ldots, \beta_d$ are pairwise distinct, the elements  $\sigma^{j(q + 1)m}(\Gamma_{\underline{\bi}})$, $j=0,1, \ldots, d-1$, are also pairwise distinct. This contradicts the choice of $d$. 
\end{claimproof}

From now on we will fix $s \in \N$, $n_0 \in \N$, $\xi>0$ such that the assertions of Claim \ref{cl:xi} hold. To obtain a collection of loops $\Theta_{\bi}$ with the desired properties we will further restrict to a suitable subset. 

For each $n \in \N$, we define a partition 
$$\{1, \ldots, k\}^{ndq} = \bigcup_{\underline{Z} \in \mathcal{D}^{n+1}} I^n(\underline{Z})$$  into $D^{n+1}$ disjoint subsets $I^n(\underline{Z})$, $\underline{Z} \in \mathcal{D}^{n + 1}$, as follows.  
For any $\bi \in \{1, \ldots, k\}^{ndq}$, write $\overline{\Theta_\bi} = (A^{(\bi)}_t)_{t\in \Z}$ in $\Sigma^{\#}(\widehat{\mathcal{G}})$, and choose a gpo $\underline{a}^{(\bi)} = (a^{(\bi)}_t)_{t\in \Z}$ in $\Sigma^{\#}(\mathcal{G})$ with $A^{(\bi)}_t \subset Z(a^{(\bi)}_t)$, for all $t\in \Z$. This is possible by Lemma \ref{lem:RinZ}.
Given this, for any $\underline{Z} = (Z_0, \ldots, Z_{n}) \in \mathcal{D}^{n+1}$, 
define the set $I^n(\underline{Z})\subset \{1, \ldots, k\}^{ndq}$  as the set of those $\bi\in \{1, \ldots, k\}^{ndq}$ such that
$$a^{(\bi)}_{(s + jd(q + 1))m} = Z_j, \, \text{ for } j=1, \ldots, n; \quad  a^{(\bi)}_0 = Z_{0}.$$
This defines our partition.
\begin{claim}\label{cl:1*}
There exist $\underline{Z}^n\in \mathcal{D}^{n+1}, n\in \N$, and $n_1\geq n_0$ such that for all $n\geq n_1$,
$$\# I^n(\underline{Z}^n) \geq e^{N(\hat{h} -\delta/2)},$$
where $N= (2s + nd(q+1))m$. 
\end{claim}
\begin{claimproof}
For any $n\in \N$ there is $\underline{Z}^n\in \mathcal{D}^{n+1}$ such that
\begin{align*}
\# I^n(\underline{Z}^n) \geq \frac{1}{D^{n+1}} k^{ndq}\geq \frac{1}{D^{n+1}}e^{mndq(\hat{h} -\delta/8)}.
\end{align*}
Let $n_1 \geq n_0$  such that 
\begin{align}\label{eq:n0*}
n_1d\geq {2s}.
\end{align}
It follows with \eqref{eq:q_delta} that 
\begin{align*}
n_1d(q\delta/4 - (\hat{h}-\delta/2)) \geq 2s(\hat{h}-\delta/2), 
\end{align*}
which means that for all $n\geq n_1$,
\begin{align}\label{eq:ndq_2s}
ndq\delta/4 \geq (2s+ nd)(\hat{h}-\delta/2).
\end{align}
Hence, with \eqref{eq:q_delta}, for all $n\geq n_1$, 
\begin{align*}
\#I^n(\underline{Z}^n)  &\geq
e^{mndq(\hat{h}-\delta/4)} \\ &\geq e^{m(ndq + (2s + nd))(\hat{h}-\delta/2)} = e^{N(\hat{h}-\delta/2)}.
\end{align*}
\end{claimproof}

We now fix the family $\underline{Z}^n \in \mathcal{D}^{n+1}$, $n\in \N$, and $n_1\geq n_0$ such that the assertions of Claim \ref{cl:1*} hold.  
\begin{claim}\label{cl:kappa}
There is $\kappa' >0$, $n_2\geq n_1$ such that for all $n\geq n_2$, $\bi,\hat{\bi} \in I^n(\underline{Z}^n)$ with $\bi\neq \hat{\bi}$,    
\begin{align}\label{eq:kappa}
\sup_{\tau\in \N} \dist_{\Lambda}(f^{\tau}(\widehat{\pi}(\overline{\Theta_\bi})), f^{\tau}(\widehat{\pi}(\overline{\Theta_{\hat{\bi}}}))) \geq \kappa'.
\end{align}
\end{claim}
\begin{claimproof}
Assume the contrary and choose sequences $(n_j)_{j\in \N}$ in $\N$ with $n_j \to \infty$, $(\bi^{j})_{j\in \N}$ and $(\widehat{\bi}^{j})_{j\in \N}$ with $\bi^j \neq \widehat{\bi}^j \in I^{n_j}(\underline{Z}^{n_j})$, $j\in \N$, 
 such that for any sequence $(\tau_j)_{j\in \N}$ in $\N$, 
$\dist_{\Lambda}(f^{\tau_j}(\widehat{\pi}(\overline{\Theta_{\bi^j}})), f^{\tau_j}(\widehat{\pi}(\overline{\Theta_{\widehat{\bi}^j}})))\to 0.$

Given a sequence $(\tau_j)_{j\in \N}$, let  $\underline{R}^j = (R^j_t)_{t\in \Z}$ and $\widehat{\underline{R}}^j = (\widehat{R}^j_t)_{t\in\Z}$ in $\Sigma^{\#}(\widehat{\mathcal{G}})$ defined by 
$${\underline{R}}^j = \sigma^{\tau_j}(\overline{\Theta_{\bi^j}}), \text{ and } \widehat{\underline{R}}^j = \sigma^{\tau_j}(\overline{\Theta_{\widehat{\bi}^j}}).$$
Note that if we consider the gpos $\underline{v}^{(j)} := \sigma^{\tau_j}(\underline{a}^{(\bi^j)})$ and $\underline{\widehat{v}}^{(j)} := \sigma^{\tau_j}(\underline{{a}}^{(\widehat{\bi}_j)})$, then  
$R^{j}_t \subset Z(v^{(j)}_t)$ and $\widehat{R}^{j}_t \subset Z(\widehat{v}^{(j)}_t)$, for all $t \in \Z$. 

Take now a subsequence of $(n_j, \bi^j, \widehat{\bi}^j)$ and a  sequence $\tau_j$ such 
 that $R^j_0 = B$ and $\widehat{R}^j_0 = \widehat{B}$  with $B\neq \widehat{B}$ in $\mathcal{R}$, and $\underline{R}^j|_{[-u,m-u]} = \gamma_{r}$ and $\widehat{\underline{R}}^j|_{[-u,m-u]} = \gamma_{\widehat{r}}$, for some $r \neq \widehat{r} \in \{1, \ldots, k\}$, where the path $\gamma_r$  passes through $B$  and the path $\gamma_{\widehat{r}}$  passes through $\widehat{B}$ at the same position $u\in \{1, \ldots, m\}$. 

For any $M\in \N$, we can pass to a further subsequence of $(n_j, \bi^j, \widehat{\bi}^j, \tau_j)_{j\in \N}$ such that then 
\begin{align*}
 &R^{j}|_{[-M,M]} =  R^{j'}|_{[-M,M]} \text{ and }\widehat{R}^{j}|_{[-M,M]} =  \widehat{R}^{j'}|_{[-M,M]}, \text{ for all }  j,j'\in \N, \\
 &v^{j}|_{[-M,M]} = v^{j'}|_{[-M,M]} \,\,\, \text{ and } \, \widehat{v}^{j}|_{[-M,M]} = \widehat{v}^{j'}|_{[-M,M]}, \, \, \text{ for all }  j,j'\in \N. 
\end{align*}
By taking diagonal sequences, given a monotonic sequence $M_j \to \infty$, we can find a subsequence of $(n_j, \bi^j, \widehat{\bi}^j, \tau_j)_{j\in \N}$, still denoted the same, such that for all $j,j'\in \N$, $j'\geq j$, 
\begin{align}
&R^{j'}|_{[-M_j,M_j]} = R^j|_{[-M_j,M_j]} \text{ and } \widehat{R}^{j'}|_{[-M_j,M_j]} = \widehat{R}^j|_{[-M_j,M_j]}, \label{goodsequence1} \\
&v^{j'}|_{[-M_j,M_j]} = v^j|_{[-M_j,M_j]}\, \, \,  \text{ and } \, \, \widehat{v}^{j'}|_{[-M_j,M_j]} = \widehat{v}^j|_{[-M_j,M_j]}. \label{goodsequence2}
\end{align}
Conditions \eqref{goodsequence1} and \eqref{goodsequence2} properly define  $\underline{R} = (R_t)_{t\in \Z} \in  \Sigma(\widehat{\mathcal{G}})$ with  $R_t  := R_t^j$, if $t\in [-M_j,M_j]$, and properly define the gpo 
$\underline{v}= (v_t)_{t\in \Z}$ with $v_t := v_t^j$, if $t\in [-M_j,M_j]$. Analogously we obtain $\widehat{\underline{R}}$ and $\widehat{\underline{v}}$. 

Since $\bi_j,  \widehat{\bi}_j \in I^{n_j}(\underline{{Z}}^{n_j})$,  there is $t_0, t_1 \in \Z$, $t_0 < 0 < t_1$, such that
\begin{align*}
Z(v_{t_0}) = Z(\widehat{v}_{t_0}) \text{ and } Z(v_{t_1}) = Z(\widehat{v}_{t_1}). 
\end{align*}  
By Lemma \ref{lem:RinZ}, $y:= \widehat{\pi}(\underline{R}) \subset Z(v_0)$ and $\widehat{y} := \widehat{\pi}(\widehat{\underline{R}}) \subset Z(\widehat{v}_0)$.
Moreover, by Lemma \ref{lem:uniqueness}, $y \neq \widehat{y}$. 
On the other hand,  $f^{\tau_j}(\widehat{\pi}(\overline{\Theta_{\bi^j}}))= \widehat{\pi} (\sigma^{\tau_j}(\overline{\Theta_{\bi^j}}) ) = \widehat{\pi}(\underline{R}^j)\subset Z(v_0)$, for all $j\in \N$. Since $\widehat{\pi}$ is H\"older continuous, it follows that  $f^{\tau_j}(\widehat{\pi}(\overline{\Theta_{\bi^j}})) \to y$ as $j \to \infty$.  
Similarly, $f^{\tau_j}(\widehat{\pi}(\overline{\Theta_{\widehat{\bi}^j}})) \to \widehat{y}$ as $j \to \infty$.
Therefore, using our assumptions, we get that $y=\widehat{y}$, a contradiction. 
\end{claimproof}

In the following fix $\kappa>0$ with $$\kappa< \min\{\kappa',  \xi/4, \mathbf{w}(\mathcal{G}')\}.$$ 
Choose a suitable infinite subset $\mathcal{N}'' \in [n_2, \infty)$ of natural numbers and  $Z \in \mathcal{D}$  such that for all $n\in \mathcal{N}''$ the first coordinate $Z_{0}$ of  $\underline{Z}^n=(Z_0, \ldots, Z_{n}) $ is $Z \in \mathcal{D}$.
We write $Z= Z(v)$ and $v = \Psi_x^{p^u,p^s}$.

For any path $R_{-t}\cdots R_t$ in $\widehat{\mathcal{G}}$, $t\in \N$, there exist $C>0$ and $\theta \in (0,1)$ such that for any $x_-$ and $x_+$ in $\overline{{}_{-t}[R_{-t}, \ldots, R_t]}$, we have that  $\dist(V_-, V_+) < C\theta^t$, where $V_-$ and $V_+$ are the local stable manifolds of $x_-$ and $x_+$ or the local unstable manifolds of $x_-$ and $x_+$. This follow by Proposition \ref{prop:contraction} applied to a chain $v_{-t} \to \cdots \to v_{t}$, where $R_i \subset Z(v_i)$ for $i=-t, \ldots, t$.   
Choose $r$ such that 
\begin{align}\label{r}
C \theta^{(s + r + \lfloor r/q \rfloor )m} < \frac{\kappa}{3C_{\Lambda}}.
\end{align}
For $n$ sufficiently large and any $\bi^* \in \{1, \ldots, k\}^{2r}$ define $I^n(\underline{Z}^n, \bi^*)$ to be the set of  $\bi\in I^n(\underline{Z}^n)$ with 
\begin{align*}
\begin{split}
&(\bi_1, \ldots, \bi_{r}) = (\bi^*_1, \ldots, \bi^*_r) \text{  and }\\
&(\bi_{(2s + nd(q+1))m - r + 1}, \ldots, \bi_{(2s + nd(q+1))m}) = (\bi^*_{r+1}, \ldots, \bi^*_{2r}).
\end{split}
\end{align*} 
Let $n_3\geq n_2$ such that additionally 
$$ (2s + n_3d(q+1))m\delta \geq 4r \log k.$$
This and the estimate in Claim \ref{cl:1*} show that there exist $\bi^* \in  \{1, \ldots, k\}^{2r}$ and an infinite set  $\mathcal{N}'\subset \mathcal{N}'' \cap [n_3,\infty)$  such that for all $n\in \mathcal{N}'$ 
\begin{align}\label{eq:v}
\# I^n(\underline{Z}^n,\bi^*) \geq \frac{1}{k^{2r}}\# I^n(\underline{Z}^n)  \geq \frac{1}{k^{2r}}e^{N(\hat{h}-\delta/2)}\geq e^{N(\hat{h}-\delta)},
\end{align}
where $N = (2s + nd(q+1))m$.
By \eqref{r}, 
\begin{align}\begin{split}\label{distV}
\dist(V^s(\widehat{\pi}(\overline{\Theta_{\bi}})), V^s(\widehat{\pi}(\overline{\Theta_{\bi'}})))&< \frac{\kappa}{3C_{\Lambda}}, \\  \dist(V^u(\widehat{\pi}(\overline{\Theta_{\bi}})), V^u(\widehat{\pi}(\overline{\Theta_{\bi'}})))&< \frac{\kappa}{3C_{\Lambda}}, \end{split}\end{align} 
for any $\bi, \bi'\in I^n(\underline{Z}^n,\bi^*)$, $n\in \mathcal{N}'$. 

Finally, let $$\mathcal{N} = \left\{ N = (2s + nd(q+1))m \,|\, n \in \mathcal{N}'\right\}.$$

For $N \in \mathcal{N}$, $N= (2s + nd(q+1))m$, we choose a denumeration   of 
$$\left\{\Theta_\bi\, |\, \bi \in I^n(\underline{Z}^n, \bi^*)\right\} =: \left\{\Theta_1, \ldots, \Theta_K\right\}.$$
Since local stable (unstable) manifolds in $v$ are either identical or disjoint, there is $i,i',j,j'\in \{1,\ldots, K\}$ 
such that for all $1\leq l\leq K$, \begin{align}\begin{split}\label{orderV} V^u_-:=V^u(\widehat{\pi}(\overline{\Theta_i})) \preceq_u V^u(\widehat{\pi}(\overline{\Theta_{l}})) &\preceq_u V^u(\widehat{\pi} (\overline{\Theta_{i'}}))=:V^u_+, \\ 
V^s_-:=V^s(\widehat{\pi}(\overline{\Theta_j})) \preceq_s V^s(\widehat{\pi}(\overline{\Theta_{l}})) &\preceq_s V^s(\widehat{\pi}(\overline{\Theta_{j'
}}))=:V^s_+.\end{split}\end{align}
This yields an admissible rectangle $\mathcal{Q}:=\mathcal{Q}(V^u_-,V^u_+,V^s_-,V^s_+)$ in $v$.

Assertion (i) of the theorem follows from \eqref{distV}, assertion (ii) from \eqref{orderV}. The separation property (iii) follows from Claim \ref{cl:kappa}.  Inequality (v) holds with \eqref{eq:v}, and inequality  (vi) with \eqref{eq:concationations}.

It remains to verify that (iv) holds. Write $\Psi := \Psi_x^{p^u \wedge p^s}$. 
Since the local stable and unstable manfifolds are $1/2$-Lipschitz, $\Psi^{-1}(\mathcal{Q})$ is contained in a rectangle in $[-p^u\wedge p^s,p^u\wedge p^s]^2$. Moreover, by \eqref{distV}, the distance between any two points on a common edge  in $\Psi^{-1}(\mathcal{Q})$ is  at most $(1+1/2)\kappa/{3C_{\Lambda}}$. For any $z_1,z_2\in \mathcal{Q}$, 
one can estimate the Euclidean distance of $\hat{z}_1 = \Psi^{-1}(z_1)$ and $\hat{z}_2 = \Psi^{-1}(z_2)$ as  
\begin{align}\label{eq:diam}
\begin{split}
\dist(\hat{z}_1, \hat{z}_2) &\leq  \min_{*\in \{\pm\}} \left( \dist(\hat{z}_1,\Psi^{-1}(V^u_*)) +   \dist(\hat{z}_2,\Psi^{-1}(V^u_*))\right) \\ &+ 
(1+1/2)\frac{\kappa}{3C_{\Lambda}} 
 < \frac{\kappa}{3C_{\Lambda}} + \frac{2\kappa}{3C_{\Lambda}}  \leq \frac{\kappa}{C_{\Lambda}}.
\end{split}
\end{align}
It follows that $$\diam_{\Lambda}(\mathcal{Q}) \leq  C_{\Lambda}\diam(\Psi^{-1}(\mathcal{Q})) < \kappa< \frac{\xi}{4}.$$
Hence, this, together with (ii) and the choice of $\xi$, implies (iv).
\end{proof}

\subsection{The horseshoe}
We are now in the situation to detect a horseshoe over the section $\mathcal{Q}$.
We show that the subshifts found above are in fact also induced 
 by rectangles of Markov type for an iterate of the return map.  
We will then use the order relations of stable and unstable manifolds of orbits in the horseshoe to pass to a further subshift which simplifies the geometric picture a little.  

\subsubsection{The first horseshoe}
We keep $\delta$ with  $0<\delta < \hat{h}$  fixed, and choose  an infinite set $\mathcal{N} \subset \N$, $R\in \mathcal{R}$, $v = \Psi_x^{p^u,p^s}$ double chart,  
 a finite $\mathcal{G}'\subset \widetilde{\mathcal{G}}$, $\kappa$ with  $0< \kappa<\mathbf{w}(\mathcal{G}')$, and, for any fixed  $N\subset \mathcal{N}$, loops   $\Theta_1, \ldots, \Theta_K$  in ${\mathcal{G}}'$ of common length $N \in \mathcal{N}$, based at $R$, and an admissible rectangle $\mathcal{Q} = \mathcal{Q}(V_-^u,V_+^u; V_-^s,V_+^s)$ in $v$ such that the assertions of Theorem \ref{thm:separated} hold. 
As it follows from their construction, $V_{\pm}^u$ and $V_{\pm}^s$ can be chosen to be unstable and stable manifolds in $v$ associated to suitable $N$-periodic gpos, and we will assume this in the following. 

We first use an adaption of an argument that appears in the approach of Katok-Mendoza \cite{Hasselblatt-Katok} for the construction of a horseshoe.  
We say that an admissible rectangle in $v$ is an \textit{$u$-admissible rectangle} resp.\ an \textit{$s$-admissible rectangle in $\mathcal{Q}$} if it is contained in $\mathcal{Q}$ and  its $u$-sides resp.\ its $s$-sides are contained in the $u$-sides resp.\ in the $s$-sides of $\mathcal{Q}$. 
For any  $i\in\{1, \ldots, K\}$, we consider the following construction. Choose a gpo $\underline{v} = (v_i)_{i\in \Z}$ with $v_0, v_{\pm N} = v$ that shadows $\widehat{\pi}(\overline{\Theta_i})$. Let $\mathcal{Q}_i$ be the admissible rectangle in $v$ defined by $\mathcal{Q}_i = \mathcal{Q}(V_-^u, V_+^u, \mathcal{F}_s^N[V^s_-], \mathcal{F}_s^N[V^s_+])$, where the graph transform is taken with respect to the chain $v_{-N} \to  \cdots \to  v_0$. 
Note that the graph transform $\mathcal{F}_s^N$ either preserves the order $\prec_s$ of all local stable manifolds in $v$ or it reverses their order.  It follows that necessarily $\widehat{\pi}(\overline{\Theta_i})\in\mathcal{Q}_i$. 
By enlarging if necessarily $\mathcal{Q}$ to a admissible rectangle (still denoted by $\mathcal{Q}$) that  is bounded by admissible manifolds of the  form $\mathcal{F}^N_s[V^s_{\pm}]$, $\mathcal{F}^N_u[V^u_{\pm}]$ for chains as above, we can assure that $\diam_{\Lambda}(\mathcal{Q}) < 3 \kappa$ and that   
 $\mathcal{Q}_i$ are  $u$-admissible rectangles in $\mathcal{Q}$ and that $f^N(\mathcal{Q}_i)$ are $s$-admissible rectangles in $\mathcal{Q}$. 
By Theorem \ref{thm:separated} (i) and Proposition \ref{prop:contraction}, there is $\theta \in (0,1)$ such that 
\begin{align*}
\begin{split}
& \dist(\mathcal{F}_s^j[V^s_{-}], \mathcal{F}_s^j[V^s_{+}]) \leq (\kappa/C_{\Lambda}) \theta^j,  \quad 1\leq j \leq N;\\ &\dist(\mathcal{F}_u^j[V^u_{-}], \mathcal{F}_u^j[V^u_{+}]) \leq (\kappa/C_{\Lambda}) \theta^j, \quad 1\leq j \leq N,  
\end{split}
\end{align*}
where $\mathcal{F}^j_s$ is applied to the chain $v_j \to \cdots \to v_0$, and $\mathcal{F}^j_u$ is applied to the chain $v_0 \to \cdots \to v_j$. 
By an estimate similar to  \eqref{eq:diam} we have in particular that  
\begin{equation}\label{diamCC} 
\diam_{\Lambda}  \left(f^j \left(\mathcal{Q}_i \right) \right) \leq \kappa, \quad 1\leq j \leq N.
\end{equation}

The above implies that by the separation property (iii) of Theorem \ref{thm:separated}, the sets $\mathcal{Q}_i$, $i=1, \ldots, K$, are pairwise disjoint. 
By \eqref{diamCC} and Theorem \ref{thm:separated} (iv),  
\begin{equation}\label{fjCCQ}
f^j\left(\mathcal{Q}_i\right) \cap \mathcal{Q} = \emptyset, 
\end{equation}
 for all $i \in \{1, \ldots, K\}$,  $j \in \{1, \ldots, N-1\}$.

Let $(\Sigma^*,\sigma^*)$ denote the full topological Markov shift over the alphabet $\{\Theta_1, \ldots, \Theta_K\}$. 
We define a map  $\pi^*: \Sigma^* \to \mathcal{Q}$ as follows. 
For  $\underline{\mathfrak{a}} = (\mathfrak{a}_t)_{t\in \Z}\in \Sigma^*$ with  $\mathfrak{a}_t = \Theta_{i_t}$, we set   
\begin{equation}\label{def:pistar}
\pi^*(\underline{\mathfrak{a}}):= \bigcap_{t \in \Z} f^{-tN}\left(\mathcal{Q}_{i_t}\right).
\end{equation}
\begin{prop}\label{prop:KM_in_LS}
The map $\pi^*:\Sigma^* \to \mathcal{Q}$
is well defined, and satisfies
$$\pi^*(\underline{\mathfrak{a}}) = \widehat{\pi}(\underline{R}),$$
where with $\underline{\mathfrak{a}} = (\mathfrak{a}_t)_{t\in \Z}$ and $\mathfrak{a}_t = \Theta_{i_t}$, the sequence ${\underline{R}} \in \Sigma(\widehat{\mathcal{G}})$ is given by $(R_j)_{j\in \Z} :=  ( \cdots \Theta_{i_{-t}} \cdots \Theta_{i_0} \cdots \Theta_{i_t} \cdots )$ and such that the $0$th entry of $\underline{R}$ is the first element of $\Theta_{i_0}$.

In particular,  $$\pi^* \circ \sigma^* = f^N \circ \pi^*.$$
\end{prop}
\begin{proof}
Any $u$-admissible rectangle in $\mathcal{Q}$  intersects any $s$-admissible rectangle in $\mathcal{Q}$. In particular $\mathcal{Q}_i \cap f^N(\mathcal{Q}_j) \neq \emptyset$ for all $i,j \in \{1, \ldots, K\}$. Moreover, $f^{-N}(\mathcal{Q}_i \cap f^N(\mathcal{Q}_j)) = f^{-N}(\mathcal{Q}_i) \cap \mathcal{Q}_j$ is an $u$-admissible rectangle in $\mathcal{Q}$.  By an induction argument we obtain that for all $k\in \N$, $\bigcap_{t=0}^k f^{-tN}(\mathcal{Q}_{i_t})$ is a $u$-admissible rectangle in $\mathcal{Q}$. Analogously we obtain that $\bigcap_{t=-k}^0 f^{-tN}(\mathcal{Q}_{i_{t}})$ is a $s$-admissible 
 rectangle in $\mathcal{Q}$, for any sequence  $\underline{i}= (i_t)_{t\in \Z}$ in $\{1, \ldots, K\}$.
It follows that the right hand side in  \eqref{def:pistar} is non-empty. 

Let $\underline{R} \in \Sigma(\widehat{\mathcal{G}})$ be as in the proposition and $x:= \widehat{\pi}(\underline{R})$.  
Choose $\underline{w} \in \Sigma(\mathcal{G})$ with $R_j \subset Z(w_j)$ for all $j\in \Z$ and $w_0 = v$.  Write $w_j = \Psi_{x_j}^{p_j^u,p^s_j}$, $\eta_j = p^s_j\wedge p_j^u$.  
For $i=1,\ldots, K$, let 
$\underline{R}^i = (R^i_j)_{j\in \Z}$ be the element in $\Sigma(\widehat{\mathcal{G}})$ that is obtained by periodically repeating $\Theta_i$. Let $t\in \Z$.
Then, for $j=Nt,\ldots, Nt+(N-1)$, the point $f^j(\widehat{\pi}(\overline{\Theta_{i_t}}))$ is contained in the closure  $\overline{R^{i_t}_j}$. Note that $\overline{R^{i_t}_j}\subset \Psi_{x_j}([-10^{-2}\eta_j, 10^{-2}\eta_j]^{2})$. 
By \eqref{diamCC} and since $\kappa < \mathbf{w}(\mathcal{G}')$, also $f^j(x) \in \Psi_{x_j}([-\eta_j, \eta_j])^{2}$, for $j=Nt, \ldots, Nt + (N-1)$. 
Altogether,  
$f^j(x) \in \Psi_{x_j}([-\eta_j, \eta_j]^2)$, for all $j \in \Z$. 
Hence $w$ shadows $x$, and so $x$ is the unique point with 
$x = \pi(\underline{w})$.
This shows that $\pi^*$ is well defined and also that $\pi^*(\underline{\mathfrak{a}}) = x = \widehat{\pi}(\underline{R})$.
\end{proof}
For $\underline{\mathfrak{a}} \in \Sigma^*$, there is is a well-defined  unstable manifold $V^u((\mathfrak{a}_i)_{i\leq0}) = V^u[\mathfrak{a}] : = V^u(\underline{w})$ and stable manifold $V^s((\mathfrak{a}_i)_{i\geq0}) = V^s[\mathfrak{a}] := V^s(\underline{w})$ 
 in $v$, where $\underline{w}$ is as in the proof of Proposition \ref{prop:KM_in_LS}, see Remark  \ref{rem:smanifold_welldefined}.
  The  following Lemma can be prove similarly as  Proposition \ref{prop:KM_in_LS}, by also using     \eqref{diamCC} and the  properties of the graph transforms.  
\begin{lem}
With $\underline{\mathfrak{a}}$ as above, 
\begin{align}
V^u[(\mathfrak{a}_{t})_{t\leq 0}] \cap Q_{i_0} &= \bigcap_{t\leq 0 }f^{-tN}(\mathcal{Q}_{i_t}), \quad \text{ and } \label{eq:image_unstable} \\
V^s[(\mathfrak{a}_{t})_{t\geq 0}] \cap Q_{i_0} &= \bigcap_{t\geq 0} f^{-tN}(\mathcal{Q}_{i_t}).  \label{eq:image_stable} 
\end{align}
\end{lem}
Proposition \ref{prop:KM_in_LS} yields an injective map  
\begin{align}\label{eq:iota}
\iota: \Sigma^* \to \Sigma(\widehat{\mathcal{G}})
\end{align}
with
$\iota \circ  \sigma^* = \sigma^N \circ \iota$, and 
$\pi^* = \widehat{\pi} \circ \iota$.
We define a roof function on $\Sigma^*$ by  
\begin{align}\label{eq:r*}
r^*:\Sigma^*\to (0,\infty), \, \, r^* := r_N \circ \iota.
\end{align}
Below, whenever it is clear that we work within the TMS $(\Sigma^*,\sigma^*)$, we will treat $\Theta_i$, $i=1, \ldots, K$, as symbols. For example,  $\overline{\Theta_i}$ denotes in that case the constant sequence $\cdots \Theta_i \cdots \Theta_i\cdots$ in $\Sigma^*$. 
\subsubsection{Partial orders on stable/unstable manifolds}
The partial orders $\preceq_s$ and $\preceq_q$ on the stable and unstable manifolds of $\widehat{\pi}(\overline{\Theta_i})$ in $v$ 
induce (total) orders     $\preceq_s$ and $\preceq_q$ on 
$\Sigma_{\rm{\per}}^*$, the set of periodic elements in $\Sigma^*$, as follows. First, 
if $\underline{\mathfrak{a}}, \underline{\mathfrak{b}} \in \Sigma^*$, we say that $\underline{\mathfrak{a}} \preceq_s \underline{\mathfrak{b}}$ if 
$V^s(\underline{\mathfrak{a}}) \preceq_s V^s(\underline{\mathfrak{b}})$. 
On $\Sigma_{\rm per}^*$, this defines a  total 
order $\preceq_s$. We say $\underline{\mathfrak{a}} \prec_s \underline{\mathfrak{b}}$ if  $\underline{\mathfrak{a}} \preceq_s \underline{\mathfrak{b}}$ and $\underline{\mathfrak{a}} \neq \underline{\mathfrak{b}}$. In an analogous way we define 
$\preceq_u, \prec_u$ on $\Sigma_{\rm per}^*$.

If $\Theta_{i_1}, \ldots,  \Theta_{i_k}$ are elements in $\{\Theta_1, \ldots, \Theta_K\}$, and $\star = s$ or $u$, we write 
 $[\overline{\Theta_{i_1}}; \ldots; \overline{\Theta_{i_k}}]_\star$ if it holds that  
\begin{align*}
\text{either }\, \overline{\Theta_{i_1}} \prec_\star \overline{\Theta_{i_2}} \prec_\star \cdots \prec_\star \overline{\Theta_{i_k}}\,  \, \, \text{  or  }\, \,\,   \overline{\Theta_{i_k}} \prec_\star \cdots \prec_\star \overline{\Theta_{i_2}} \prec_\star \overline{\Theta_{i_1}}.
\end{align*}
We omit $\star$ in the notation if it holds for both $\star=s$ and $\star= u$. 
We use the same notation for other elements in $\Sigma_{\per}^*$.

For $i=1, \ldots, K$, we say that $f^N$ \textit{keeps the orientation at $\Theta_i$} if $f^N|_{V^s(\overline{\Theta_i})}: V^s(\overline{\Theta_i}) \hookrightarrow V^s(\overline{\Theta_i})$  is orientation preserving. This holds if, and only if,  $f^{-N}|_{V^u(\overline{\Theta_i})}:V^u(\overline{\Theta_i})\hookrightarrow   V^u(\overline{\Theta_i})$ is orientation preserving. We say that $f^N$ \textit{reverses the orientation at $\Theta_i$} if $f^N$ does not keep the orientation at $\Theta_i$.

The following lemma will be useful later.  
\begin{lem}\label{lem:Sigma_order}
Let  $\underline{\mathfrak{a}}=(\mathfrak{a}_t)_{t\in \Z}$, $\mathfrak{a}_t=\Theta_{i_t}$ and $\underline{\mathfrak{b}} = (\mathfrak{b}_t)_{t\in \Z}$, $\mathfrak{b}_t=\Theta_{j_t}$ be two elements in $\Sigma_{\per}^*$. The following holds.  
\begin{enumerate}[(i)]
    \item If $\overline{\Theta_{i_0}}\prec_s \overline{\Theta_{j_0}}$ then   $\underline{a} \prec_s \underline{b}$. 
    \item Assume $\Theta_{i_0} = \Theta_{j_0} =: \Theta^0$ and $\overline{\Theta_{i_1}} \prec_s \overline{\Theta_{j_1}}$. 
    If $f^N$ keeps the orientation at $\Theta^0$, then  
     $\underline{\mathfrak{a}} \prec_s \underline{\mathfrak{b}}$, and if $f^N$ reverses the orientation at $\Theta^0$, then  $\underline{\mathfrak{b}}  \prec_s \underline{\mathfrak{a}}$. 
\item If $\overline{\Theta_{i_{-1}}} \prec_u \overline{\Theta_{j_{-1}}}$, then $\underline{\mathfrak{a}} \prec_u \underline{\mathfrak{b}}$.
  \item  Assume that $\Theta_{i_{-1}} = \Theta_{j_{-1}}=:\Theta^{-1}$ and $\overline{\Theta_{i_{-2}}}\prec_u \overline{\Theta_{j_{-2}}}$. 
  If  $f^N$ keeps the orientation at $\Theta^{-1}$, then 
     $\underline{\mathfrak{a}} \prec_u \underline{\mathfrak{b}}$, and if  $f^N$ reverses the orientation at $\Theta^{-1}$, then $\underline{\mathfrak{b}}  \prec_u \underline{\mathfrak{a}}$.     
\end{enumerate}
\end{lem}
\begin{proof}
The sequence in $\Sigma(\widehat{\mathcal{G}})$ that corresponds to $\underline{\mathfrak{a}}$  starts with the path $\Theta_{i_0}$.  Hence by considering the graph transform  $\mathcal{F}^N_s[V^s(\sigma^*(\underline{\mathfrak{a}}))]$, as well as the graph transforms $\mathcal{F}^N_s[V^s_{\pm}]$, all with respect to a chain $v_{-N}\to \cdots \to v_0$ induced by a gpo that shadows $\widehat{\pi}(\overline{\Theta_{i_0}})$, we obtain that the segment  
$V^s(\underline{\mathfrak{a}})\cap \mathcal{Q} = \mathcal{F}^N_s[V^s(\sigma^*(\underline{\mathfrak{a}}))]\cap \mathcal{Q}$ is contained in $\mathcal{Q}_{i_0}$. Similarly, $V^s(\underline{\mathfrak{b}})\cap \mathcal{Q} \subset \mathcal{Q}_{j_0}$. Hence, if $\overline{\Theta_{i_0}} \prec_s \overline{\Theta_{j_0}}$ then $V^s(\underline{\mathfrak{a}})\prec_s V^s(\underline{\mathfrak{b}})$, and so (i) holds. Also, 
if $\overline{\Theta_{i_1}} \prec_s \overline{\Theta_{j_1}}$, then  
$V^s(\sigma^*(\underline{\mathfrak{a}})) \prec_s V^s(\sigma^*(\underline{\mathfrak{b}}))$. If additionally $\Theta_{i_0} = \Theta_{i_1}=:\Theta^0$, then we can apply the graph transform $\mathcal{F}_s^N$, associated to a suitable chain $v_{-N}\to \cdots \to v_0$, to the stable manifolds $V^s_-$, $V_+^s$, $V^s(\sigma^*(\underline{\mathfrak{a}}))$, and $V^s(\sigma^*(\underline{\mathfrak{b}}))$. We obtain that $V^s(\underline{\mathfrak{a}}) =\mathcal{F}^N_s[V^s(\sigma^*(\underline{\mathfrak{a}}))]  \prec_s \mathcal{F}^N_s[V^s(\sigma^*(\underline{\mathfrak{b}}))]= V^s(\underline{\mathfrak{b}}) $ if $f^N$ keeps the orientation at $\Theta^0$, and   $V^s(\underline{\mathfrak{b}}) \prec_s V^s(\underline{\mathfrak{a}})$, otherwise. This shows assertion (ii). 

Applying the graph transform $\mathcal{F}_u^N$ to $V^u((\sigma^*)^{-1}(\underline{\mathfrak{a}}))$ associated to a suitable chain $v_0 \to \cdots \to v_N$ yields that $V^u(\underline{\mathfrak{a}}) \cap \mathcal{Q} \subset f^N(\mathcal{Q}_{i_{-1}})$. Similarly,  $V^u(\underline{\mathfrak{b}}) \cap \mathcal{Q} \subset f^N(\mathcal{Q}_{j_{-1}})$. If $\overline{\Theta_{i_{-1}}} \prec_u \overline{\Theta_{j_{-1}}}$, then $V^u(\underline{\mathfrak{a}})\prec_u V^u(\underline{\mathfrak{b}})$, and (iii) holds. Also, if $\overline{\Theta_{i_{-2}}}\prec_u\overline{\Theta_{j_{-2}}}$, then $V^u((\sigma^*)^{-1}(\underline{\mathfrak{a}}))\prec_u V^u((\sigma^*)^{-1}(\underline{\mathfrak{b}}))$. If additionally $\Theta_{i_{-1}} = \Theta_{j_{-1}}=:\Theta^{-1}$, we apply the graph transform $\mathcal{F}^N_u$ to $V^u_-$, $V^u_+$, $V^u((\sigma^*)^{-1}(\underline{\mathfrak{a}}))$, and $V^u((\sigma^*)^{-1}(\underline{\mathfrak{b}}))$. If $f^N$ keeps the orientation at $\Theta^{-1}$, then $V^u(\underline{\mathfrak{a}})\prec_u V^u(\underline{\mathfrak{b}})$, and $V^u(\underline{\mathfrak{b}})\prec_u V^u(\underline{\mathfrak{a}})$, otherwise. This shows assertion (iv).
\end{proof}
\begin{rem}
The statements of the lemma can be extended by induction in a straightforward way to the situation that the periodic    $\underline{\mathfrak{a}}$, $\underline{\mathfrak{b}}$ in $\Sigma^*$  agree on a couple of first entries.
\end{rem}

\subsubsection{The second horseshoe, contained inside the first}
We now construct a certain full subshift of $\Sigma^*$, which will be the shift in Theorem \ref{thm:main_coding}. 
\begin{lem}\label{lem:Omega}
Up to a restriction to an infinite subset of $\mathcal{N}$, there is for all $N\in \mathcal{N}$,  elements  $\Theta_{++}, \Theta_{--}, \Theta_{+}, \Theta_{-}$ and $\Omega_1, \ldots, \Omega_L$, $L \in \N$, in $\{\Theta_1, \ldots, \Theta_K\}$, such that
\begin{itemize}
\item either $f^N$ keeps the orientation of the stable and unstable manifolds of all of  $\overline{\Theta_{++}}, \overline{\Theta_{--}}, \overline{\Theta_{+}}, \overline{\Theta_{-}}, \overline{\Omega_1}, \ldots, \overline{\Omega_L}$ or it reverses the orientation of all of them;
\item $[\overline{\Theta_{--}}; \overline{\Theta_-}; \overline{\Theta_{++}}]$ and  $[\overline{\Theta_{--}}; \overline{\Theta_+}; \overline{\Theta_{++}}]$; 
\item $[\overline{\Theta_{-}}; \overline{\Omega_i}; \overline{\Theta_{+}}]$, for all $i \in \{1, \ldots, L\}$;
\item  $L \geq e^{N(\hat{h}-2\delta)}$. 
\end{itemize}
\end{lem}
\begin{proof}
We use the following combinatorial fact, which one 
 easily checks (cf.\ \cite[Lemma B.4]{AlvesMeiwesBraids}): Any finite  sequence $x_1, \ldots, x_n$, $n\in \N$, of  pairwise distinct natural numbers contains a subsequence $x_{i_1}, \ldots, x_{i_l}$ with $l\geq n/5$ and such that either $x_{i_1} \leq x_{i_j} \leq x_{i_l}$, or $x_{i_l} \leq x_{i_j} \leq x_{i_1}$ for all $1\leq j\leq l$.

After a restriction to an infinite subset of $\mathcal{N}$ if necessary, we can assume that for all $N\in \mathcal{N}$,    
\begin{align}\label{Ndelta}
N(\hat{h}-\delta) \geq \log 200, \, \, N\delta \geq \log 100 .
\end{align}
In a first step, we choose a subset $\{\Theta'_1, \ldots, \Theta'_{K'}\} \subset \{\Theta_1, \ldots, \Theta_K\}$ such that
$f^N$ keeps the orientation of the stable and unstable manifolds of all of  $\overline{\Theta'_1}, \ldots, \overline{\Theta'_{K'}}$ or it changes the orientation of all of them, and such that
$$K'\geq \frac{K}{2}.$$
By the fact above, there exist $\Theta_{--}$ and $\Theta_{++}$ in $\{\Theta'_1, \ldots, \Theta'_{K'}\}$ such that 
the number of $\Theta''\in\{\Theta'_1, \ldots, \Theta'_{K'}\} $
with 
$[\overline{\Theta_{--}};\overline{\Theta''};\overline{\Theta_{++}}]$ is at least $K'/5 -2$. Moreover, again by that fact, there exists $\Theta_-, \Theta_+$ and  $\Omega_1, \ldots, \Omega_L$, $L\in \N$, among those symbols $\Theta''$ above, such that $[\overline{\Theta_-};\overline{\Omega_i};\overline{\Theta_+}]$, and for which 
$L\geq (K'/{5}-2)/5 -2 \geq K/50 - 2$. 
Hence, by \eqref{Ndelta}, 
$L\geq K/100 \geq e^{N(\hat{h}-\delta) - \log 100} \geq e^{N(\hat{h}-2\delta)}$.
\end{proof} 
We fix $\Theta_{++}$, $\Theta_{--}$, $\Theta_+$, $\Theta_-$ and $\Omega_1, \ldots, \Omega_L$, $L\in \N$,  as in the lemma. 
Let $(\Sigma^!,\sigma^!)$ be the full shift in the symbols $\{\Omega_1, \ldots, \Omega_L\}$. We consider $\Sigma^!$ as a subset of $\Sigma^*$ and define  $\pi^! = \pi^*|_{\Sigma^!}:\Sigma^! \to \Lambda$.

\subsection{Fried surfaces and rectangles of Markov type}\label{sec:Fried}
In this section we will obtain the link $\mathcal{L}$ of our main theorems as periodic orbits in the first horseshoe $\im( \pi^*)$. 
The link $\mathcal{L}$ will consist of  a union of links $\mathcal{L}_i$, $i=0, \ldots, L$, each having three components that bound a piecewise embedded pair of pants ${F}_i$. These surfaces will be pairwise disjoint.  Surfaces of that type in horseshoes  were considered by Fried in order to construct global surfaces of sections, see e.g.\  \cite{Fried83}, and we refer to them as \textit{Fried surfaces}. See also the  recent works \cite{Hryniewicz_generic} and \cite{Contreras_generic} for the implications of the work of Fried for the construction of global surfaces of sections for Reeb flows.

We construct the family of surfaces $\mathfrak{F} = \{ F_0, F_1, \ldots, F_{L}\}$ as follows.  
Each surface $F_i$ will be piecewise embedded and will consists of a diamond shaped rectangle $D_i$ contained in $\mathcal{Q}$  and two rectangular regions $E^1_i$, $E^2_i$ parallel two the flow. $D_0$ will contain all the regions $D_i$, $i=1, \ldots, L$, and the latter will be pairwise disjoint, as will be the surfaces $F_i$, $i=1, \ldots, L$. For the construction see also Figure \ref{fig:D_i}.

Choose $\Theta_{++}, \Theta_{--}, \Theta_{+}, \Theta_{-}$ and $\Omega_1, \ldots, \Omega_L$, $L \in \N$, in $\{\Theta_1, \ldots, \Theta_K\}$ as in Lemma \ref{lem:Omega}. Specifically, $f^N$ either keeps the orientations of the stable and unstable manifolds of all of  $\overline{\Theta_{++}}, \overline{\Theta_{--}}, \overline{\Theta_{+}}, \overline{\Theta_{-}}, \overline{\Omega_1}, \ldots, \overline{\Omega_L}$ or it reverses all of them.

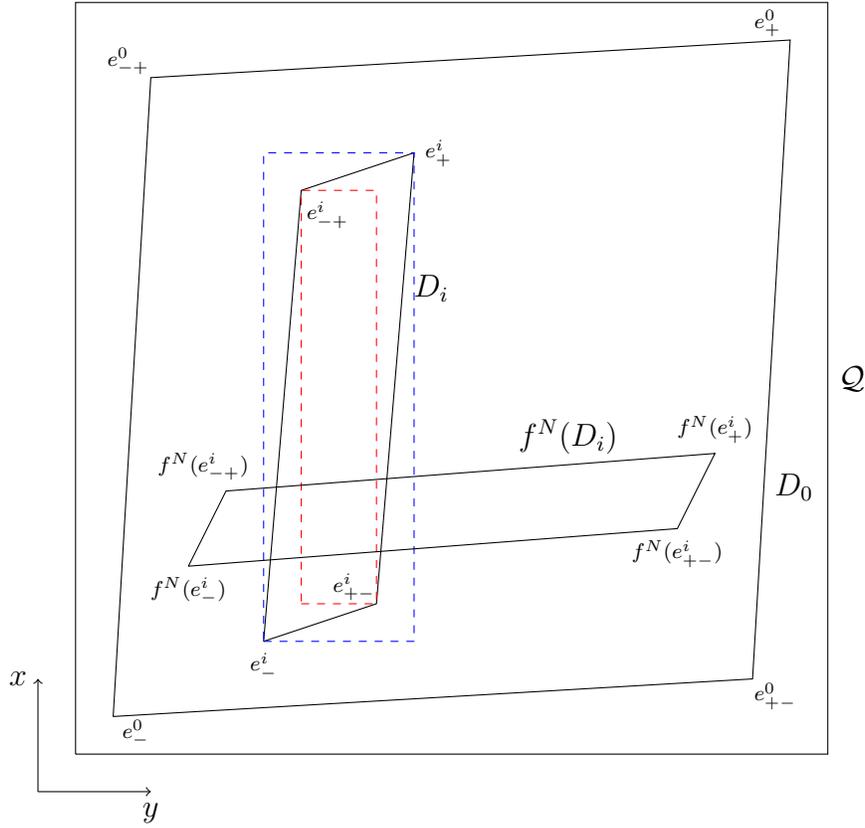
\begin{figure}
\begin{tikzpicture}
\draw[->]  (-0.5,-0.5) -- (-0.5,1) node[left]{$x$};
\draw[->]  (-0.5,-0.5) -- (1,-0.5) node[below]{$y$};
\draw (0,0)--(10,0) --(10,10) node[pos=0.5, right]{$\mathcal{Q}$}
--(0,10) -- (0,0);

\draw (0.5,0.5) node[below, shift={(0.3,0.1)}]{\tiny $e_-^0$} -- (9,1)node[below, shift={(0.3,0.1)}]{\tiny $e_{+-}^0$} -- (9.5,9.5) node[above, shift={(-0.3,-0.1)}]{\tiny $e_+^0$}node[pos=0.3,right]{$D_0$}
-- (1,9) node[above, shift={(-0.3,-0.1)}]{\tiny $e_{-+}^0$} -- (0.5,0.5);

\draw (2.5,1.5) node[below]{\tiny $e^i_-$} -- (4,2) node[above,shift={(-0.3,-0.1)}]{\tiny $e^{i}_{+-}$}-- (4.5,8) node[right]{ \tiny $e^{i}_{+}$}
node[pos=0.7,right]{$D_i$}-- (3,7.5) node[below, shift={(0.35,0)}]{ \tiny $e^{i}_{-+}$} -- (2.5,1.5);

 \color{blue}\draw[dashed](2.5,1.5) -- (4.5,1.5)  -- (4.5,8) -- (2.5,8) -- (2.5,1.5);
\color{red} \draw[dashed] (3,2) -- (4,2) -- (4,7.5) -- (3,7.5) -- (3,2);
\color{black}
\draw[-]
(1.5,2.5) node[below]{\tiny $f^N(e_-^i)$}-- (8,3) node[below]{\tiny $f^N(e_{+-}^i)$} --(8.5,4) node[above]{\tiny $f^N(e_+^i)$} -- (2,3.5) node[above, shift={(-0.3,0)}]{\tiny $f^N(e_{-+}^i)$} node[above,pos=0.3]{$f^N(D_i)$} -- (1.5,2.5);
\end{tikzpicture}
\caption{Schematic picture of $\mathcal{Q}$, $D_0$, and one region $D_i$ with its image under $f^N$. Here, we assume that  
$\Theta_- \prec_s \Omega_i \prec_s \Theta_+$ and $\Theta_- \prec_u \Omega_i \prec \Theta_+$, and that 
 $f^N$ keeps the orientation of the stable and unstable manifolds of all of 
the loops involved.
The two dashed rectangles are $\mathcal{Q}(e_-^i; e_+^i)$ and $\mathcal{Q}(e_{-+}^i;e_{+-}^i)$. The coordinates are those of the Pesin  chart, and note that in this picture,  $s$-admissible manifolds are  vertical and $u$-admissible manifolds  horizontal.}\label{fig:D_i}
\end{figure}

We start by defining $D_0$. 
We consider the following points in $\mathcal{Q}$: 
\begin{align*}
&e^0_- := {\pi}^*(\overline{\Theta_{--}}), \quad e^0_{+-} := \pi^*(\overline{\Theta_{++}\Theta_{--}}) \\ &e^0_+ := {\pi}^*(\overline{\Theta_{++}}), \quad e^0_{-+} := \pi^*(\overline{\Theta_{--}\Theta_{++}}).
\end{align*}
Let $v^0_-$ be a segment in $\mathcal{Q}(e^0_-;e^0_{-+})\subset \mathcal{Q}$ that connects the points $e^0_-$ and $e^0_{-+}$, and  
$v^0_+$ be a segment in $\mathcal{Q}(e^0_+;e^0_{+-})\subset \mathcal{Q}$ that connects the points  $e^0_+$ and $e^0_{+-}$.
Using the properties of stable and unstable manifolds, one easily checks that $$f^N(\mathcal{Q}(e^0_-;e^0_{-+})) = \mathcal{Q}(e^0_-;e^0_{+-}),\, \, f^N(\mathcal{Q}(e^0_+;e^0_{+-})) = \mathcal{Q}(e^0_+;e^0_{-+}).$$
Moreover, by  \eqref{fjCCQ},  for $j=1, \ldots, N-1$,  $$f^j(\mathcal{Q}(e^0_-;e^0_{-+})) \cap \mathcal{Q} = \emptyset \, \text{ and } f^j(\mathcal{Q}(e^0_+;e^0_{+-})) \cap \mathcal{Q} = \emptyset.$$  
Let $h^0_- := f^N(v^0_-) \in \mathcal{Q}(e^0_-; e^0_{+-})$ and  $h^0_+ := f^N(v^0_+) \in \mathcal{Q}(e^0_+; e^0_{-+})$. We can choose $v^0_{\pm}$  such that $v^0_{\pm}$ and $h^0_{\pm}$ do not intersect apart from their  endpoints, and denote by $D_0 \subset \mathcal{Q}$  the rectangle  bounded by $v^0_-$, $v^0_+$, $h^0_-$, $h^0_+$. 

It follows from Lemma \ref{lem:Sigma_order} that if $f^N$ keeps the orientations, then 
\begin{equation*}
[\overline{\Theta_{--}};\, \overline{\Theta_{--}\Theta_{++}};\, \overline{\Theta_{++}\Theta_{--}};\, \overline{\Theta_{++}}]_s, \text{ and }
[\overline{\Theta_{--}};\, 
\overline{\Theta_{++}\Theta_{--}};
\,\overline{\Theta_{--}\Theta_{++}};\, \overline{\Theta_{++}}]_u,
\end{equation*}
and if $f^N$ reverses the orientations, then 
\begin{equation*}
[\overline{\Theta_{--}\Theta_{++}};\, \overline{\Theta_{--}};\, 
\overline{\Theta_{++}};\, 
\overline{\Theta_{++}\Theta_{--}}]_s \text{ and } [\overline{\Theta_{--}\Theta_{++}};\,
\overline{\Theta_{++}};\,
\overline{\Theta_{--}};\, 
\overline{\Theta_{++}\Theta_{--}}]_u.
\end{equation*}
In the first case,  $\mathcal{Q}(e^0_{-+};e^0_{+-}) \subset \mathcal{Q}(e^0_-; e^0_+)$, and in the second case \linebreak 
$\mathcal{Q}(e^0_-; e^0_+) \subset \mathcal{Q}(e^0_{-+};e^0_{+-})$. 
Moreover, the sets 
$\mathcal{Q}(e^0_-;e^0_{-+})$, $\mathcal{Q}(e^0_-;e^0_{+-})$, $\mathcal{Q}(e^0_+;e^0_{+-})$, $\mathcal{Q}(e^0_+;e^0_{-+})$ 
are all contained in  $\overline{ \mathcal{Q}(e^0_-; e^0_+) \setminus \mathcal{Q}(e^0_{-+};e^0_{+-})}$ in the first case, 
and contained in $\overline{ \mathcal{Q}(e^0_{-+};e^0_{+-}) \setminus \mathcal{Q}(e^0_-; e^0_+)}$ in the second case. 
Hence, in the first case,   
\begin{equation}\label{D0_incl}
\mathcal{Q}(e^0_{-+};e^0_{+-})\subset D_0 \subset \mathcal{Q}(e^0_-; e^0_+), 
\end{equation}
and in the second case,  
\begin{equation}\label{D0_incl1}
\mathcal{Q}(e^0_-; e^0_+) \subset D_0 \subset \mathcal{Q}(e^0_{-+};e^0_{+-}).
\end{equation}
The region $E^0_-$ is given by $\bigcup_{z \in v^0_-}\{\varphi^t(z)\, |\, t \in [0, R_N(z)]\}$, the region $E^0_+$ by  $\bigcup_{z \in v^0_+}\{\varphi^t(z)\, |\, t \in [0, R_N(z)]\}$. Here, and below, $R= R_{\Lambda}$ denotes the roof function of $\Lambda$. 
The union $F_0 = D_0 \cup E^0_- \cup E^0_+$ is a piecewise embedded surface with three boundary components given by the image of loops $P^{0}_1(t):= \varphi^t(e^0_-), \, 0 \leq t \leq R_N(e^0_-)$, $P^{0}_2(t):= \varphi^t(e^0_+), \, 0 \leq t \leq R_N(e^0_+)$, and $P^0_3(t):= \varphi^t(e^0_{-+}),\, 0 \leq t \leq R_{2N}(e^0_{-+})$ (the image of the last loop coincides with the image of the $t\mapsto \varphi^t(e^0_{+-})$, $0 \leq t \leq R_{2N}(e^0_{+-})$). These three orbits from the link $\mathcal{L}_0$.

We proceed with the definition of the surfaces $F_i$, $i\in \{1, \ldots, L\}$. For any $i\in \{1, \ldots, L\}$, consider the following points in $\mathcal{Q}$:
\begin{align*}
&e^i_- := {\pi}^*(\overline{\Omega_i\Theta_{-}}), \quad e^i_{+-} := \pi^*(\overline{\Omega_i\Theta_{+}\Omega_i\Theta_{-}}), \\ &e^i_+ := {\pi}^*(\overline{\Omega_i\Theta_{+}}), \quad e^i_{-+} := \pi^*(\overline{\Omega_i\Theta_{-}\Omega_i\Theta_{+}}).
\end{align*}
Let $v^i_-$ be a segment in $\mathcal{Q}(e^i_-;e^i_{-+})\subset \mathcal{Q}$ that connects the points $e^i_-$ and $e^i_{-+}$, and  
$v^i_+$ be a segment in $\mathcal{Q}(e^i_+;e^i_{+-})\subset \mathcal{Q}$ that connects the points $e^i_+$ and $e^i_{+-}$.
Note that 
\begin{equation}
    f^{2N}(\mathcal{Q}(e^i_-;e^i_{-+})) = \mathcal{Q}(e^i_-;e^i_{+-}),
    \end{equation}
    \begin{equation}\label{FNei}
  f^N(\mathcal{Q}(e^i_-;e^i_{-+})) = \mathcal{Q}\left({\pi}^*(\overline{\Theta_{-}\Omega_i}); \pi^*(\overline{\Theta_{-}\Omega_i\Theta_{+}\Omega_i})\right),
  \end{equation}
  \begin{equation}\label{Fjei}
 f^j(\mathcal{Q}(e^i_-;e^i_{-+})) \cap \mathcal{Q} = \emptyset
  \end{equation}
  for $j\in \{1, \ldots, 2N-1\} \setminus \{N\}$, and, analogously,  
\begin{equation}
    f^{2N}(\mathcal{Q}(e^i_+;e^i_{+-})) = \mathcal{Q}(e^i_+;e^i_{-+}),
    \end{equation}
    \begin{equation}\label{FNei2}
  f^N(\mathcal{Q}(e^i_+;e^i_{+-})) = \mathcal{Q}\left({\pi}^*(\overline{\Theta_{+}\Omega_i}); \pi^*(\overline{\Theta_{+}\Omega_i\Theta_{-}\Omega_i})\right),
  \end{equation}
  \begin{equation}\label{Fjei2}
 f^j(\mathcal{Q}(e^i_+;e^i_{+-})) \cap \mathcal{Q} = \emptyset,
  \end{equation}
  for $j\in \{1, \ldots, 2N-1\} \setminus \{N\}$.

Let $h^i_- = f^{2N}(v^i_-) \subset \mathcal{Q}(e^i_-; e^i_{+-})$, $h^i_+ = f^{2N}(v^i_+) \subset \mathcal{Q}(e^i_-; e^i_{-+})$. We can assume that $v^i_{\pm}$ and  $h^i_{\pm}$ do not intersect apart from their endpoints, and let $D_i \subset \mathcal{Q}$ be the region bounded by $v^i_-$, $v^i_+$, $h^i_-$, $h^i_+$. 
Since $f^{2N}$ keeps the orientation of stable and unstable manifolds, one can easily see, similar to Lemma \ref{lem:Sigma_order}, that
\begin{align*}
&[\overline{\Omega_i\Theta_-};\,  \overline{\Omega_i\Theta_-\Omega_i\Theta_+};\, \overline{\Omega_i\Theta_+\Omega_i\Theta_-}; \, \overline{\Omega_i\Theta_+}]_s \quad \text{   and } \\
&[\overline{\Omega_i\Theta_-};\, \overline{\Omega_i\Theta_+\Omega_i\Theta_-}; \,  \overline{\Omega_i\Theta_-\Omega_i\Theta_+};\,  \overline{\Omega_i\Theta_+}]_u.
\end{align*}
Similar to  \eqref{D0_incl}, one checks that  
\begin{equation}\label{Di_incl}
\mathcal{Q}(e^i_{-+}; e^i_{+-}) \subset D_i \subset \mathcal{Q}(e^i_-; e^i_+).
\end{equation}
Note that 
\begin{align}\label{Di_Q}
\mathcal{Q}(e^i_-; e^i_+) \subset \mathcal{Q}_{i'},
\end{align}
where $\Omega_i = \Theta_{i'}$, 
and hence the sets 
\begin{equation}\label{Didisjoint}
D_i, \, i=1, \ldots, L,  \text{ are pairwise disjoint. } 
\end{equation}
Also,  $\mathcal{Q}(e^i_-; e^i_+)$ and $f^N(\mathcal{Q}(e_-^i,e_+^i))= \mathcal{Q}(\widehat{\pi}(\overline{\Theta_-\Omega_i});\widehat{\pi}(\overline{\Theta_+\Omega_i}))$ are  contained in the intersection 
 $\mathcal{Q}(e^0_{-+};e^0_{+-}) \cap \mathcal{Q}(e^0_-;e^0_+)$.  Hence by \eqref{D0_incl} and \eqref{D0_incl1},   
\begin{equation}\label{DisubsetDO}
D_i \subset D_0, 
\end{equation}
\begin{equation}\label{DisubsetDO1}
f^{N}(D_i) \subset D_0.
\end{equation}
Moreover, one verifies that  $f^N(D_i)$ crosses $D_j$, for all $i,j\in \{1,\ldots, L\}$. 

Define $E^i_-:= \bigcup_{z \in v^i_-}\{\varphi^t(z)\, |\, t \in [0, R_{2N}(z)]\}$, $E^i_+:=\bigcup_{z \in v^i_+}\{\varphi^t(z)\, |\, t \in [0, R_{2N}(z)]\}$. 
The union $F_i := D_i \cup E^i_- \cup E^i_+$ is a piecewise embedded surface with three boundary components given by the orbits $P^{i}_1(t) := \varphi^t(e^i_-)$, $0 \leq t \leq R_{2N}(e^i_-)$, $P^{i}_2(t) := \varphi^t(e^i_+)$, $0 \leq t \leq R_{2N}(e^i_+)$, and $P^i_3(t) := \varphi^t(e^i_{-+})$, $0 \leq t \leq R_{4N}(e^i_{-+})$. These orbits form the link  $\mathcal{L}_i$. 

Using that  $[\Theta_-;\Omega_i;\Theta_+]$ holds, it follows that  \begin{equation}\label{nointers}
\mathcal{Q}\left({\pi}^*(\overline{\Theta_{-}\Omega_i}); \pi^*(\overline{\Theta_{-}\Omega_i\Theta_{+}\Omega_i})\right) \cap 
\mathcal{Q}(e^i_-; e^i_+)= \emptyset,
\end{equation}
Hence, by \eqref{Di_incl}, and  \eqref{FNei},\eqref{Fjei},\eqref{FNei2},\eqref{Fjei2}, the surfaces $F_i$ are piecewise embedded. 
Since also, for $j\neq i$, 
\begin{align}\label{zz}
\begin{split}
\mathcal{Q}\left({\pi}^*(\overline{\Theta_{-}\Omega_i}); \pi^*(\overline{\Theta_{-}\Omega_i\Theta_{+}\Omega_i})\right) \cap 
\mathcal{Q}(e^j_-; e^j_+) &= \emptyset, \\
\mathcal{Q}\left({\pi}^*(\overline{\Theta_{+}\Omega_i}); \pi^*(\overline{\Theta_{+}\Omega_i\Theta_{-}\Omega_i})\right) \cap 
\mathcal{Q}(e^j_-; e^j_+) &= \emptyset,
\end{split}
\end{align}
the surfaces  $F_i$, $i=1, \ldots, L$, are pairwise disjoint. 
\begin{lem}\label{lem:Sigma!_inters}
Let $(i_j)_{j\in \Z}$ be any sequence with $i_j \in \{1, \ldots, L\}$, and consider $\mathfrak{a} = (\mathfrak{a}_j)_{j\in \Z}$ with  $\mathfrak{a}_j := \Omega_{i_j}$. Then the intersection $\bigcap_{j\in\Z} f^{-jN}(D_{i_j})$ defines a unique point and 
\begin{align}\label{eq:D_ij}\bigcap_{j\in\Z} f^{-jN}(D_{i_j}) = \pi^*((\mathfrak{a}_j)_{j\in \Z}).
\end{align}
\end{lem}
\begin{proof}
For $i=1, \ldots, L$, let $\mathcal{Q}'_i = \mathcal{Q}_{i'}$, where $i'$ is such that $\Theta_{i'} = \Omega_i$. 
That the left hand side in \eqref{eq:D_ij} is non-empty follows from the first inclusion in \eqref{Di_incl} and then similarly as in the beginning of the proof of Proposition \ref{prop:KM_in_LS}.
By the second inclusion in  \eqref{Di_incl} and by \eqref{Di_Q}, we have that $D_i \subset \mathcal{Q}'_i$, and hence 
the equality \eqref{eq:D_ij} follows then from Proposition \ref{prop:KM_in_LS}.
\end{proof}
Below we will also consider the surfaces $\widehat{F}_i$, $i=1,\ldots, L$, that one obtains by the union of the rectangles $\widehat{D}_i := f^N(D_i)\subset D_0$ with the regions $\widehat{E}^i_{\pm} = \bigcup_{z \in v^i_{\pm}} \{\varphi^t(z)\, |\, t \in [R_N(z), R_{3N}(z)]\, \}$. The surfaces $\widehat{F}_i$, $i=1,\ldots, L$, are pairwise disjoint and have the same boundary links $\mathcal{L}_i$ as the surfaces $F_i$.

We put $\mathcal{L} :=\bigcup_{i=0}^L \mathcal{L}_i$. 
Consider $n \in \N$, and some $n$-tuple $i = (i_0, \ldots, i_{n-1})\in \{1, \ldots, L\}^n$. Keeping the notation from the introduction, let  $P_{i_{0}, \ldots, i_{n-1}}$ be the periodic orbit that admits the parametrization 
$$P_{i_0, \ldots, i_{n-1}}(t) = \varphi^t(p), \, \,  0\leq t \leq R_{nN}(p)=: \per(P_{i_0, \ldots, i_{n-1}}),$$
where $p:=\pi^*(\overline{\Omega_{i_{0}}\cdots \Omega_{i_{n-1}}})$.
Denote by $\rho_{i_0, \ldots, i_{n-1}}$ the free homotopy class of loops in $M \setminus \mathcal{L}$ that is represented by $P_{i_{0}, \ldots, i_{n-1}}$.
\begin{lem}\label{lem:unique}
\begin{enumerate}
    \item If $\rho_{i_0, \ldots, i_{n-1}} = \rho_{i'_0, \ldots, i'_{n'-1}}$, then $n = n'$ and \linebreak $(i_0, \ldots, i_{n-1})$ is a cyclic permutation of $(i'_0, \ldots, i'_{n'-1})$.
\item $P_{i_0, \ldots, i_{n-1}}$ is, up to reparametrization, the only periodic orbit of $\varphi$ in $\rho_{i_0, \ldots, i_{n-1}}$. 
\end{enumerate}
\end{lem}
\begin{proof}
We can define a co-orientation of the surfaces $F_i$ via the vector field of the flow along $D_i$.  
Two loops that intersect the disjoint surfaces $F_1, \ldots, F_L$ transversally and positively but not in the same order, up to cyclic permutation, are not freely homotopic in $M\setminus \mathcal{L}$. The first assertion then follows.

To see that $(2)$ holds, let $\eta:[0,T] \to M$, $\eta(0)=\eta(T)$, be a loop that parametrizes some periodic orbit of the flow and has free homotopy class $[\eta]= \rho_{i_0, \ldots, i_{n-1}}$. If $\eta$ intersects the boundary of $D_0, \ldots, D_{L-1}$,  or $ D_L$, by a small perturbation of $\eta$,  this intersection becomes a (positive)  interior intersection  and no additional intersection will be created. It follows that $\eta$ intersects,  both, $\bigcup_{i=1}^L D_i\subset D_0$ and $D_0$,  exactly $n$ times. Hence there are no  intersection points of $\eta$ with   $D_0 \setminus \bigcup_{i=1}^L D_i$. 

We may assume, after possibly changing the parametrization, that $\eta$ intersects the rectangles $D_i$, $i=1, \ldots, L$, in the order $D_{i_0}, \ldots, D_{i_{n-1}}$, say in the points $x_{0}$, \ldots, $x_{{n-1}}$. 
Since $$f^j(D_i) \cap \mathcal{Q} = \emptyset , \text{ for } 1< j < N,$$
and 
$$f^N(D_i)  \subset D_0,$$
we necessarily have that $f^N(x_{0}) = x_{1}, f^{N}(x_{1}) = x_{2}, \ldots, f^N(x_{n-1}) = x_0$.  
By Lemma \ref{lem:Sigma!_inters}, the loop $\eta$ must be a parametrization of the orbit $P_{i_0, \ldots, i_{n-1}}$. 
\end{proof}
Fix a negative sequence  $\underline{\mathfrak{a}} = (\mathfrak{a}_t)_{t\leq 0}$, $\mathfrak{a}_t = \Omega_{j_t}$, and a positive sequence $\underline{\mathfrak{b}} = (\mathfrak{b}_t)_{t\geq 0}$, $\mathfrak{b}_t = \Omega_{j'_t}$. We   put  $l_{u} := V^u(\underline{\mathfrak{a}}) \cap D_{j_{0}}$ and 
$l_{s} := V^s(\sigma^*(\underline{\mathfrak{b}})) \cap \widehat{D}_{j'_0} = f^N(V^s(\underline{\mathfrak{b}})\cap D_{j'_0})$. 
Let $n \in \N$, $n\geq 3$, and consider  some $(n-2)$-tuple $(i_1, \ldots, i_{n-2})\in \{1, \ldots, L\}^{n-2}$.  Denote, as in the introduction, by  $\sigma_{i_1, \ldots, i_{n-2}}$ 
 the  homotopy class of paths in $M \setminus \mathcal{L}$ from $l_u$ to $l_s$ that is represented by the  chord $C_{i_1, \ldots, i_{n-2}}$ given by  $$C_{i_1, \ldots, i_{n-2}}(t) = \varphi^t(\pi^*(\underline{\mathfrak{c}})), \, \,  0\leq t \leq R_{nN}(\pi^*(\underline{\mathfrak{c}})),$$
where $\underline{\mathfrak{c}} = \underline{\mathfrak{c}}^{\underline{\mathfrak{a}},\underline{\mathfrak{b}};i_1, \ldots, i_{n-2}} = (\mathfrak{c}_j)_{j\in\N}$ is given by
$\mathfrak{c}_{j} = \mathfrak{a}_j$ if $j\leq 0$,  
 $\mathfrak{c}_j = \Omega_{i_j}$ if $1\leq j\leq n-2$, and  
$\mathfrak{c}_{j} = \mathfrak{b}_{j-(n-1)}$ if $j\geq n-1$.
\begin{lem}\label{lem:unique2}
\begin{enumerate}
    \item If $\sigma_{i_1, \ldots, i_{n-2}} = \sigma_{i'_1, \ldots, i'_{n'-2}}$, then $n = n'$ and \linebreak  $(i_1, \ldots, i_{n-2}) = (i'_1, \ldots, i'_{n'-2})$.
\item The chord $C_{i_1, \ldots, i_{n-2}}$ is, up to reparametrization, the only chord of $\varphi$ in the class $\sigma_{i_1, \ldots, i_{n-2}}$.
\item  
It is possible to extend 
$l_u$ and $l_s$ to embedded closed curves $\Lambda_u$ and $\Lambda_s$ in $M\setminus \mathcal{L}$ such that $C_{i_1, \ldots, i_{n-2}}$ is, up to reparametrization, the only chord in its homotopy class of paths in $M\setminus\mathcal{L}$ relative to $(\Lambda_u, \Lambda_s)$. 
\end{enumerate}
\end{lem}
\begin{proof}
To make the argument more transparent, we choose a sufficiently small $\delta>0$, push $l_u$ slightly along the negative direction of the flow to some $\hat{l}_u:= \varphi^{-\delta}(l_u)$,  and push $l_s$ slightly along the positive direction of the flow to some $\hat{l}_s:=\varphi^{\delta}(l_s)$. 
This gives a  natural 
 bijection from  the set of chords from $l_u$ to $l_s$ to the set of chords from $\hat{l}_u$ to $\hat{l}_s$, and since these are homotopies in  $M\setminus \mathcal{L}$, this yields also a bijection  from  homotopy classes of paths from $l_u$ to $l_s$ to homotopy classes of paths from $\hat{l}_u$ to $\hat{l}_s$.
It is sufficient to show the corresponding statements for $\hat{l}_u$ and $\hat{l}_s$.
We write $\hat{\eta}$ for the image of a chord $\eta$ under the above bijection, and write $\hat{\sigma}$ for the image of a homotopy class $\sigma$. 
 Note that $\hat{l}_u$ and $\hat{l}_s$ do not intersect the surfaces $F_0,\ldots, F_L$, and hence any two paths from $\hat{l}_u$ to $\hat{l}_s$ that are homotopic relative $(\hat{l}_u,\hat{l}_s)$ and that intersect the surfaces $F_0, F_1, \ldots, F_L$ transversally and positively, have the same number of intersections with $F_0$, and intersect the disjoint surfaces $F_1,\ldots, F_L$ in the same order. 
 
Assertion $(1)$ follows as Lemma \ref{lem:unique} $(1)$, since any chord of the form $\hat{C}_{i_1,\ldots, i_{n-2}}$ intersects the surfaces $F_0,\ldots, F_L$ transversally.

To see that $(2)$ holds, let  $\hat{\eta}:[0,T] \to M$, $\hat{\eta}(0) \in \hat{l}_u$, $\hat{\eta}(T) \in \hat{l}_s$, be any chord of $\varphi$ from $\hat{l}_u$ to $\hat{l}_s$ that represents the homotopy class $\hat{\sigma}_{i_1,\ldots, i_{n-2}}$. 
First, one can reason as in the proof of Lemma \ref{lem:unique} $(2)$: the path $\hat{\eta}$ does not intersect $D_0\setminus \bigcup_{i=1}^L D_i$ and it intersects $\bigcup_{i=1}^L D_i \subset D_0$ exactly $n+1$ times.
  The path $\hat{\eta}$ intersects the rectangles $D_i$, $i=1,\ldots, L$, in the order $D_{j_{0}}, D_{i_1}, \ldots, D_{i_{n-2}},D_{j'_0}, D_{j'_1}$, say in $\hat{x}_{0}, \ldots, \hat{x}_{n}$. 
The corresponding chord $\eta:[a,b] \to M$ from $l_u$ to $l_s$ intersects the rectangles $D_i$, $i=1, \ldots, L$, in the same order, in points $x_{0}, \ldots, x_n$ with $x_{0} = \eta(a)$, $f^N(x_{0}) = x_1$, $f^N(x_1) = x_2, \ldots, f^N(x_{n-1}) = x_n = \eta(b)$.
With \eqref{eq:image_unstable} and \eqref{eq:image_stable}, 
 we get that 
\begin{align*}
\begin{split}
x_{0} \in & \left(V^u[\underline{\mathfrak{a}}]\cap D_{j_{0}}\right)\cap\bigcap_{j=1}^{n-2} f^{-jN} (D_{i_j}) \cap f^{-(n-1)N}(V^s[\underline{\mathfrak{b}}] \cap D_{j'_0})\\ &\subseteq 
\bigcap_{i\leq 0} f^{-iN}(\mathcal{Q}_{j_{i}}) \cap \bigcap_{j=1}^{n-2} f^{-jN} (\mathcal{Q}_{i_{j}}) \cap \bigcap_{i\geq 0}
f^{-(i+(n-1))N}(\mathcal{Q}_{j'_{i}}),
\end{split}
\end{align*}
and hence $\eta(a) = x_{0}=\pi^*(\underline{\mathfrak{c}})$, and the chord  $\eta$ coincides with $C_{i_1,\ldots, i_{n-2}}$. 

We show assertion $(3)$. Let $\Lambda_u\subset M$ be an embedded closed curve such that
\begin{itemize}
    \item $\Lambda_u \cap \widehat{D}_{j_{-1}} = V^u[\underline{\mathfrak{a}}]$;
    \item the curves $\varphi^{t}( \Lambda_u)$, $t\in[-\delta, 0]$, do not intersect $\mathcal{L}$, and 
     $\hat{\Lambda}_u := \varphi^{-\delta}(\Lambda_u)$ does not intersect the surfaces $F_i$, $i=0,\ldots, L, i\neq j_{-1}$, nor  $\widehat{F}_i$, $i=1,\ldots, L$, and intersects ${F}_{j_{-1}}$ only in the two points $p_{\pm}$ on the boundary of the segment $\varphi^{-\delta}(\Lambda_u \cap \widehat{D}_{j_{-1}})$.
    \end{itemize}
 Moreover, let $\Lambda_s\subset M$ be an embedded closed curve such that
 \begin{itemize}
    \item $\Lambda_s \cap {D}_{j'_{1}} = V^s[\underline{\mathfrak{b}}]$;
    \item the curves $\varphi^{t}( \Lambda_s)$, $t\in[0, \delta]$, do not intersect $\mathcal{L}$, and
    $\hat{\Lambda}_s := \varphi^{\delta}(\Lambda_s)$ does not intersect the surfaces $\widehat{F}_i$, $i= 1,\ldots, L$, $i\neq j'_{1}$, nor ${F}_i$, $i=0,\ldots, L$, and intersects $\widehat{F}_{j'_{1}}$ only in the two points $q_{\pm}$ on the boundary of the segment $\varphi^{\delta}(\Lambda_s \cap {D}_{j'_{1}})$.
    \end{itemize} 
Consider a chord $\hat{\eta}:[a,b] \to M$ of $\varphi$ from $\hat{\Lambda}_{u}$ to $\hat{\Lambda}_s$ that is homotopic to $\hat{C}_{i_1,\ldots, i_{n-2}}$ in $M\setminus \mathcal{L}$ relative to $(\hat{\Lambda}_u,\hat{\Lambda}_s)$. We claim that the endpoints of $\hat{\eta}$ must lie in $\hat{l}_u$ and $\hat{l}_s$. We show that the first endpoint lies in $\hat{l}_u$.
Note that the horizontal component of the co-orientation of $F_{j_{-1}}$ induced from the vector field along $D_{j_{-1}}$ points towards the inside of $\varphi^{-\delta}(\widehat{D}_{j_{-1}})$ at the two intersection points $p_{\pm} \in \hat{\Lambda}_u \cap 
 F_{j_{-1}}\subset \varphi^{-\delta}(\widehat{D}_{j_{-1}})$.  
This means that $\hat{\eta}$ (or, if necessary, a small perturbation of $\hat{\eta}$ that is positively transverse to the surfaces $F_0,\ldots, F_L$) must intersect $F_0$ exactly $n+1$ times, and 
\begin{enumerate}[(i)]
\item intersects $F_1,\ldots, F_L$ in the order $F_{j_0}, F_{i_1},\ldots, F_{i_{n-2}}, F_{j'_0}, F_{j'_1}$, if \linebreak  $\eta(a)\in (\hat{\Lambda}_u\cap \varphi^{-\delta}(\widehat{D}_{j_{-1}}))\setminus \{p_{\pm}\}$, 
\item intersects $F_1,\ldots, F_L$ in the order $F_{j_{-1}}, F_{j_0}, F_{i_1}, \ldots, F_{i_{n-2}}, F_{j'_0}, F_{j'_1}$, \linebreak otherwise. 
\end{enumerate}
If $\hat{\eta}(a) \notin \hat{\Lambda}_u \cap \varphi^{-\delta}(\widehat{D}_{j_{0}})$, then  $\hat{\eta}$ intersects also $\bigcup_{i=1}^L D_i\subset D_0$ exactly $n+2$ times, a contradiction. If $\hat{\eta}(a) \in \{p_{\pm}\}\subset \varphi^{-\delta}(D_0 \setminus \bigcup_{i=1}^L D_i)$, and similarly if $\hat{\eta}(a)$ lies in any other point of $\varphi^{-\delta}(D_0 \setminus \bigcup_{i=1}^L D_i)$, we also obtain more than $n+1$ intersection points of $\hat{\eta}$ with $D_0$. Finally, by (i), $\hat{\eta}(a)$ cannot lie in $\varphi^{-\delta}(D_i)\subset \varphi^{-\delta}(F_i)$ if $i= 1,\ldots, L, i\neq j_{-1}$. 
We conclude that $\hat{\eta}(a) \in \hat{l}_u$.
The argument that the second endpoint lies in $\hat{l}_s$ is symmetric, looking at intersections of $\hat{\eta}$  with the surfaces $F_0$ and $\widehat{F}_i$, $i=1,\ldots, L$, instead.  We can then  proceed as in the proof of $(2)$ to show that $\hat{\eta}$ coincides with $\hat{C}_{i_1,\ldots, i_{n-2}}$. Assertion $(3)$ then follows from the identification described above of chords resp.\  homotopy classes of paths from $\Lambda_u$ to $\Lambda_s$ in $M\setminus \mathcal{L}$ with chords resp.\  homotopy classes of paths from $\hat{\Lambda}_u$ to $\hat{\Lambda}_s$ in $M\setminus \mathcal{L}$.
\end{proof}

\subsection{Proof of Theorem \ref{thm:main_coding}, and Theorems \ref{thm:link_growth} and \ref{thm:link_growth_rel}}
We are now in the situation to finish the proof of Theorem \ref{thm:main_coding}.
\begin{proof}[Proof of Theorem \ref{thm:main_coding}]
Let $\varepsilon$ with $0<\varepsilon< h_{\topo}(\varphi)= h$. We then carry out the constructions in the previous subsections, where we choose $\delta>0$ in the beginning of Section \ref{sec:separated} as  $\delta := \frac{R_{\min}\varepsilon}{6}< \hat{h}= h \int r d\mu_{\Sigma}$, where $R_{\min} = \min_{x\in \Lambda} R_{\Lambda}(x)$.   This yields an infinite subset $\mathcal{N} \subset \N$ and for any $N\in \mathcal{N}$ a  Markov flow $(\Sigma^!_r,\sigma^!_r)$  with roof function $r^! = r_N \circ \iota|_{\Sigma^!}: \Sigma^{!} \to (0,\infty)$, where $\iota:\Sigma^* \to \Sigma(\widehat{\mathcal{G}})$ is given in \eqref{eq:iota}. Let $\pi^!_r : \Sigma_r^! \to M$ be the map 
 induced by $\pi^! : = \pi^*|_{\Sigma^!}$, where 
$\pi^* = \widehat{\pi} \circ \iota$.  We show that there is $N\in \mathcal{N}$ such that $(\Sigma_r^!, \sigma_r^!, \pi^!_r)$ provide a coding for $K:=\im (\pi^!_r)$ with the claimed properties. 

We specify $N\in \mathcal{N}$ later below. Let $D_0, \ldots, D_L$ be the rectangles and $\mathcal{L}$ the link as constructed above. We set $D:= D_0$ and consider the return map $f_D: \dom(D) \to D$.  Set $\hat{V} = \bigcup_{i=1}^L D_i$.  Note that by the construction, $\hat{V} \subset \dom(f_D)$ and $\hat{V} \subset \dom(f^{-1}_D)$, and that \begin{align}\label{f_D} 
f_{D}|_{\hat{V}}= f^N_{\Lambda}|_{\hat{V}},\, \, f^{-1}_D|_{\hat{V}} = f^{-N}_{\Lambda}|_{\hat{V}}.
\end{align}
Furthermore, the collection of rectangles $D_1, \ldots, D_L$ is of Markov type. 
We now verify that  $K$ is a  horseshoe over $D$. 
Property (i) is true by definition. Since $\pi^!$ is injective 
 and H\"older continuous, also $\pi^!_r$ is injective and  H\"older continuous, which is property (ii).  
Property (iii) follows directly from the construction of the coding.   
It holds that $\im( \pi^!_r(\cdot,0)) \subset \hat{V} \subset D$.
Moreover, for $1\leq i \leq L$ we have that  
$$f^j(D_i) \cap \mathcal{Q} = \emptyset , \text{ for } 1< j < N.  $$
Hence the image $\im (\pi^!_r(\cdot, t)) $ is not contained in $D\subset \mathcal{Q}$ for $t\neq 0$. This proves (iv). 
It follows from Lemma \ref{lem:Sigma!_inters}  
that the image of $\pi^!$ coincides with 
$\bigcap_{n\in \Z}f_{\Lambda}^{nN}(\hat{V})$. With \eqref{f_D} and with  
 $\hat{V}^j$ defined as in \eqref{hatVj}, we obtain that for $j\geq 0$, $\hat{V}^j 
= \bigcap_{k=0}^j f_{\Lambda}^{-kN}(\hat{V})$, and that for $j\leq 0$, $\hat{V}^j = \bigcap_{k=0}^{-j} f_{\Lambda}^{kN}(\hat{V})$. 
Hence $\bigcap_{k\in \Z} f^{kN}_{\Lambda}(\hat{V}) = \bigcap_{j\in \Z} \hat{V}^j$. 
Similarily, for any $\underline{\mathfrak{a}} = (\mathfrak{a}_j)_{j\in \Z} = (\Omega_{i_j})_{j\in \Z}$ in $\Sigma^!$, we have that 
$\bigcap_{j\in \Z}D^{j}_{i_j} = \bigcap_{j\in \Z} f_{\Lambda}^{-jN}(D_{i_j})$.
Property (v) then follows from Lemma \ref{lem:Sigma!_inters}. That $K$ is in fact hyperbolic follows from the fact that $\im(\pi) \in \NUH_{\chi}(f_{\Lambda})$ and from the continuous dependence of expansive and contracting vectors with respect to the symbolic metric (\hspace*{-3px}\cite[Lemma 5.7]{LimaSarig2019}).

By construction, the link  $\mathcal{L}$ intersects $D, D_1,  \ldots, D_L$ in their corners, which is item $(2)$ of the theorem.  Item $(3)$ follows directly from Lemma \ref{lem:unique}, and item $(4)$ and $(5)$ from Lemma \ref{lem:unique2}. 
It remains to show $(1)$, that is,  that $h_{\topo}(\sigma^!_r)\geq h_{\topo}(\varphi) - \varepsilon$. 

By Theorem \ref{thm:separated} (vi), for all $i\in \{1,\ldots, L\}$, $$hr_N(\overline{\Omega_i})\leq  N(\hat{h} + \delta).$$

Hence, with  Lemma \ref{lem:Omega}, 
\begin{equation}\label{delta'*}
\sum_{i=1}^L e^{-hr_N(\overline{\Omega_i})} \geq e^{N(\hat{h}-2\delta)}e^{-N(\hat{h} + \delta)} =  e^{-3\delta N}= e^{-\frac{N\varepsilon}{2}R_{\min}}.
\end{equation}
Since 
\begin{equation*}
\sum_{i=1}^L e^{-(h-\frac{\varepsilon}{2})r_N(\overline{\Omega_i})} \geq  \sum_{i=1}^L e^{-hr_N(\overline{\Omega_i})}e^{\frac{N\varepsilon}{2}R_{\min}}\geq 1,
\end{equation*}
there is, if $N$ is sufficiently large, $h'$ with $h\geq h'\geq h-\frac{\varepsilon}{2}$ such that
\begin{align}\label{eq:1}\sum_{i=1}^L e^{-h' r_N(\overline{\Omega_i})} = 1.
\end{align}
Define $p_i :=e^{-h' r_N(\overline{\Omega_i})}$, and let $\nu$ the probability measure on $\Sigma^!$ given by  
$\nu(\underline{\mathfrak{a}})$ = $p_{i}$, where $\mathfrak{a}_0 = \Omega_i$. In other words, we consider  $\Sigma^!$ as the Bernoulli shift $(p_1, \ldots, p_N)$. Its entropy is  $h_{\nu}(\sigma^!)=-\sum_{i=1}^L p_i \log p_i$. 
Note that 
\begin{align}\label{eq:hnu}
h_{\nu}(\sigma^!)\geq e^{-3\delta N} N(h-{\epsilon}/{2})R_{\min}.\end{align}
By H\"older continuity of the roof function $r:\Sigma \to (0, \infty)$, there is $C>0$, independent on $N$,  such that for all $\underline{\mathfrak{a}} = (\mathfrak{a}_j)_{j\in \Z}= (\Omega_{i_j})_{j\in \Z} \in \Sigma^!$, 
$$|r^{!}(\mathfrak{a}) - r_N(\overline{\Omega_{i_0}})| \leq C.$$
We will now specify a choice of $N\subset\mathcal{N}$: by \eqref{eq:hnu}, we can choose before $N\subset \mathcal{N}$ so large that \begin{align}\label{N:last}
p_{\max}C < h_{\nu}(\sigma^!)\frac{\epsilon}{2h'(h'-\epsilon/2)},
\end{align}
where $p_{\max} := \max\{p_i \, |\, 1\leq i\leq N\}$. 
We estimate 
\begin{align*}
\int_{\Sigma^!}r^!\, d\nu  &= \sum_{i=1}^L \int_{\underline{\mathfrak{a}} \in \Sigma^!, \mathfrak{a}_0 = \Omega_i}r^!(\underline{\mathfrak{a}}) \, d\nu \\
&\leq \sum_{i=1}^L p_i \left(r_N(\overline{\Omega_i})
+C\right)
\\
&\leq\sum_{i=1}^L p_i r_N(\overline{\Omega_i}) + p_{\max}C \\
&= \frac{1}{h'}h_{\nu}(\sigma^{!}) + p_{\max}C. 
\end{align*}
The Lebesgue measure $\lambda$ induces an  invariant measure $\nu \times \lambda$ on $\Sigma^!_r$ for the flow $\sigma^!_r$. 
By the variational principle and by Abramov's formula \cite{Abramov} we obtain with \eqref{N:last} that
\begin{align*}
h_{\topo}(\sigma^!_{r}) &\geq h_{\nu\times \lambda}(\sigma^!_{r})   \\ &= \frac{h_{\nu}(\sigma^!)}{\int_{\Sigma^!}r^! d\nu } \\
  &> 
    \frac{h'}{h_{\nu}(\sigma^!) +  p_{\max} C h'} h_{\nu}(\sigma^!) > h' - \varepsilon/2 >h- \varepsilon.
 \end{align*}
\end{proof}
\begin{proof}[Proof of Theorem \ref{thm:link_growth}]
Let $\varepsilon>0$. Let  $(\Sigma_r,\sigma_r)$ be the Markov flow, $\mathcal{L}$ the link, and  $D$ the local section from  Theorem \ref{thm:main_coding}.   In particular it holds that   
$h_{\topo}(\sigma_r) \geq h_{\topo}(\varphi)-\varepsilon$.
Since $\sigma_r$ is expansive with respect to the Bowen-Walters metric, $h_{\topo}(\sigma_r) = h_{\topo,\epsilon_0}(\sigma_r)$ for any  $\epsilon_0>0$ sufficiently small. 
Moreover, if, in addition, $T$ is  sufficiently large, the points $(\underline{\mathfrak{\mathfrak{a}}},0)\in \Sigma_r$ with $\underline{\mathfrak{a}} = \overline{\Omega_{i_1}\cdots\Omega_{i_n}} \in \Sigma$ periodic  and $r_n(\underline{\mathfrak{a}}) \leq T$, together with a number of at most  $\frac{\max_{\underline{\mathfrak{b}}\in \Sigma} r(\underline{\mathfrak{b}})}{\epsilon_0}$ of their shifts by the flow $\sigma_r$, form a $(T, \epsilon_0)$-spanning set in $\Sigma_r$ for the Bowen-Walters metric, and hence 
$$h_{\topo}(\sigma_r) \leq \limsup_{T\to +\infty} \frac{1}{T} \log \#\{(n,\underline{\mathfrak{a}})\in \N \times\Sigma\, |\, \underline{\mathfrak{a}} = \overline{\Omega_{i_1}\cdots \Omega_{i_n}}, \,  r_n({\underline{\mathfrak{a}}}) \leq T\}.$$
Since $\pi_r: \Sigma_r \to M$ is injective and H\"older continuous, Theorem~\ref{thm:main_coding}~$(3)$ implies that   
\begin{align*}\#&\{(n,\underline{\mathfrak{a}})\in \N \times\Sigma\, |\, \underline{\mathfrak{a}} = \overline{\Omega_{i_1}\cdots \Omega_{i_n}}, \,  r_n({\underline{\mathfrak{a}}}) \leq T\}\\&\leq \# \{\rho \in \mathcal{H}_{\mathcal{L},\mathrm{sing}} \, |\, \per(\rho) \leq T\}.\end{align*} 
Altogether,
$$H^{\infty}(\varphi, M \setminus \mathcal{L}) > h_{\topo}(\varphi) - \varepsilon.$$
\end{proof}
\begin{proof}[Proof of Theorem \ref{thm:link_growth_rel}]
Let $\varepsilon>0$. The proof is similar to the proof of Theorem \ref{thm:link_growth}, in addition we use  Theorem~\ref{thm:main_coding}~$(4)$ and~$(5)$. 
This yields, for any previously fixed sequences  $\underline{\mathfrak{a}}= (\mathfrak{a}_i)_{i\leq 0}$ and $\underline{\mathfrak{b}}= (\mathfrak{b}_i)_{i\geq 0}$ of symbols in $\{\Omega_1, \ldots, \Omega_L\}$, two loops $\Lambda_1 = \Lambda_u$ and $\Lambda_2 = \Lambda_s$ in $M$ such that, with the notation in the introduction,
\begin{align*}
    \#&\{(i_1, \ldots, i_{n-2}) \in \{1, \ldots, L\}^{n-2} \,   |\, r_{n}(\underline{\mathfrak{c}}^{\underline{\mathfrak{a}}, \underline{\mathfrak{b}}, i_1, \ldots, i_{n-2}} ) \leq T, n\in \N \}\\ &\leq \# \{\rho \in \mathcal{P} \text{ singular}\, |\, \mathrm{len}(\rho) \leq T\}.
\end{align*}
Similar as above, the points $\sigma(\underline{\mathfrak{c}}^{\underline{\mathfrak{a}}, \underline{\mathfrak{b}}, i_1, \ldots, i_{n-2}}, 0)$ with $r_n(\underline{\mathfrak{c}}^{\underline{\mathfrak{a}}, \underline{\mathfrak{b}}, i_1, \ldots, i_{n-2}}) \leq T$, $n\in \N$,  
together with a number of at most $\frac{\max_{\underline{\mathfrak{d}}} r(\underline{\mathfrak{d}})}{\epsilon_0}$ of their shifts,  form a $(T, \epsilon_0)$-spanning set, we obtain that
$$\mathrm{H}^{\infty}(\varphi, M \setminus \mathcal{L},\Lambda_1,\Lambda_2) > h_{\topo}(\varphi) - \varepsilon.$$
\end{proof}

\section{Growth of contact homology in a link complement}\label{sec:proofforc}
In this section we prove 
 Theorem \ref{thm:CH_recover}  as an application of Theorem \ref{thm:main_coding}.  
We begin by recalling some situations from \cite{AlvesPirnapasov} in which a transverse link $\mathcal{L}$ forces topological entropy. It is shown that  certain properties of the cylindrical contact homology in a complement of a transverse link ($\CH_{\mathcal{L}}$) yield the forcing property for the link  $\mathcal{L}$.
The homology  $\CH_{\mathcal{L}}$ was defined and studied by Momin in \cite{Momin}.
We recall some definitions, cf.\   \cite{AlvesPirnapasov, Momin}.
Let $(M,\xi)$ be a contact manifold, $\mathcal{L}$ a transverse link in $(M,\xi)$. 
We say that a loop $\gamma: S^1\to M$ is \textit{contractible in the complement of $\mathcal{L}$} if there is a continuous map $\mathfrak{h}: \overline{\D} \to M$ with $\mathfrak{h}(e^{2\pi it}) = \gamma(t)$ such that there is no  $x\in \D$ for which  $\mathfrak{h}(x)\in \mathcal{L}\subset M$. 
\begin{defn}\label{def:hypertight}
A supporting contact form  $\alpha$ for $(M,\xi)$ is called \textit{hypertight in the complement of $\mathcal{L}$} if 
\begin{itemize}
\item the components of $\mathcal{L}$ are closed Reeb orbits of $\varphi_{\alpha}$, \item 
 each closed Reeb orbit $\gamma$ of $\varphi_{\alpha}$ (not necessarily simple or in $M\setminus \mathcal{L}$) is contractible in the complement of $\mathcal{L}$.
 \end{itemize}
  \end{defn}
From now on, assume  that $\alpha$ is hypertight in the complement of $\mathcal{L}$.
\begin{defn}\label{PLC}
Let $\rho$ be a homotopy class of loops in $M\setminus \mathcal{L}$. 
The triple $(\alpha, \mathcal{L}, \rho)$ is said to satisfy the \textit{proper link class (PLC) condition} if 
\begin{itemize}
\item the loops  $\gamma:S^1 \to M\setminus \mathcal{L}$ with $[\gamma] = \rho$ are non-homotopic in the complement of $\mathcal{L}$ to a multiple $\beta:S^1 \to M$ of a component of $\mathcal{L}$, that is, there is no continuous map  $I:S^1 \times [0,1]$, $I(t,0) = \gamma(t)$, $I(t,1) = \beta(t)$ such that $\im I|_{S^1 \times (0,1)}\cap \mathcal{L} = \emptyset$. 
\item each closed Reeb orbit of $\alpha$ that represents $\rho$ is non-degenerate and simply covered. 
\end{itemize}
\end{defn}
Momin showed that if $(\alpha, \mathcal{L}, \rho)$ satisfies the PLC condition, then one can define the \textit{cylindrical contact homology} of $\alpha$ on the complement of $\mathcal{L}$ for orbits in $\rho$ ($\CH^{\rho}_{\mathcal{L}}(\alpha)$).
It is the homology of a chain complex which is  a $\Z_2$ graded vector space generated by Reeb orbits of $\alpha$ that represent the class $\rho$, with a $\Z_2$ grading given by the pairity of the Conley-Zehnder index of a Reeb orbit $\gamma$ with respect to any trivialization of $\xi$ over $\gamma$. The differential on the chain complex counts certain finite energy pseudoholomorphic cylinders in the symplectization of $\alpha$.  For the details of the construction, see \cite{HryiniewiczMominSalomao} and \cite{Momin}. If $\rho$ contains only one Reeb orbit, then $\CH^{\rho}_{\mathcal{L}}(\alpha)\neq 0$.

Let $\mathcal{H}^*_{\mathcal{L}}(\alpha)$ be the set of homotopy classes $\rho$ of loops in $M\setminus \mathcal{L}$ for which 
\begin{itemize}
    \item $(\alpha, \mathcal{L}, \rho)$ satisfies the PLC condition, 
    \item $\CH^{\rho}_{\mathcal{L}}(\alpha) \neq 0$. 
    \end{itemize}
    Let $N^*_{\alpha}(T):= \#\{\rho \in \mathcal{H}^*_{\mathcal{L}} \, |\, \per(\rho) \leq T\}  .$ 
The \textit{exponential homotopical growth rate of $\CH_{\mathcal{L}}(\alpha)$} is defined as 
$$\Gamma_{\mathcal{L}}(\alpha) :=\limsup_{T\to +\infty} \frac{1}{T} \log N^*_{\mathrm{\alpha_0}}(T).$$
Before we prove Theorem \ref{thm:CH_recover}, let us recall the following result.
\begin{thm}\label{thm:alvesPirna2}\cite[Thm.\ 1.5]{AlvesPirnapasov}
Let $(M,\xi)$ be 
     a closed contact manifold. Let $\alpha_0\in \mathcal{C}(\xi)$ 
 be hypertight in the complement of $\mathcal{L}$. 
Then, for any $\alpha  \in \mathcal{C}(\xi)$ with $\mathcal{L} \subset \Per(\varphi_{\alpha})$, 
\begin{align}\label{eq:htopGamma}
h_{\topo}(\varphi_{\alpha}) \geq \frac{\Gamma_{\mathcal{L}}(\alpha_0)}{\max f_{\alpha}},
\end{align}
where $f_{\alpha} :M \to (0,+\infty)$ is the function such that $\alpha = f_{\alpha}\alpha_0$.
\end{thm}
\begin{proof}[Proof of Theorem \ref{thm:CH_recover}]
Let $\alpha$ be a  contact form, hypertight in 
 the complement of some link $\mathcal{L}_0$  of Reeb orbits of $\varphi_{\alpha}$. Assume that $h_{\topo}(\varphi_{\alpha}) >0$. Let $\varepsilon>0$. Choose a link $\mathcal{L}= \mathcal{L}(\varepsilon)$  as in Theorem \ref{thm:main_coding}. 
Clearly $\alpha$ is still hypertight in the complement of $\mathcal{L}'= \mathcal{L} \cup \mathcal{L}_0$.
The periodic orbits in the horseshoe $K$ are hyperbolic, in particular, non-degenerate. Any orbit  $P_{(i_0,\ldots, i_{n-1})}$ for which $(i_0,\ldots, i_{n-1})$ cannot be written as a repetition of $k$-tuples, with $k|n$ (in particular for which  $i_0,\ldots, i_{n-1}$ are not all identical)   represents a homotopy class that satisfies the PLC condition. (We must additionly exclude   those finitely many orbits  $P_{i_0,\ldots, i_{n-1}}$ that happen to be  components of $\mathcal{L}_0$.)
It follows  from Theorem \ref{thm:main_coding} (1) and (3), that  $\Gamma_{\mathcal{L}'}(\alpha) > h_{\topo}(\varphi_{\alpha}) 
-\varepsilon$. The result then follows from Theorem \ref{thm:alvesPirna2} with $\alpha = \alpha_0$.   
\end{proof}
Finally, we remark that if there exists a global surface of section for the Reeb flow of  a contact form $\alpha$,  then there is also a link $\mathcal{L}_0$   such that the contact form is hypertight in the complement of $\mathcal{L}_0$.  And hence it follows from the work \cite{Colin_broken} on the existence of  broken book decompositions, and the recent results of \cite{Irie_equi}  and \cite{Hryniewicz_generic}, or also   \cite{Contreras_generic}, that the condition to be a contact form that is hypertight in the complement of some link of Reeb orbits holds on a $C^{\infty}$-open and dense set. 
A global surface of section is a  compact surface  
$S$ (with or without boundary)  and an immersion $\iota: S \to M$, such that 
\begin{itemize}
\item $i(\partial S)$ is a (possibly empty) link $\mathcal{L}_b$ whose components are periodic orbits of $\varphi_{\alpha}$; \item $\iota^{-1}(\iota(\partial S)) = \partial S$, and $\iota$ defines an embedding $S\setminus \partial S \to M\setminus \iota (\partial S)$; 
\item for any $p \in M$, there $s<0<t$ such that $\varphi_{\alpha}^s(p), \varphi_{\alpha}^t(p) \in \iota(S)$.
\end{itemize}
The following link $\mathcal{L}_0$ satisfies the claimed property. If $\mathcal{L}_b = \emptyset$, then $\varphi_\alpha$ has no contractible orbits at all, hence we can choose  $\mathcal{L}_0  = \emptyset$;  if  $S$ is a disk, we can choose $\mathcal{L}_0$ to be the union of $\mathcal{L}_b$ and one additional closed orbit;  and if $S$ is not a disk, we can put   
 $\mathcal{L}_0 = \mathcal{L}_b$, see e.g.\ \cite[Lemma 6.6]{Momin}.


\bibliographystyle{plain}
\bibliography{biblio}

\end{document}